\documentclass{article}
\usepackage{import}
\usepackage{AGUtils}

\newcounter{subtheorem}
\counterwithin{theoremgrp}{subtheorem}

\numberwithin{equation}{section}

\title{Analysis of Log-Weighted Quadrature Domains}

\author{Andrew J. Graven}

\date{\vspace{-4ex}}

\begin{document}

\maketitle

\begin{abstract}
This paper studies plane domains satisfying a quadrature identity with respect to the singular weight $\rho_0(w)=|w|^{-2}$. These are referred to as \emph{log-weighted quadrature domains} (LQDs). The logarithmic singularity at $w=0$ leads to phenomena not present in the classical theory: in particular, when the domain contains the origin, the associated quadrature data are no longer unique, but are determined only up to a point charge at $0$. A generalized Schwarz function characterization of LQDs is established together with a natural formulation of the inverse problem in the singular setting. In the simply connected case, it is shown that a domain is an LQD if and only if the outer factor of its Riemann map extends to the exponential of a rational function. This characterization yields explicit formulae relating the quadrature function and the Riemann map via the Faber transform, thereby extending earlier formulae from the non-singular theory. Several basic classes of LQDs are also covered, and explicit examples are computed.
\end{abstract}

\section{Introduction}\label{sec:Introduction}
The purpose of this article is to analyze plane domains admitting a quadrature identity with respect to the weight $\rho_0(w):=|w|^{-2}$. In particular, we are concerned with domains $\Omega\subset\Ch$ for which
\begin{equation}
\int_{\Omega}\dfrac{f(w)}{|w|^2}dA(w)=\oint_{\partial\Omega}f(w)h(w)dw
\end{equation}
for all $f\in L_a^1(\Omega;\rho_0)$, for some rational function $h$. We refer to such domains as \emph{log-weighted quadrature domains} (LQDs). LQDs arise naturally as the image of classical quadrature domains under the exponential map (Theorem \ref{thm:LQDsFromQDsExponential}). LQDs also appear as the complement of local droplets of the logarithmic potential
$$Q(w)=\frac{1}{2}\ln^2|w|^2-2\Re(H(w)),$$
where $H'$ is rational (\cite{GravenMakarov2025}, \S4.1.4). We also note (but do not show here) that LQDs are preserved under Hele-Shaw flow weighted with respect to $\rho_0(w)=|w|^{-2}$. Weighted Hele-Shaw flow has received some attention in the literature, particularly in the context of Hele-Shaw flow on Riemann surfaces and the Polubarinova-Galin equation \cite{hedenmalm2004HypHS,GustafssonLinPGRatFcns}.

This work is a generalization of the theory developed in Graven \& Makarov \cite{GravenMakarov2025} to include \emph{singular} LQDs - domains containing the metric singularity at $w=0$. The presence of the singularity leads to a basic phenomenon absent from the classical and non-singular theory: when $0\in\Omega$, every admissible test function vanishes at the origin, and as a result the quadrature function is no longer uniquely determined by the domain, but only up to the addition of a point charge $q/w$ at $0$. Thus the singular theory therefore requires new formulations of the coincidence equation, of the Schwarz function, and of the direct/inverse problems. The main contribution of the present paper is to show that, despite this loss of uniqueness, much of the classical and non-singular structure survives in modified form.

Our first main result is a Schwarz function characterization of singular LQDs. We show that a domain is an LQD if and only if it admits a generalized Schwarz function $S_0$ with boundary values
$$S_0(w)\dEquals\dfrac{\ln|w|^2}{w}.$$
This yields, in particular, a Sakai-type boundary regularity theorem: the boundary of an LQD has only finitely many singular points, each of which is a cusp or a double point. We also establish several basic invariance properties and an electrostatic interpretation of LQDs.

In the simply connected setting, we obtain a sharper description in terms of the Riemann map. Whereas classical QDs are precisely those domains admitting a rational Riemann map, in the log-weighted setting the correct object is the outer factor of the Riemann map. In particular, we show that a simply connected domain is an LQD if and only if the outer part of its Riemann map extends to the exponential of a rational function. This leads to explicit Faber transform formulae relating the quadrature function and the Riemann map. In later sections, we apply these results to null, monomial, and one-point families of LQDs, with complete classifications in some cases and partial classifications in others.

After fixing some conventions and general notation, we will begin with a discussion of classical quadrature domains and several structural results which will later be adapted to the log-weighted setting.\\

{\it General notation:} Throughout, $\Ch$ denotes the Riemann sphere. A domain is a connected open subset of $\Ch$. For a set $\Omega\subseteq\Ch$, we denote by $\Cl(\Omega)$ its closure, $\Omega^c$ its complement, $\partial\Omega$ its boundary, $\Int(\Omega)$ its interior, and
$$\Omega\IntComp:=\Ch\setminus\Cl(\Omega)$$
the exterior of $\Omega$. We call a domain rectifiable if $\partial\Omega$ is rectifiable. When $f$ and $g$ are functions defined on $\partial\Omega$, we write $f\dEquals g$ to denote equality on $\partial\Omega$. $f^{\#}(z):=\overline{f(\overline{z}^{-1})}$ denotes the reflection of $f$ about the unit circle.\\
$\ln(w)$ denotes the principal branch of the natural logarithm (branch cut along the negative real axis), unless otherwise stated.\\
Contour integrals are normalized such that the standard factor of $2\pi i$ is suppressed. $dA=\frac{dxdy}{\pi}$ denotes the normalized area measure. If $\Omega$ is an unbounded domain then integrals are understood in terms of their Cauchy principal value,
\begin{align*}
    \int_{\Omega}fdA&:=\lim_{R\to\infty}\int_{\Omega\cap\D_R}fdA.
\end{align*}

{\it Function spaces:} If $\Omega$ is a domain, we denote by $\M(\Omega)$ the space of meromorphic functions on $\Omega$, $H(\Omega)$ the space of analytic functions on $\Omega$, and $H^\infty(\Omega)$ the Hardy space of bounded analytic functions on $\Omega$. We define
$$A(\Omega):=H(\Omega)\cap C^0(\Cl(\Omega)),$$
the space of functions analytic in $\Omega$ which extend continuously to the boundary.\\
If $\Omega$ is unbounded, we define
$$\A_0(\Omega):=\left\{f\in\A(\Omega):f(\infty)=0\right\},$$
the $\A(\Omega)$-subspace of functions which vanish at $\infty$.\\
For a non-negative weight $\rho$, we denote by $L_a^1(\Omega;\rho)$ the subspace of $\A(\Omega)$ consisting of functions $f$ which are $L^1-$bounded with respect to the weight $\rho$:
$$\int_{\Omega}|f|\rho dA<\infty.$$
We denote by $\Rat(\Omega)$ the set of rational functions whose poles lie only in $\Omega$, and $\Rat_0(\Omega)$ the subclass vanishing at $\infty$.

\subsection{Quadrature Domains}\label{subsec:QDs}

Given a rectifiable domain $\Omega$ and an integrable test class $\mathcal{F}=\mathcal{F}(\Omega)$ of functions analytic in $\Omega$ and continuous on $\Cl(\Omega)$, we say that $\Omega$ is a \emph{quadrature domain} with respect to $\mathcal{F}$ if it satisfies a Green's theorem-like identity of the form
\begin{equation}\label{eqn:ClassicContourQID}
\int_{\Omega}fdA=\oint_{\partial\Omega}f(w)h(w)dw,\;\;\;\;f\in\mathcal{F}(\Omega),
\end{equation}
where $h$ is a rational function with poles only in $\Omega$, referred to as a \emph{quadrature function}. The poles of $h$ are referred to as \emph{quadrature nodes}. The most well-known example of such a domain is the open disk $\Omega=\D_r(w_0)$. In particular, if $f\in\A(\D_r(w_0))$ then by Green's theorem and the fact that $|w-w_0|^2\dEquals r^2$, we find
\begin{align*}
\int_{\D_r(w_0)}fdA(w)&=\oint_{\partial\D_r(w_0)}f(w)\overline{w}dw=\oint_{\partial\D_r(w_0)}f(w)\left(\dfrac{r^2}{w-w_0}+\overline{w_0}\right)dw=\oint_{\partial\D_r(w_0)}f(w)\dfrac{r^2}{w-w_0}dw
\end{align*}

Hence, $\Omega=\D_r(w_0)$ satisfies an identity in the form of Equation \ref{eqn:ClassicContourQID}, with $h(w)=\frac{r^2}{w-w_0}$. Aharonov \& Shapiro \cite{AharonovShapiro} demonstrated the converse statement: disks are the only finitely connected bounded domains which satisfy the mean value property.

Note that if $\Omega$ is bounded, then Equation \ref{eqn:ClassicContourQID} is equivalent to the existence of a finite collection of $a_k\in\Omega$, $c_k\in\C$, and $m_k\in\N$ such that
\begin{equation}\label{eqn:ClassicPtQID}
\int_{\Omega}fdA=\sum_{k}c_kf^{(m_k)}(a_k),\;\;\;\;f\in\mathcal{F}(\Omega).
\end{equation}
This is commonly referred to as a \emph{quadrature identity}. These two formulations are related via the residue theorem, whence we obtain the formula $h(w)=\sum_{k}\frac{c_km_k!}{(w-a_k)^{m_k+1}}$. Returning to the prior example of $\Omega=\D_r(w_0)$, we recover the classical mean value property for the disk,
\begin{align*}
    \int_{\D_r(w_0)}fdA&=r^2f(w_0).
\end{align*}

More formally, we define a quadrature domain in the bounded case along the lines of \cite{GravenMakarov2025} as
\begin{definition}[Bounded quadrature domain]
A bounded rectifiable domain $\Omega\subset\C$ is a quadrature domain if there exists $h\in\Rat_0(\Omega)$ such that Equation \ref{eqn:ClassicContourQID} holds for all $f\in\A(\Omega)$.
\end{definition}
However, if $\Omega$ is an unbounded domain, then $\A(\Omega)$ is not an integrable test class (even in principal value). Hence, a slightly smaller test class is necessary.
\begin{definition}[Unbounded quadrature domain]
An unbounded rectifiable domain $\Omega\subset\Ch$ is a quadrature domain if there exists $h\in\Rat(\Omega)$ such that Equation \ref{eqn:ClassicContourQID} holds for all $f\in\A_0(\Omega)$.
\end{definition}
By $\Omega\in\QD(h)$, we denote that $\Omega$ is a bounded or unbounded QD with quadrature function $h$. We furthermore write $\Omega\in\QD$ if there exists $h$ such that $\Omega\in\QD(h)$. \emph{We will always require that $\Omega=\Int(\Cl(\Omega))$} because otherwise arbitrarily many QDs with the same quadrature function could be constructed via deletion of subsets of measure zero. The following discussion (see e.g. \cite{GravenMakarov2025,Lee_2015}) summarizes several important properties/characterizations of quadrature domains.
\begin{lemma}\label{lemma:QDChars}
Let $\Omega$ be a domain. Then the following are equivalent
\begin{enumerate}
    \item $\Omega\in\QD$,
    \item $C^\Omega\vert_{\Omega\IntComp}$ extends to a rational function on $\Ch$,
    \item $\Omega$ admits a Schwarz function.
\end{enumerate}
\end{lemma}
Where a \emph{Schwarz function} $S:\Cl(\Omega)\rightarrow\Ch$ is a meromorphic function continuous up to the boundary such that $S(w)\dEquals\overline{w}$,\footnote{Recall that $\dEquals$ denotes equality on $\partial\Omega$.} and $C^\Omega$ denotes the \emph{Cauchy transform}
\begin{align*}
    C^{\Omega}(w)&=\int_{\Omega}\dfrac{dA(\xi)}{w-\xi}.
\end{align*}
Furthermore if $\Omega\in\QD(h)$, then its Schwarz function satisfies the identity \cite{davis_1974}
\begin{equation}\label{eqn:QuadCoincidence}
S(w)=h(w)+C^{\Omega\IntComp}(w).
\end{equation}

The boundary regularity of quadrature domains depends crucially on the existence of a Schwarz function; in particular, a \emph{local Schwarz function}. This follows from the celebrated Sakai regularity theorem \ref{thm:SakaiRegularity}. We say that $S$ is a local Schwarz function for $\Omega$ at $w_0\in\partial\Omega$ if there exists $\epsilon>0$ such that $S\in\A(\D_\epsilon(w_0)\cap\Cl(\Omega))$ and $S(w)=\overline{w}$ on $\D_\epsilon(w_0)\cap\partial\Omega$.
\begin{theorem}[Sakai Regularity \cite{SakaiRegularity}]\label{thm:SakaiRegularity}
    If $S$ is a {\it local Schwarz function} at a singular point $w_0$ of $\partial\Omega$, then $w_0$ is either a (conformal) {\it cusp}, a {\it double point}, or a {\it degenerate point}.
\end{theorem}
See \cite{Lee_2015}, \S3.2 for a more formal discussion of Sakai regularity and the relevant terms (cusps, double point, local Schwarz function, etc.). An immediate consequence of Sakai's theorem is the following regularity theorem for QDs
\begin{corollary}\label{cor:QDBoundaryRegularity}
    If $\Omega\in\QD$ then $\partial\Omega$ has finitely many singular points, and each is a cusp or a double point.
\end{corollary}
An important consequence of Corollary \ref{cor:QDBoundaryRegularity} is that the boundary of a quadrature domain is piecewise $C^1$. This level of boundary regularity implies a number of nice properties for function spaces on $\Omega$ (e.g. Lemma \ref{lemma:AnalyticDecompositionLemmaV2}).

A central problem in the theory of quadrature domains is uniqueness. On the one hand, uniqueness of the quadrature function associated to a given quadrature
domain is well established. However, the inverse problem of existence and uniqueness of quadrature domains
associated to a given quadrature function remains open. Non-uniqueness occurs in the multiply connected setting: for instance, a disk and the complement of an annulus with the same inner radius give distinct examples with the same quadrature data; see also Varchenko and Etingof (\cite{DropOrderFour}, \S2.3). In contrast, it is conjectured that simply connected quadrature domains are uniquely determined by their quadrature function, up to conformal radius. Aharonov and Shapiro (1976) showed that disks are the only one-point quadrature domains among bounded finitely connected domains \cite{AharonovShapiro}, and proved a similar result for the cardioid. Later on, we will establish the uniqueness of null LQDs (\S\ref{subsec:NullLQDs}), and obtain partial uniqueness results for several other families.

\subsection{Simply Connected Quadrature Domains}
A core question in the study of quadrature domains is the {\it inverse problem} of recovering a domain from its quadrature function. The {\it direct problem} conversely is concerned with recovering the quadrature function associated to a given domain. In the simply connected setting, both questions reduce to a problem of relating the quadrature function to the Riemann map. The following classical theorem characterizes simply connected quadrature domains precisely via their Riemann maps.
\begin{theorem}[\cite{AharonovShapiro}]\label{thm:QDIffRationalRiemann}
    A simply connected domain is a quadrature domain if and only if its Riemann map extends to a rational function.
\end{theorem}

Hence, Theorem \ref{thm:QDIffRationalRiemann} reduces the inverse and direct problems for simply connected QDs to a problem of relating the finitely many coefficients of the quadrature function to the finitely many coefficients of the Riemann map. Theorems 1.5 and 1.6 in \cite{GravenMakarov2025} exactly characterize this relationship via the Faber transform $\Phi_{\varphi}$. The result in the bounded case states:

\begin{theorem}[\cite{GravenMakarov2025}]\label{thm:ClassicBQDRationalRiemannIffQD}
Let $\Omega$ be a bounded and simply connected QD with quadrature function $h\in\Rat_0(\Omega)$ and Riemann map $\varphi:\D\rightarrow\Omega$. In this case,
\begin{equation}\label{eqn:BoundedFaberTransformHFormula}
h=\Phi_{\varphi}\left(\varphi^{\#}-\overline{\varphi(0)}\right)
\end{equation}
and
\begin{equation}\label{eqn:BoundedFaberTransformPhiFormula}
\varphi=\varphi(0)+\Phi_{\varphi}^{-1}(h)^{\#}.
\end{equation}
\end{theorem}
In particular, the Faber transform provides an implicit solution to the inverse problem, and an explicit solution to the direct problem. The result in the unbounded case (\cite{GravenMakarov2025}, Theorem 1.6) is analogous. The details of the Faber transform and its properties are covered in the Appendix \S\ref{subsec:FaberTransform}.\\

In brief, the \emph{interior Faber transform} $\Phi_\varphi$ associated to a simply connected bounded domain $\Omega$ is a bijective linear operator from functions analytic in the exterior disk $\A_0(\D\IntComp)$ to functions analytic in the exterior of $\Omega$, $\A_0(\Omega\IntComp)$. Similarly, the \emph{exterior Faber transform} is a bijective linear operator $\Phi_{\varphi}:\A(\D)\rightarrow\A(\Omega\IntComp)$.

In particular, if $\varphi$ is a Riemann map associated to $\Omega$ and $\psi=\varphi^{-1}$, then we compute the Faber transform of $f$ by precomposing $f$ with $\psi$, then projecting onto the part of $f\circ\psi$ which is analytic in $\Omega\IntComp$: 
$$\Phi_{\varphi}(f):=\AnalyticIn{f\circ\psi}{\Omega\IntComp},$$
where $\AnalyticIn{\cdot}{\Omega}$ is the Cauchy projection, given by
$$\AnalyticIn{f}{\Omega}(w):=\oint_{\partial\Omega}\dfrac{f(\xi)}{\xi-w}d\xi.$$
The most important three properties of the Faber transform $\Phi_{\varphi}$ for our purposes are as follows
\begin{itemize}
    \item $\Phi_{\varphi}$ is an isomorphism;
    \item $\Phi_{\varphi}$ takes rational functions to rational functions;
    \item $\Phi_{\varphi}$ is \emph{computable}: the Faber transform of a rational function can be computed in finitely many steps.
\end{itemize}
These three properties are crucial for establishing a useful relationship between the quadrature function and Riemann map associated to a given QD or LQD. See Appendix \S\ref{subsec:CauchyTransformFaberTransform} for a more detailed exposition of the Cauchy projection and Faber transform, as well as their various properties (see also \cite{GravenMakarov2025}, \S1.5.1).

\subsection{Outline}

Section \ref{sec:TheoryOfLQDs} develops the structural theory of LQDs: the coincidence equation, the generalized Schwarz-function, essential uniqueness of the quadrature function, boundary regularity, and several invariance properties. Section \ref{sec:LQDInvDirectProbs} treats the first model inverse problems, namely null and certain one-point LQDs. Section \ref{sec:SimplyConnectedLQDs} gives the main simply connected characterization of LQDs in terms of the inner/outer factorization of the Riemann map and derives the corresponding Faber transform formulae. Sections \ref{sec:MonomialLQDs} and \ref{sec:OnePointLQDs} then apply this general framework to concrete families of monomial and one-point LQDs.

\section{Theory of Log-Weighted Quadrature Domains}\label{sec:TheoryOfLQDs}
In this section we develop the basic theory of \emph{log-weighted quadrature domains} (LQDs). We begin by clearly defining LQDs and highlighting a method of constructing LQDs from classical QDs via the exponential map (Theorem \ref{thm:LQDsFromQDsExponential}). We then derive the singular and non-singular coincidence equations, use them to establish the essential uniqueness of the quadrature data, and prove a generalized Schwarz-function characterization of LQDs. From this characterization we obtain boundary regularity, an electrostatic characterization, and several basic invariance properties that will be used throughout the remainder of the paper.

We consider domains $\Omega\subset\Ch$ such that
\begin{equation}\label{eqn:LQDQID}
\int_{\Omega}\dfrac{f(w)}{|w|^2}dA(w)=\oint_{\partial\Omega}f(w)h(w)dw,
\end{equation}
for all $f\in L_a^1(\Omega;\rho_0)$, with $h$ rational. Note that when $0\in\Omega$, admissible test functions vanish at $0$; when $\Omega$ is unbounded, admissible test functions vanish at infinity. To ensure essential uniqueness of the quadrature function $h$, we place the additional constraint that $h\in\Rat_0(\Omega)$ when $\Omega$ is bounded, and $h\in\Rat(\Omega)$ when $\Omega$ is unbounded.

\begin{definition}[Log-Weighted Quadrature Domain]
Let $\Omega\subset\Ch$ be a rectifiable domain for which $0,\infty\notin\partial\Omega$. We say that $\Omega$ is a \emph{log-weighted quadrature domain} if there exists a rational $h$ (satisfying the above conditions) such that Equation \ref{eqn:LQDQID} is satisfied.
\end{definition}
Throughout the remainder of the paper we will denote by $\Omega\in\QD_0(h)$ that $\Omega$ is an LQD with quadrature function $h$, and by $\Omega\in\QD_0$ we denote that there exists an $h$ for which $\Omega\in\QD_0(h)$. As with classical QDs, we will always require that $\Omega=\Int(\Cl(\Omega))$ because otherwise arbitrarily many QDs with the same quadrature function could be constructed via deletion of subsets of measure zero. We can generate some easy first examples of LQDs by considering the images of certain classical QDs under the exponential map. In particular, 
\begin{theorem}\label{thm:LQDsFromQDsExponential}
If $\Omega\in\QD(h)$ is a bounded domain on which $w\mapsto e^w$ is injective, then $e^\Omega\in\QD_0(\widetilde{h})$, where $\widetilde{h}(w)=\AnalyticIn{\frac{h\circ\ln(w)}{w}}{e^{\Omega}\IntComp}$.
\end{theorem}
Theorem \ref{thm:LQDsFromQDsExponential} also tells us that if $\{p_k\}$ are the poles of $h$, then $\{e^{p_k}\}$ are the poles of the quadrature function for $e^{\Omega}$. It follows from the residue theorem that these poles also have the same multiplicity.
\begin{proof}[Proof of Theorem \ref{thm:LQDsFromQDsExponential}]
Under the above hypotheses, $f\in L_a^1(e^{\Omega};\rho_0)$ $\iff$ $f(e^w)\in \A(\Omega)$, so
\begin{align*}
    \int_{e^{\Omega}}\dfrac{f(w)}{|w|^2}dA(w)&=\int_{\Omega}f(e^w)dA(w)=\oint_{\partial\Omega}f(e^w)h(w)dw=\oint_{\partial e^{\Omega}}f(w)\dfrac{h\circ\ln(w)}{w}dw.
\end{align*}
By Theorem \ref{thm:SakaiRegularity}, $\partial\Omega$ is piecewise $C^1$, which implies the same for $\partial e^{\Omega}$. Hence, by Lemma \ref{lemma:AnalyticDecompositionLemmaV2}(1),
$$\frac{h\circ\ln(w)}{w}\dEquals\AnalyticIn{\frac{h\circ\ln(w)}{w}}{e^\Omega}+\AnalyticIn{\frac{h\circ\ln(w)}{w}}{e^\Omega\IntComp},$$
where the first term is analytic in $e^{\Omega}$. Moreover, as $0$ is in the unbounded component of $(e^{\Omega})^c$, $\frac{h\circ\ln(w)}{w}\in\M(e^{\Omega})$ and extends continuously to $\partial e^{\Omega}$, so $\widetilde{h}(w):=\AnalyticIn{\frac{h\circ\ln(w)}{w}}{e^{\Omega}\IntComp}\in\Rat_0(e^{\Omega})$ by Lemma \ref{lemma:AnalyticDecompositionLemmaV2}. The $\AnalyticInNoBracket{}{e^\Omega}$ term drops out because it is analytic in $e^\Omega$, and we obtain
\begin{align*}
    \int_{e^{\Omega}}\dfrac{f(w)}{|w|^2}dA(w)&=\oint_{\partial e^{\Omega}}f(w)\AnalyticIn{\frac{h\circ\ln(w)}{w}}{e^\Omega}dw+\oint_{\partial e^{\Omega}}f(w)\widetilde{h}(w)dw=\oint_{\partial e^{\Omega}}f(w)\widetilde{h}(w)dw,
\end{align*}
and we may conclude that $e^{\Omega}\in\QD_0(\widetilde{h})$.
\end{proof}

For example, consider $\D_{2}(-1)\in\QD\left(\frac{4}{w+1}\right)$. As the exponential map is injective on $\D_{2}(-1)$, we find that $e^{\D_{2}(-1)}$ is an LQD, with quadrature function $\widetilde{h}$ having a simple pole at $e^{-1}$. Note that $\frac{1}{\ln(w)+1}=\frac{1}{ew-1}+G(w)$ for some $G\in\A(e^{\D_{2}(-1)})$. Hence, by Lemma \ref{lemma:AnalyticDecompositionLemmaV2}(3,4),
\begin{align*}
    \widetilde{h}(w)&=\AnalyticIn{\dfrac{4}{w}\dfrac{1}{\ln(w)+1}}{e^{\D_{2}(-1)}\IntComp}=\AnalyticIn{\dfrac{4}{w}\left(\dfrac{1}{ew-1}+G(w)\right)}{e^{\D_{2}(-1)}\IntComp}=\AnalyticIn{\dfrac{4}{w-e^{-1}}-\dfrac{4}{w}}{e^{\D_{2}(-1)}\IntComp}=\dfrac{4}{w-e^{-1}}.
\end{align*}
That is, $e^{\D_2(-1)}\in\QD_0\left(\frac{4}{w-e^{-1}}\right)$. Figure \ref{fig:LQDsFromQDsExponential} shows $\D_2(-1)$ and its image under the exponential map.

\begin{figure}[ht]
  \centering
\includegraphics[width=0.3\linewidth,height=0.3\linewidth]{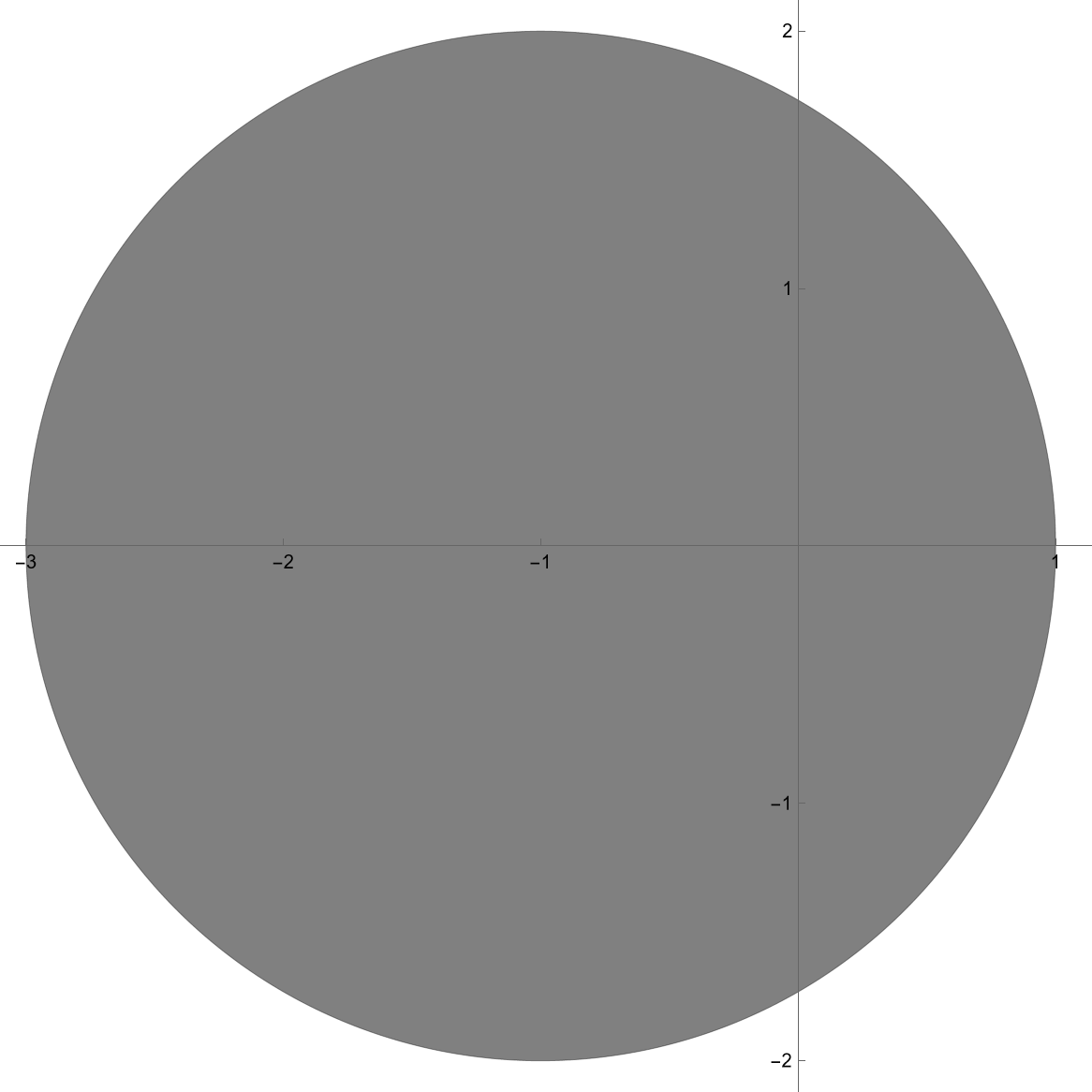}\tab\tab\tab\tab\tab\tab\tab\includegraphics[width=0.3\linewidth,height=.3\linewidth]{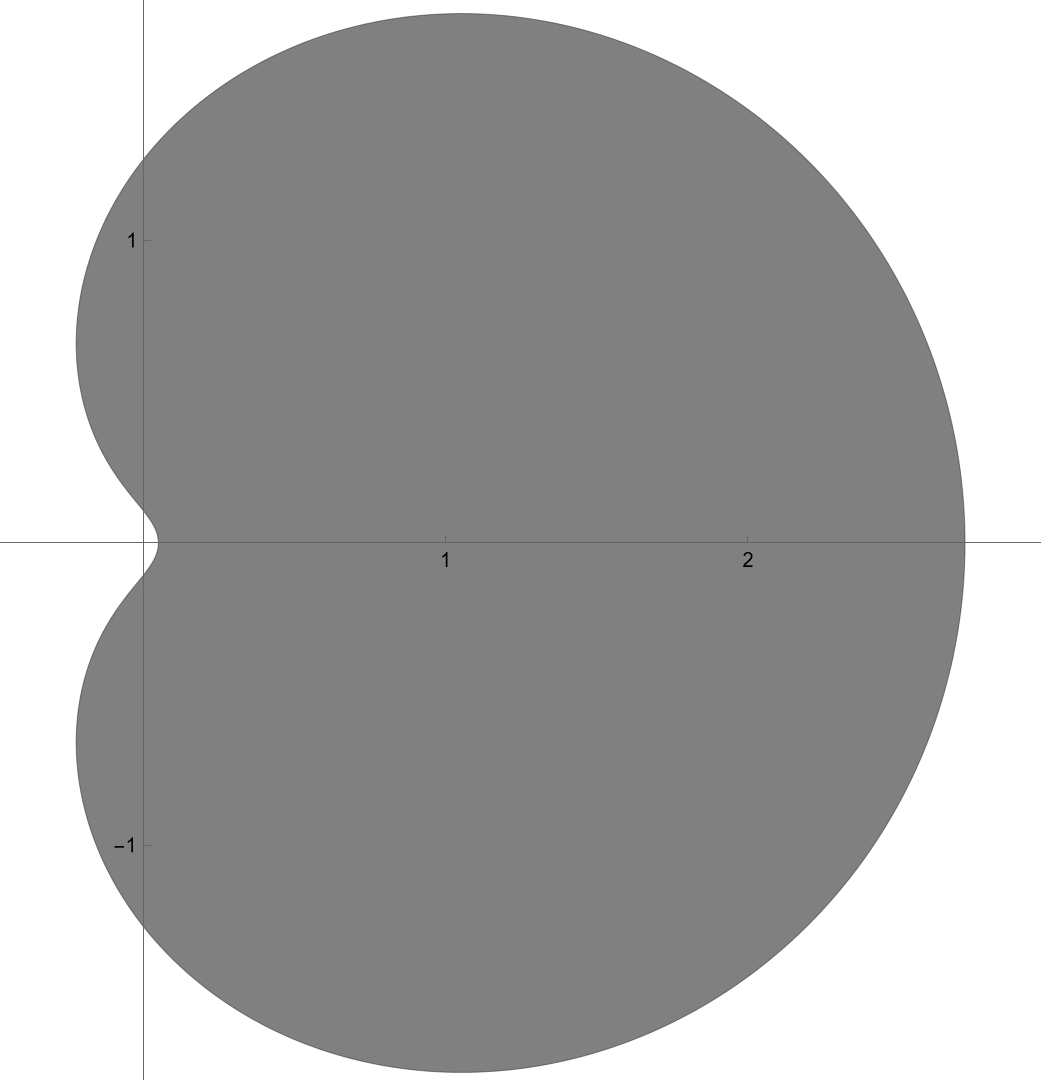}
\vspace{.5em}
\caption{$\D_{2}(-1)\in\QD\left(\frac{4}{w+1}\right)$ and its image under the exponential map, $e^{\D_2(-1)}\in\QD_0\left(\frac{4}{w-e^{-1}}\right)$.}\label{fig:LQDsFromQDsExponential}\vspace{1.5em}
\end{figure}

\subsection{The Generalized Coincidence Equation}
The primary distinction of LQDs from classical QDs is the existence of a metric singularity at $w=0$. We refer to LQDs containing zero as \emph{singular}. We likewise refer to LQDs not containing $0$ as \emph{non-singular}. The effect of this singularity is captured by the additional $\frac{q}{w}$ term in the \emph{coincidence equation} for singular LQDs (Equation \ref{eqn:LQDCEZero}). Graven \& Makarov \cite{GravenMakarov2025} obtained the following generalization of the coincidence equation (Equation \ref{eqn:QuadCoincidence}) for non-singular LQDs.
\begin{theorem}[\cite{GravenMakarov2025}, \S4]\label{thm:LQDCENoZero}
    If $\Omega\in\QD_0(h)$ is a domain not containing zero, then there exists a unique $G\in\A(\Omega)$ such that
    \begin{equation}\label{eqn:LQDCENoZero}
        \dfrac{\ln|w|^2}{w}\dEquals h(w)+G(w).
    \end{equation}
    Moreover, $G(w)=C_{\rho_0}^{\Omega\IntComp\setminus\D_r}(w)+\frac{\ln|r|^2}{w}$ for any $r\in(0,d(0,\partial\Omega))$.
\end{theorem}

$C_{\rho}^{U}$ denotes the $\rho-$weighted {\it Cauchy transform} of $U\subseteq\Ch$, which is given by
\begin{equation*}
C_\rho^{U}(w)=\int_{U}\dfrac{\rho(\xi)}{w-\xi}dA(\xi).
\end{equation*}

Obtaining a generalized coincidence equation for singular LQDs is complicated by the fact that the Cauchy kernel $k_w(\xi)=\frac{1}{\xi-w}$ is no longer an element of the test class $L_a^1(\Omega;\rho_0)$. Hence, the quadrature identity cannot be directly applied to $k_w$ to obtain the quadrature function in terms of the weighted Cauchy transform. Instead, the modified Cauchy kernel $\frac{\xi}{\xi-w}$, vanishing at $0$ is employed. We obtain the following generalization of Theorem \ref{thm:LQDCENoZero} to singular LQDs.
\begin{theorem}\label{thm:LQDCEZero}
    If $\Omega\in\QD_0(h)$ is a domain containing zero, then there exist unique $q\in\C$ and $G\in\A(\Omega)$ such that
    \begin{equation}\label{eqn:LQDCEZero}
        \dfrac{\ln|w|^2}{w}\dEquals h(w)+\dfrac{q}{w}+G(w),
    \end{equation}
    where $G=C_{\rho_0}^{\Omega\IntComp}$.
\end{theorem}

In order to proceed with the proof of Theorems \ref{thm:LQDCENoZero} and \ref{thm:LQDCEZero}, we will need the following lemma which allows us to appropriately ``renormalize'' the weighted Cauchy transform:
\begin{lemma}\label{lemma:LOGRenormCauchyTrans}
    For all $|w|>r>0$,
    \begin{equation}\label{eqn:LOGRenormCauchyTrans}
    C_{\rho_0}^{\C\setminus\D_r}(w)=\dfrac{\ln|w|^2-\ln|r|^2}{w}
\end{equation}
\end{lemma}
To see this, we express the left side in Cauchy principal value and apply Green's theorem
\begin{align*}
    C_{\rho_0}^{\C\setminus\D_r}(w)&=\lim_{R\to\infty}\int_{r<|\xi|<R}\dfrac{|\xi|^{-2}}{w-\xi}dA(\xi)\\
    &=\dfrac{\ln|w|^2}{w}+\lim_{R\to\infty}\oint_{\partial\D_R}\dfrac{\ln|\xi|^2}{\xi(w-\xi)}d\xi-\oint_{\partial\D_r}\dfrac{\ln|\xi|^2}{\xi(w-\xi)}d\xi\\
    &=\dfrac{\ln|w|^2}{w}+\lim_{R\to\infty}\ln|R|^2\oint_{\partial\D_R}\dfrac{d\xi}{\xi(w-\xi)}d\xi-\ln|r|^2\oint_{\partial\D_r}\dfrac{d\xi}{\xi(w-\xi)}.
\end{align*}
Applying the residue theorem to the latter two terms, we obtain $0$ and $-\frac{\ln|r|^2}{w}$ respectively. The desired formula follows.

\begin{proof}[Proof of Theorem \ref{thm:LQDCENoZero}]\label{proof:LQDCENoZero}
Fix $0\notin\Omega\in\QD_0(h)$ and $r<d(0,\partial\Omega)$. Note that when $w\in\Omega\IntComp$, the quadrature identity implies
\begin{align*}
    C_{\rho_0}^{\Omega}(w)&=\int_{\Omega}\dfrac{|\xi|^{-2}}{w-\xi}dA(\xi)=\oint_{\partial\Omega}\dfrac{h(\xi)}{w-\xi}d\xi=h(w).
\end{align*}
On the other hand if $w\in\Omega$, then $C_{\rho_0}^{\Omega\IntComp\setminus\D_r}$ is analytic and extends continuously to the boundary (\cite{BellCauchyTransform}, Chapter 17). Hence, for each $w\in\partial\Omega$,
\begin{align*}
    C_{\rho_0}^{\C\setminus\D_r}(w)=C_{\rho_0}^{\Omega\IntComp\setminus\D_r}(w)+C_{\rho_0}^{\Omega}(w)=h(w)+C_{\rho_0}^{\Omega\IntComp\setminus\D_r}(w).
\end{align*}
Hence, by Lemma \ref{lemma:LOGRenormCauchyTrans},
\begin{align*}
    \dfrac{\ln|w|^2}{w}\dEquals h(w)+\left(C_{\rho_0}^{\Omega\IntComp\setminus\D_r}(w)+\dfrac{\ln|r|^2}{w}\right).
\end{align*}
Equation \ref{eqn:LQDCENoZero} follows by setting $G$ equal to the rightmost term. For uniqueness of $G$: Suppose there existed $G_1$ and $G_2$ such that Equation \ref{eqn:LQDCENoZero} holds. Subtracting these equations from each other, we immediately find that $G_1\dEquals G_2$. As $G_1$ and $G_2$ extend continuously to the boundary, it follows from the maximum modulus principle that $G_1=G_2$ on $\Omega$.
\end{proof}

The proof Theorem \ref{thm:LQDCEZero} has two steps. In the bounded case, the modified kernel $\frac{\xi}{\xi-w}$ directly yields the required decomposition. In the unbounded case, one instead uses a two-point kernel to compensate for the behavior at infinity, and the resulting constant term is again absorbed into the charge $q$.

\begin{proof}[Proof of Theorem \ref{thm:LQDCEZero}]\label{proof:LQDCEZero}
Fix $0\in\Omega\in\QD_0(h)$ and $r<d(0,\partial\Omega)$.
We first consider the case in which $\Omega$ is bounded. If $w\in\Omega\IntComp$ then, applying the quadrature identity to $\left(\xi\mapsto\frac{\xi}{w-\xi}\right)\in L_a^1(\Omega;\rho_0)$ yields
\begin{align*}
    \int_{\Omega}\dfrac{\xi|\xi|^{-2}}{\xi-w}dA(\xi)&=\oint_{\partial\Omega}\dfrac{\xi h(\xi)}{\xi-w}d\xi=-wh(w).
\end{align*}
On the other hand, if $w\in\Omega$ then, by Green's theorem ($\Omega\in\QD_0$ by assumption, so its boundary is rectifiable), 
\begin{align*}
    \int_{\Omega}\dfrac{\xi|\xi|^{-2}}{\xi-w}dA(\xi)&=-\ln|w|^2+\oint_{\partial\Omega}\dfrac{\ln|\xi|^2}{\xi-w}d\xi=-\ln|w|^2+w\oint_{\partial\Omega}\dfrac{\ln|\xi|^2}{\xi(\xi-w)}d\xi+\oint_{\partial\Omega}\dfrac{\ln|\xi|^2}{\xi}d\xi.
\end{align*}
Applying Green's theorem to the first integral, and splitting the second to isolate the pole at $0$, we obtain
\begin{align*}
    \int_{\Omega}\dfrac{\xi|\xi|^{-2}}{\xi-w}dA(\xi)&=-\ln|w|^2-w\int_{\Omega\IntComp}\dfrac{|\xi|^{-2}}{\xi-w}d\xi+\left(\oint_{\partial(\Omega\setminus\D_r)}\dfrac{\ln|\xi|^2}{\xi}d\xi+\oint_{\partial\D_r}\dfrac{\ln|\xi|^2}{\xi}d\xi\right)
\end{align*}
We recognize the first integral as a weighted Cauchy transform, and note that the second two integrals integrate to a constant depending only on $\Omega$:
\begin{align*}
    \int_{\Omega}\dfrac{\xi|\xi|^{-2}}{\xi-w}dA(\xi)&=-\ln|w|^2+wC_{\rho_0}^{\Omega\IntComp}(w)+q_{\Omega}.
\end{align*}
where $q_\Omega=\int_{\Omega\setminus\D_r}|\xi|^{-2}dA(\xi)+\ln|r|^2$. Hence,
\begin{align*}
    \dfrac{\ln|w|^2}{w}&\dEquals h(w)+C_{\rho_0}^{\Omega\IntComp}(w)+\dfrac{q_{\Omega}}{w}
\end{align*}

Finally, we consider the case in which $\Omega$ is unbounded. Fix $w_0\in\Omega\IntComp$. If $w\in\Omega\IntComp$ then, applying the quadrature identity to $\frac{\xi}{(w-w_0)(w-\xi)}\in L_a^1(\Omega;\rho_0)$ yields
\begin{align*}
    \int_{\Omega}\dfrac{\xi|\xi|^{-2}}{(\xi-w_0)(\xi-w)}dA(\xi)&=\oint_{\partial\Omega}\dfrac{\xi h(\xi)}{(\xi-w_0)(\xi-w)}d\xi=-\dfrac{wh(w)-w_0h(w_0)}{w-w_0}.
\end{align*}
On the other hand, if $w\in\Omega$ then, by Green's theorem,
\begin{align*}
    \int_{\Omega}\dfrac{\xi|\xi|^{-2}}{(\xi-w_0)(\xi-w)}dA(\xi)&=-\dfrac{\ln|w|^2}{w-w_0}+\oint_{\partial\Omega}\dfrac{\ln|\xi|^2}{(\xi-w_0)(\xi-w)}d\xi\\
    &=\dfrac{-1}{w-w_0}\left(\ln|w|^2-w\oint_{\partial\Omega}\dfrac{\ln|\xi|^2}{\xi(\xi-w)}d\xi+w_0\oint_{\partial\Omega}\dfrac{\ln|\xi|^2}{\xi(\xi-w_0)}d\xi\right)
\end{align*}
Applying Green's theorem to the first integral, and splitting the second to isolate the pole at $0$, we obtain
\begin{align*}
    \int_{\Omega}\dfrac{\xi|\xi|^{-2}}{(\xi-w_0)(\xi-w)}dA(\xi)&=\dfrac{-1}{w-w_0}\left(\ln|w|^2+w\int_{\Omega\IntComp}\dfrac{|\xi|^{-2}}{\xi-w}dA(\xi)+w_0\left(\oint_{\partial(\Omega\setminus\D_r)}\dfrac{\ln|\xi|^2d\xi}{\xi(\xi-w_0)}+\oint_{\partial\D_r}\dfrac{\ln|r|^2d\xi}{\xi(\xi-w_0)}\right)\right)
\end{align*}
We recognize the first integral as a weighted Cauchy transform, the second integral is also recognized as a weighted Cauchy transform after applying Green's theorem, and the third integral may be calculated directly via the residue theorem.
\begin{align*}
    \int_{\Omega}\dfrac{\xi|\xi|^{-2}}{(\xi-w_0)(\xi-w)}dA(\xi)&=\dfrac{-1}{w-w_0}\left(\ln|w|^2-wC_{\rho_0}^{\Omega\IntComp}(w)+w_0\left(\int_{\Omega\setminus\D_r}\dfrac{|\xi|^{-2}}{\xi-w_0}dA(\xi)-\dfrac{\ln|r|^2}{w_0}\right)\right)\\
    &=\dfrac{-1}{w-w_0}\left(\ln|w|^2-wC_{\rho_0}^{\Omega\IntComp}(w)-w_0C_{\rho_0}^{\Omega\setminus\D_r}(w_0)+\ln|r|^2\right).
\end{align*}
Hence for each $w\in\partial\Omega$,
$$\dfrac{\ln|w|^2}{w}\dEquals h(w)+C_{\rho_0}^{\Omega\IntComp}(w)+\dfrac{q_{\Omega,w_0}}{w},$$
where $q_{\Omega,w_0}=w_0C_{\rho_0}^{\Omega\setminus\D_r}(w_0)-w_0h(w_0)+\ln|r|^2$. For the uniqueness of $G$ and $q$, suppose there existed two pairs $G_1,q_1$ and $G_2,q_2$ for which Equation \ref{eqn:LQDCEZero} holds. Subtracting these from each other, and multiplying through by $w$, we find that $wG_1(w)+q_1\dEquals wG_2(w)+q_2$. Both sides of the equation are analytic in $\Omega$ and extend continuously to the boundary, it follows from the identity theorem that $wG_1(w)+q_1=wG_2(w)+q_2$ on $\Omega$. Plugging in $w=0$, we find that $q_1=q_2$, from which it follows that $G_1=G_2$ on $\Omega$.
\end{proof}

With the generalized coincidence equation established, our next goal is to characterize the conditions under which the quadrature function $h$ associated to a given LQD is unique. In particular, we show that $h$ is unique outright when $0\notin\Omega$, and unique modulo an additional point-charge term $q/w$ when $0\in\Omega$. Next, we recover an analogue of the Schwarz function for LQDs (Lemma \ref{lemma:QDChars}): we introduce a \emph{generalized Schwarz function} $S_0$ whose boundary values satisfy $S_0(w)\dEquals\frac{\ln|w|^2}{w}$, and show that $\Omega$ is an LQD if and only if such an $S_0$ exists. We then record three further consequences that will be used repeatedly later: a Sakai-type boundary regularity theorem (only finitely many singular boundary points, each a cusp or double point), an electrostatic characterization of LQDs, and basic invariance properties of the class under scaling, inversion, and power maps.

\subsection{Essential Uniqueness of the Quadrature Function}
A direct consequence of Theorems \ref{thm:LQDCENoZero} and \ref{thm:LQDCEZero} is the essential uniqueness of the quadrature function associated to a given LQD. For non-singular LQDs, we obtain:
\begin{corollary}\label{cor:LQDQuadFuncUniquenessNoZero}
    If $\Omega\in\QD_0(h)$ is non-singular, then $h$ is unique.
\end{corollary}
\begin{proof}
Suppose $0\notin\Omega$ is an LQD with two quadrature functions $h_1$ and $h_2$ (that is, $\Omega\in\QD_0(h_1)\cap\QD_0(h_2)$). Then, $h_1$ and $h_2$ each satisfy Equation \ref{eqn:LQDCENoZero} for the same $G$, as $G$ is independent of $h$. Subtracting these equations, we obtain $h_1\dEquals h_2$. It follows from the identity theorem that $h_1=h_2$. 
\end{proof}

On the other hand, if $0\in\Omega$, then we only obtain uniqueness up to a one-parameter family of quadrature functions.
\begin{corollary}\label{cor:LQDQuadFuncUniquenessZero}
    If $\Omega\in\QD_0(h)$ is a domain containing zero, then $h$ is unique modulo the addition of a point charge $\frac{q}{w}$ at the origin. In particular, $\left\{\widetilde{h}:\Omega\in\QD_0(\widetilde{h})\right\}=\left\{h(w)+\frac{q}{w}:q\in\C\right\}$.
\end{corollary}
\begin{proof}
Suppose $0\in\Omega$ is an LQD with two quadrature functions $h_1$ and $h_2$ (that is, $\Omega\in\QD_0(h_1)\cap\QD_0(h_2)$). Then, $h_1$ and $h_2$ each satisfy Equation \ref{eqn:LQDCEZero}. Noting that $G$ is independent of $h$ (it depends only on the domain), we obtain $\frac{\ln|w|^2}{w}\dEquals h_j(w)+\frac{q_j}{w}+G(w)$ ($q_j\in\C$, $j\in\{1,2\}$). Subtracting these equations and applying the identity theorem yields $h_1(w)-h_2(w)=\frac{q_2-q_1}{w}$ for all $w\in\C$. In particular, there exists $q\in\C$ such that $h_2(w)=h_1(w)+\frac{q}{w}$. 

On the other hand, if $\Omega\in\QD_0(h)$, then $\Omega\in\QD_0\left(h(w)+\frac{q}{w}\right)$ for all $q\in\C$ by inspection of the quadrature identity, using the fact that $f\in L_a^1(\Omega;\rho_0)\Rightarrow f(0)=0$.
\end{proof}

Considering Theorem \ref{thm:LQDCEZero} together with Corollary \ref{cor:LQDQuadFuncUniquenessZero}, we arrive at an apparent paradox: The former asserts that $q$ (in Equation \ref{eqn:LQDCEZero}) is uniquely associated to $\Omega\in\QD_0(h)$, whereas the latter asserts that $h(w)+\frac{q}{w}$ is also a quadrature function for $\Omega$ for any $q\in\C$. What is happening here? The root of the problem is that $q$ (in Equation \ref{eqn:LQDCEZero}) depends on both $\Omega$ and $h$. Hence we must instead consider the sum $h(w)+\frac{q}{w}$ which is, indeed uniquely associated to $\Omega$ (this follows from essentially the same argument as in the proof of Corollary \ref{cor:LQDQuadFuncUniquenessZero}). Now consider the decomposition
$$h(w)+\frac{q}{w}=\left(h(w)-\dfrac{\Res{0}h}{w}\right)+\frac{q+\Res{0}h}{w},$$
where the first term is the part which is residue-free at $0$, and the second captures the residue at $0$. It follows that the second term is, in fact, independent of $h$. This motivates the introduction of the following definition. 
\begin{definition}\label{def:QD0hq}
    Fix $q\in\C$ and $h$ rational such that $\Res{0}h=0$. We denote by $\Omega\in\QD_0(h;q)$ that $\Omega\in\QD_0(h)$, and Equation \ref{eqn:LQDCEZero} is satisfied with this value of $q$.
\end{definition}
It is immediate that if $\Omega\in\QD_0(h;q)$, then the pair $(h,q)$ is uniquely associated to $\Omega$. From the perspective of potential theory, $q$ can be interpreted as the charge at $0$.

\subsection{The Generalized Schwarz Function and Boundary Regularity}
Recall that quadrature domains may be equivalently characterized as those domains which admit a Schwarz function (Lemma \ref{lemma:QDChars}). The coincidence equations suggest that the correct analogue of the classical Schwarz function should have boundary values $\frac{\ln|w|^2}{w}$. We now show that this is indeed the right notion. In particular, we show that a domain $\Omega$ is an LQD if and only if it admits a \emph{generalized Schwarz function}: A function $S_0\in\M(\Omega)$ extending continuously to $\partial\Omega$ such that $S_0(w)\dEquals\frac{\ln|w|^2}{w}$. 

It is important to note that \emph{$S_0$ always has finitely many poles} - a fact which is clear when $\Omega$ is bounded. When $\Omega$ is unbounded, this follows from the fact that $S_0$ is meromorphic at $\infty$, which means that the singularity at infinity (if any) is isolated.

\begin{theorem}\label{thm:LQDSFEEquiv}
If $\Omega$ is a domain with $0,\infty\notin\partial\Omega$, then $\Omega\in\QD_0$ if and only if it admits a generalized Schwarz function.
\end{theorem}
In this case, the generalized Schwarz function $S_0$ is given by the RHS of Equation \ref{eqn:LQDCENoZero} when $0\notin\Omega$ and RHS of Equation \ref{eqn:LQDCEZero} when $0\in\Omega$. In order to complete the proof of Theorem \ref{thm:LQDSFEEquiv}, we will need the following regularity result: 

\begin{lemma}\label{lemma:GSFBoundaryRegularity}
    If $\Omega$ admits a generalized Schwarz function, then $\partial\Omega$ has finitely many singular points, and each is a cusp or a double point.
\end{lemma}
\begin{proof}
Let $S_0$ be a generalized Schwarz function for $\Omega$. Define
$$S(w):=\dfrac{e^{wS_0(w)}}{w}.$$
For $w\in\partial\Omega$, the boundary condition for $S_0$ gives
$$S(w)\dEquals\dfrac{e^{\ln|w|^2}}{w}=\dfrac{|w|^2}{w}=\overline{w}.$$
Since $S_0\in\M(\Omega)$ and is continuous up to $\partial\Omega$, its poles are finite in number and bounded away from $\partial\Omega$. Choose $\epsilon>0$ sufficiently small that $S_0$ is holomorphic in 
$$N_\epsilon=\{w\in\Omega:d(w,\partial\Omega)<\epsilon\},$$
and $d(0,\partial\Omega)>\epsilon$. Then $S$ is holomorphic on $N_\epsilon\cap\Cl(\Omega)$ with boundary values $S(w)\dEquals\overline{w}$. Thus $S$ is a local Schwarz function at every boundary point of $\Omega$. Sakai's regularity theorem (\ref{thm:SakaiRegularity}) then implies that $\partial\Omega$ has finitely many singular points, each a cusp or a double point ($\partial\Omega$ has no degenerate points on account of the assumption that $\Omega=\Int(\Cl(\Omega))$). The finiteness of the singular set follows exactly as in the classical quadrature-domain case (\S3.2 of \cite{Lee_2015}).
\end{proof}

\noindent With this regularity result in hand, we are now prepared to prove Theorem \ref{thm:LQDSFEEquiv}.

\begin{proof}[Proof of Theorem \ref{thm:LQDSFEEquiv}]
Suppose $\Omega\in\QD_0(h)$. Then by Theorems \ref{thm:LQDCENoZero} and \ref{thm:LQDCEZero}, $\Omega$ satisfies Equation \ref{eqn:LQDCEZero} (with $q=0$ when $0\notin\Omega$). Set $S_0$ equal to the RHS of the equation: $S_0(w):=h(w)+qw^{-1}+G(w)$. As $S_0\in\M(\Omega)$ and extends continuously to the boundary ($0\notin\partial\Omega$ by assumption), we find that $S_0(w)\dEquals\frac{\ln|w|^2}{w}$ is a generalized Schwarz function for $\Omega$.

On the other hand, suppose $\Omega$ admits a generalized Schwarz function $S_0$. By Lemma \ref{lemma:GSFBoundaryRegularity}, we know that $\partial\Omega$ is piecewise $C^1$, so Green's theorem applies. Hence, for each $f\in L_a^1(\Omega;\rho_0)$,
\begin{align*}
    \int_{\Omega}\dfrac{f(w)}{|w|^2}dA(w)&=\oint_{\partial\Omega}f(w)\dfrac{\ln|w|^2}{w}dw=\oint_{\partial\Omega}f(w)S_0(w)dw.
\end{align*}
Since $S_0$ has finitely many poles in $\Omega$, let $h$ be the rational function obtained by summing the principal parts of $S_0$ at those poles, normalized so that $h\in\Rat_0(\Omega)$ in the bounded case and $h\in\Rat(\Omega)$ in the unbounded case. Then
$$G=S_0-h$$
is holomorphic in $\Omega$ and extends continuously to $\partial\Omega$. Hence,
$$\int_{\partial\Omega}f(w)G(w)dw=0$$
by Cauchy's theorem. It follows that
\begin{align*}
    \int_{\Omega}\dfrac{f(w)}{|w|^2}dA(w)&=\oint_{\partial\Omega}f(w)h(w)dw,
\end{align*}
so $\Omega\in\QD_0$.
\end{proof}

An immediate consequence of Theorem \ref{thm:LQDSFEEquiv} and Lemma \ref{lemma:GSFBoundaryRegularity} is the following boundary regularity result for LQDs:
\begin{theorem}\label{thm:LQDBoundaryRegularity}
    If $\Omega\in\QD_0$ then $\partial\Omega$ has finitely many singular points, and each is a cusp or a double point. In particular, $\partial\Omega$ is piecewise $C^1$.
\end{theorem}
\begin{proof}
By Theorem \ref{thm:LQDSFEEquiv}, $\Omega$ has a generalized Schwarz function. The conclusion follows by Lemma \ref{lemma:GSFBoundaryRegularity}. 
\end{proof}

\subsection{LQDs: An Electrostatic Perspective}
LQDs have a natural physical interpretation from the perspective of electrostatics. In particular, \emph{a domain $\Omega$ is an LQD if and only if there is a finite configuration of point charges (and multipoles) on the interior of $\Omega$ whose electrostatic field in $\Omega^c$ is equal to that of a charge density $\rho_0$ on $\Omega$}. More formally,
\begin{theorem}\label{thm:ElectrostaticLQDInterpretation}
    Fix a domain $\Omega$ such that $0,\infty\notin\partial\Omega$, and $r\in(0,d(0,\partial\Omega))$. Then,
    \begin{enumerate}
        \item If $0\notin\Omega$ then $\Omega\in\QD_0(h)$ if and only if $C_{\rho_0}^{\Omega}(w)=h(w)$ on $\Omega\IntComp$;
        \item If $0\in\Omega$ then $\Omega\in\QD_0(h;q)$ if and only if $C_{\rho_0}^{\Omega\setminus\D_r}(w)+\frac{\ln|r|^2}{w}=h(w)+\frac{q}{w}$ on $\Omega\IntComp$.
    \end{enumerate}  
\end{theorem}
\noindent See Appendix \ref{subsection:ElectrostaticLQDInterpretation} for the proof.

For example, when $0\notin\Omega$, $\overline{C_{\rho_0}^{\Omega}}$ is precisely the electrostatic field due to the charge density $\rho_0$ on $\Omega$, and $\overline{h}$ is the electrostatic field due to a collection of point charges (and multipoles) at the poles of $h$. This phenomenon is illustrated in Figure \ref{fig:LQDCauchyMotherBody}.

\FloatBarrier
\begin{figure}[H]
\centering
\includegraphics[width=.95\linewidth,height=.27\linewidth]{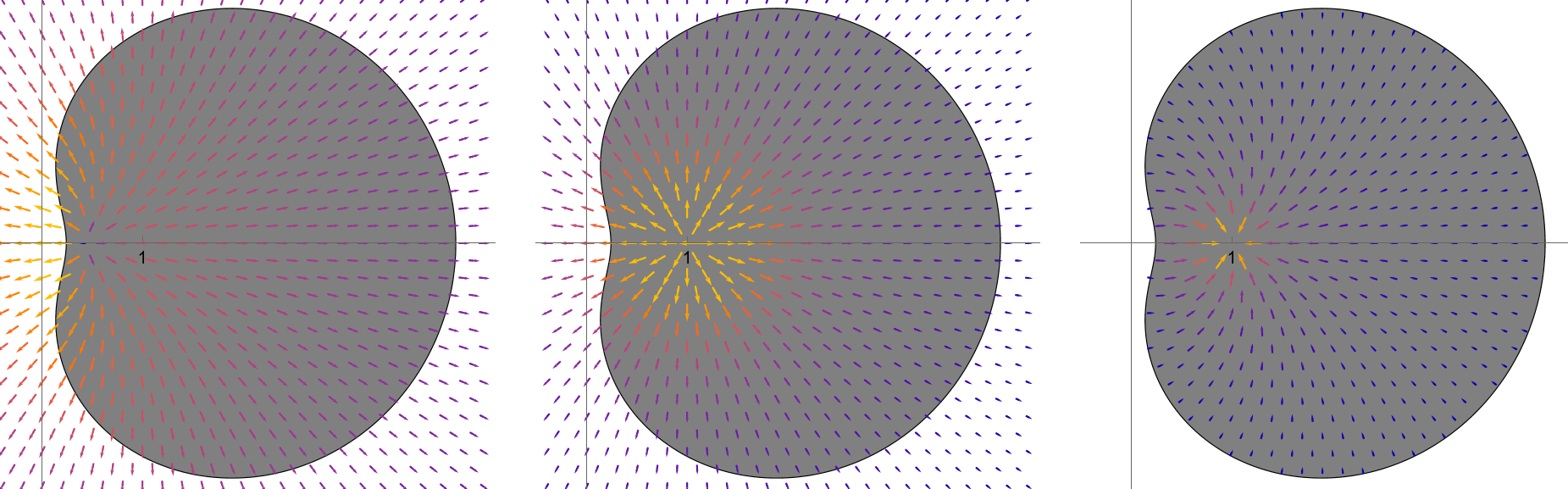}\vspace{.6em}
\caption{The electrostatic field for $\Omega\in\QD_0\left(h\right)$ ($h(w)=\frac{2}{w-1}$) for $\overline{C_{\rho_0}^{\Omega}}$ (left), $\overline{h}$ (center), and $\overline{C_{\rho_0}^{\Omega}}-\overline{h}$ (right).}\label{fig:LQDCauchyMotherBody}\vspace{1em}
\end{figure}

\subsection{Invariance Properties}
The class of log-weighted quadrature domains has a number of invariance properties. In particular, if $a\in\C\setminus\{0\}$, and $k\in\Z_{+}$, then
\begin{enumerate}
    \item Scale invariance (\S\ref{subsubsec:ScalingLaw}):\;\;\;\;\;\; $\Omega\in\QD_0$ $\iff$ $a^{-1}\Omega\in\QD_0$,
    \item Inversion invariance (\S\ref{subsubsec:InversionInvariance}): $\Omega\in\QD_0$ $\iff$ $\Omega^{-1}\in\QD_0$,\footnote{Where $\Omega^{-1}=\{w^{-1}\in\Ch:w\in\Omega\}$.}
    \item Power invariance (\S\ref{subsubsec:PowerInvariance}):\;\;\;\;\; $\Omega\in\QD_0$ $\iff$ $\Omega^{\frac{1}{k}}\in\QD_0$, (assuming $\Omega^{\frac{1}{k}}$ is connected),
\end{enumerate}
where $\Omega^{\frac{1}{k}}:=\{w\in\Ch:w^k\in\Omega\}$ is the full $k$th root preimage of $\Omega$. The specific relationships between the quadrature functions of these domains and their transformed counterparts are given in the \S\ref{subsubsec:ScalingLaw}-\ref{subsubsec:PowerInvariance}. Some of these properties are genuinely new while others are, in some sense, ``inherited'' from invariance properties of classical quadrature domains via Theorem \ref{thm:LQDsFromQDsExponential}. For example, the scale-invariance property of $\QD_0$ can be interpreted as a consequence of the translation invariance of $\QD$ on account of the fact that the exponential map takes addition to multiplication. Of course, LQDs are not generally translation invariant, as this is not a symmetry of the underlying weight $\rho_0(w)=|w|^{-2}$.\\

\noindent We will begin by covering scale invariance and its implications for ``traveling wave'' families of LQDs.

\subsubsection{Scaling Law and Traveling Waves}\label{subsubsec:ScalingLaw}
Fix $\alpha\in\C\setminus\{0\}$ and $\Omega\in\QD_0(h)$, and note that $f\in L_a^1(a^{-1}\Omega;\rho_0)$ $\implies$ $f(a^{-1}w)\in L_a^1(\Omega;\rho_0)$. Applying the quadrature identity, we obtain
\begin{align*}
    \int_{a^{-1}\Omega}\dfrac{f(w)}{|w|^2}dA(w)&=\int_{\Omega}\dfrac{f(a^{-1}w)}{|a^{-1}w|^2}|a|^{-2}dA(w)=\oint_{\partial\Omega}f(a^{-1}w)h(w)dw=\oint_{\partial a^{-1}\Omega}f(w)h(aw)adw.
\end{align*}
Hence, $a^{-1}\Omega\in\QD_0(h(aw)a)$. Carrying out the same argument with $a^{-1}$, we conclude that
\begin{theorem}\label{thm:ScalingLaw}
    For each $a\in\C\setminus\{0\}$ and domain $\Omega$, $\Omega\in\QD_0(h)$ if and only if $a^{-1}\Omega\in\QD_0(h(aw)a)$.
\end{theorem}

We say that a family $\{\Omega_t\}_{t}\subseteq\QD_0(h)$ is a traveling wave if there exists $t_0>0$ and a smooth increasing function $\alpha(t)$ such that $\alpha(\infty)=\infty$ and for each $t>t_0$, $\Omega_t=\alpha(t)\Omega_{t_0}$.

\begin{corollary}\label{cor:TravelingWaveV1}
    If $\{\Omega_t\}_{0<t<\infty}\subseteq\QD_0(h)$ is a traveling wave, then it consists of either disks or exterior disks centered at the origin.
\end{corollary}
\begin{proof}
By Theorem \ref{thm:ScalingLaw}, $\Omega_{t_0}\in\QD_0(h)\cap \QD_0(h(\alpha(t)w)\alpha(t))$. Then by Corollaries \ref{cor:LQDQuadFuncUniquenessNoZero} and \ref{cor:LQDQuadFuncUniquenessZero}, either
\begin{enumerate}
    \item $h(w)=h(\alpha(t)w)\alpha(t)$, when $0\notin\Omega$,
    \item $h(w)=h(\alpha(t)w)\alpha(t)+\frac{q}{w}$ for some $q\in\C$, when $0\in\Omega$.
\end{enumerate}
Applying the assumption that $h$ is rational, and expanding each functional equation in a Laurent series about $w=0$, we find that $h(w)=\frac{c}{w}$ for some $c\in\C$.

In the first case, we find that $c=0$ because $h\in\A(\Omega\IntComp)$ and $0\in\Omega\IntComp$. Hence, $\Omega_t\in\QD_0(0)$ for all $t>t_0$. By Theorem \ref{thm:NullLQDClass}, each $\Omega_t$ must be an exterior disk centered at the origin. Similarly, in the second case, we find that $\Omega_t\in\QD_0\left(\frac{c}{w}\right)$. By Corollary \ref{cor:LQDQuadFuncUniquenessZero} we conclude that $\Omega_t\in\QD_0(0)$. Again by Theorem \ref{thm:NullLQDClass}, each $\Omega_t$ must be a disk centered at the origin.
\end{proof}

\subsubsection{Invariance Under Inversion}\label{subsubsec:InversionInvariance}
Another interesting property of the class of LQDs is their invariance under inversion. This is essentially a consequence of the fact that $(w^{-1})^{\ast}\frac{dA(w)}{|w|^2}=\frac{|-w^{-2}|^2dA(w)}{|w^{-1}|^2}=\frac{dA(w)}{|w|^2}$, where $(w^{-1})^{\ast}$ is the pullback by $w\mapsto w^{-1}$. The following theorem is a direct generalization of the Graven \& Makarov's \cite{GravenMakarov2025} result for non-singular LQDS.
\begin{theorem}\label{thm:LQDInversion}
If $\Omega\subseteq\Ch$ then $\Omega\in\QD_0(h)$ if and only if $\Omega^{-1}\in\QD_0(-h(w^{-1})w^{-2})$.
\end{theorem}
\begin{proof}[Proof of Theorem \ref{thm:LQDInversion}]
If $\Omega\in\QD_0(h)$ then $f\in L_a^1(\Omega^{-1};\rho_0)$ $\implies$ $f(w^{-1})\in L_a^1(\Omega;\rho_0)$. Taking the following integral, inverting, applying the quadrature identity, and inverting back, we obtain
\begin{align*}
\int_{\Omega^{-1}}\dfrac{f(w)}{|w|^2}dA(w)&=\int_{\Omega}\dfrac{f(w^{-1})}{|w|^2}dA(w)\\
&=\oint_{\partial\Omega}f(w^{-1})h(w)dw\\
&=\oint_{\partial\Omega^{-1}}f(w)(-h(w^{-1})w^{-2})dw
\end{align*}
so $\Omega^{-1}\in\QD_0\left(-h(w^{-1})w^{-2}\right)$. This implies the backwards direction as well because $\widetilde{h}(w)=-h(w^{-1})w^{-2}$ $\implies$ $-\widetilde{h}(w^{-1})w^{-2}=h(w)$.
\end{proof}

When $\Omega\in\QD_0(h;q)$ is singular and unbounded (in which case $\Omega^{-1}\in\QD_0(\widetilde{h},\widetilde{q})$ is likewise singular and unbounded), we obtain a direct relation between $q$ and $\widetilde{q}$ in terms of the $\rho_0-$weighted area of $\Omega\IntComp$, $A_{\rho_0}(\Omega\IntComp)=\int_{\Omega\IntComp}|\xi|^{-2}dA(\xi)$.

\begin{corollary}\label{cor:LQDInversion}
If $\Omega\subseteq\Ch$ contains both $0$ and $\infty$, then $\Omega\in\QD_0(h;q)$ if and only if $\Omega^{-1}\in\QD_0(-h(w^{-1})w^{-2};-q-A_{\rho_0}(\Omega\IntComp))$.
\end{corollary}
\begin{proof}
That $q\mapsto-q$, follows by substituting $w\mapsto w^{-1}$ in Equation \ref{eqn:LQDCEZero}, and multiplying through by $-w^{-2}$ to obtain
\begin{align*}
    \dfrac{\ln|w|^2}{w}&\dEquals -h(w^{-1})w^{-2}-\dfrac{q}{w}-\dfrac{G(w^{-1})}{w^2}.
\end{align*}
Note that $\frac{G(w^{-1})}{w^2}\in\A(\Omega^{-1})$, unless $\infty\in\Omega$. In this case, $0\in\Omega^{-1}$, and $G(\infty)=0$, so
$$-\dfrac{G(w^{-1})}{w^2}=\dfrac{\Res{\infty}G}{w}-\dfrac{G(w^{-1})+w\Res{\infty}G}{w^2}=\dfrac{\Res{\infty}G}{w}+\widetilde{G}(w),$$
where $\widetilde{G}\in\A(\Omega^{-1})$. Finally, note that
\begin{align*}
    G(w)&=\int_{\Omega\IntComp}\dfrac{|\xi|^{-2}dA(\xi)}{w-\xi}=\dfrac{1}{w}\int_{\Omega\IntComp}|\xi|^{-2}dA(\xi)+O(w^{-2}),
\end{align*}
so $\Res{\infty}G=-A_{\rho_0}(\Omega\IntComp)$.
\end{proof}

An important consequence of inversion invariance is that it gives a bijective correspondence between bounded singular LQDs and unbounded non-singular LQDs. Hence, we can study one family by considering corresponding domains in the other.

\subsubsection{The Power Map}\label{subsubsec:PowerInvariance}
Considering that $\ln|w^k|=k\ln|w|$, another natural transformation to investigate is the power map $w\mapsto w^k$. For a domain $\Omega\subseteq\Ch$, we therefore introduce its full $k$th root preimage, $\Omega^{\frac{1}{k}}:=\{w\in\Ch:w^k\in\Omega\}$, which is $k-$fold rotationally symmetric. The theorem below shows that $\Omega\in\QD_0$ if and only if $\Omega^{\frac{1}{k}}$ is a disjoint union of LQDs (if $\Omega^{\frac{1}{k}}$ is connected, then it is a disjoint union of one LQD). The key point is that the weight $\rho_0(w)=|w|^{-2}$ interacts especially well with the change of variables induced by $w\mapsto w^k$, so the log-weighted quadrature identity is preserved. In particular, this result gives a systematic way to generate symmetric examples and, conversely, to represent symmetric LQDs via simpler ones.

\begin{theorem}\label{thm:PowerInvariance}
    Fix a domain $\Omega$, and $k\in\Z_{+}$. Then $\Omega\in\QD_0$ $\iff$ $\Omega^{\frac{1}{k}}=\bigsqcup_{j}\Omega_j$, where $\Omega_j\in\QD_0$. In this case, there exist rational functions $h$ and $\{h_j\}_j$ such that $\Omega\in\QD_0(h)$, $\Omega_j\in \QD_0(h_j)$, and $\sum_{j}h_j=\frac{1}{k}w^{k-1}h(w^k)$.
\end{theorem}
\emph{When $\Omega^{\frac{1}{k}}$ is connected, Theorem \ref{thm:PowerInvariance} reduces to the statement $\Omega\in\QD_0(h)$ $\iff$ $\Omega^{\frac{1}{k}}\in\QD_0\left(\frac{1}{k}w^{k-1}h(w^k)\right)$.} Before proceeding with the proof, we will need the following lemma (proof in Appendix \ref{proof:LQDDomainSymmetryToQuadSymmetry}) which asserts a that $k-$fold rotational symmetry of an LQD implies a certain symmetry of its quadrature function.
\begin{lemma}\label{lemma:LQDDomainSymmetryToQuadSymmetry}
If $\Omega=\bigsqcup_{j}\Omega_j$, where $\Omega_j\in\QD_0(h_j)$, is a $k-$fold rotationally symmetric disjoint finite union of LQDs, then there exists a rational function $g$ such that $\sum_{j}h_j(w)=:h(w)=w^{k-1}g(w^k)$. 
\end{lemma}

\begin{proof}[Proof of Theorem \ref{thm:PowerInvariance}]
For the forward direction, we suppose that $\Omega\in\QD_0(h)$. By Theorem \ref{thm:LQDSFEEquiv}, $\Omega$ admits a generalized Schwarz function $S_0(w)\dEquals\frac{\ln|w|^2}{w}$. We then consider $w\in\partial\Omega^{\frac{1}{k}}$ ($\implies w^k\in\partial\Omega$), so
$$\widetilde{S_0}(w):=\frac{1}{k}w^{k-1}S_0(w^k)=\frac{1}{k}w^{k-1}\frac{\ln|w^k|^2}{w^k}=\frac{\ln|w|^2}{w}.$$
Fix a connected component $\Omega_j\subseteq\Omega^{\frac{1}{k}}$ and note that at each $w\in\Omega^{\frac{1}{k}}\supseteq\Omega_j$ ($\implies w^k\in\Omega$), $\widetilde{S_0}$ is meromorphic in $\Omega_j$ as the product and composition of a meromorphic function with a polynomial. It also inherits continuity up to the boundary from $S_0$. Hence, $\widetilde{S_0}$ is a generalized Schwarz function for each connected component $\Omega_j$ of $\Omega^{\frac{1}{k}}$, and we conclude that $\Omega^{\frac{1}{k}}$ is a disjoint union of LQDs.\\

For the reverse direction, suppose that $\Omega^{\frac{1}{k}}=\bigsqcup_{j}\Omega_j$, where $\Omega_j\in\QD_0(h_j)$. As $\Omega^{\frac{1}{k}}$ is $k-$fold symmetric about the origin, Lemma \ref{lemma:LQDDomainSymmetryToQuadSymmetry} implies that there exists a rational function $g$ such that
$$\sum_{j}h_j(w)=:h(w)=\frac{1}{k}w^{k-1}g(w^k).$$
If $w^k\in\Omega$, that means that there exists $z\in\Omega^{\frac{1}{k}}$ such that $w^k=z^k$. In particular, $e^{2\pi i\frac{j}{k}}w\in\Omega^{\frac{1}{k}}$ for some $j\in\Z$. As $\Omega^{\frac{1}{k}}$ is $k-$fold rotationally symmetric, this means that $e^{2\pi i\frac{j}{k}}w\in\Omega^{\frac{1}{k}}$ for every $j\in\Z$. In particular, $\Omega^{\frac{1}{k}}$ is a $k-$to$-1$ cover of $\Omega$ under the map $w\mapsto w^k$ (and hence on the boundary as well). Fix $f\in L_a^1(\Omega;\rho_0)$, so that $f(w^k)\in L_a^1(\Omega^{\frac{1}{k}};\rho_0)$. Then,
\begin{align*}
    \int_{\Omega}\dfrac{f(w)}{|w|^2}dA(w)&=\dfrac{1}{k}\int_{\Omega^{\frac{1}{k}}}\dfrac{f(w^k)}{|w^k|^2}k^2|w|^{2(k-1)}dA(w)=k\sum_j\int_{\Omega_j}\dfrac{f(w^k)}{|w|^2}dA(w)=k\sum_j\oint_{\partial\Omega_j}f(w^k)h_j(w)dw
\end{align*}
As each $h_j$ is residue-free in $\Omega_j^c$, we may deform each contour $\partial\Omega_j$ to the contour $\partial\Omega^{\frac{1}{k}}$ without changing the value of the integral. Doing this, we obtain
\begin{align*}
    \int_{\Omega}\dfrac{f(w)}{|w|^2}dA(w)&=k\oint_{\partial\Omega^{\frac{1}{k}}}f(w^k)\sum_jh_j(w)dw=\oint_{\partial\Omega^{\frac{1}{k}}}f(w^k)w^{k-1}g(w^k)dw=\oint_{\partial\Omega}f(w)g(w)dw.
\end{align*}
That is, $\Omega\in\QD_0(g)$. The desired conclusion follows.
\end{proof}

\section{The Inverse and Direct Problems for LQDs}\label{sec:LQDInvDirectProbs}

The general results of §2 make it possible to begin addressing the inverse and direct problems for log-weighted quadrature domains in explicit families. The aim of the present section is to analyze the simplest such families, both in order to obtain concrete classification theorems and to illustrate the new geometric features introduced by the singular weight $\rho_0(w)=|w|^{-2}$. The null case, $\QD_0(0)$, provides the most basic test of the theory: the singularity at the origin induces an additional family of null LQDs.

\subsection{Classification of Null LQDs}\label{subsec:NullLQDs}
An LQD $\Omega$ is referred to as a \emph{null LQD} if $0$ is its quadrature function: $\Omega\in\QD_0(0)$. We prove the following classification theorem for null LQDs.
\begin{theorem}\label{thm:NullLQDClass}
A domain is a null LQD if and only if it is a disk or exterior disk centered at the origin.
\end{theorem}
\begin{proof}[Proof of Theorem \ref{thm:NullLQDClass}]\label{proof:NullLQDClass}
Let $\Omega\in\QD_0(0)$. We first note that $\Omega$ must contain at least one of $0,\infty$ because otherwise $1\in L_a^1(\Omega;\rho_0)$ and
$$0<\int_{\Omega}|w|^{-2}dA(w)=\oint_{\partial\Omega}1\cdot0dw=0.$$
Then, by Equations \ref{eqn:LQDCENoZero} and \ref{eqn:LQDCEZero}, we find that $\ln|w|^2\dEquals wG(w)+q$ for some $q\in\C$ ($q=0$ if $0\notin\Omega$) and $G\in\A(\Omega)$.\\\\
If $\Omega$ is unbounded then $G\in\A_0(\Omega)$ and there exists $q'$, $\widetilde{G}\in\A_0(\Omega)$ such that
$$\ln\left|we^{-q'}\right|^2\dEquals \widetilde{G}(w)\;\;\;\;\implies\;\;\;\;\left|we^{-q'}\right|^2\dEquals e^{\widetilde{G}(w)},$$
where $e^{\widetilde{G}(w)}-1\in\A_0(\Omega)$. Hence rearranging the above equation we obtain $U\in\A_0(\Omega)$ and $A\in\C$ for which $|w|^2\dEquals U(w)+A$. Thus if $f\in\A_0(\Omega)$,
$$\int_{\Omega}f(w)dA(w)=\oint_{\partial\Omega}f(w)\overline{w}dw=\oint_{\partial\Omega}f(w)\dfrac{U(w)+A}{w}dw.$$
If $0\notin\Omega$, then this equals $0$. On the other hand, if $0\in\Omega$, then we obtain $f(0)\left(U(0)+A\right)$. In particular, $\Omega$ is either a (classical) null quadrature domain or a (classical) one-point quadrature domain with its quadrature node at $w=0$. 

Aharonov \& Shapiro \cite{AharonovShapiro} showed that a finitely connected bounded domain is a one-point QD if and only if it is a disk. When combined with Lee \& Makarov's \cite{Lee_2015} topological bounds for QDs, we are able to drop the assumption of finite-connectedness. Hence, if $\Omega\in\QD_0(0)$ is bounded, it is a disk centered at $0$.

On the other hand, Sakai \cite{SakaiNullQD} showed that a domain with compact boundary is a null QD with respect to the test class $L_a^1(\Omega)$ if and only if it is an exterior disk or an exterior ellipse. However, when considered with respect to the larger test class $A_0(\Omega)$, only exterior disks are null (that ellipses are not follows from direct computation). Hence, if $\Omega\in\QD_0(0)$ is unbounded, it is an exterior disk centered at $0$.

For the reverse direction, take $r>0$ and $f\in L_a^1(\D_r;\rho_0)$, which implies $f(0)=0$. Integrating against $\rho_0$ and changing to polar coordinates, we obtain
\begin{align*}
    \int_{\D_r}\dfrac{f(w)}{|w|^2}dA(w)&=2\int_{0}^r\oint_{\partial\D}\dfrac{f(sw)}{sw}dw ds=0
\end{align*}
The last equality follows from the fact that $f(0)=0$, so $\frac{f(sw)}{sw}$ is analytic in $\D$. The argument for $(\D_r)\IntComp$ is entirely analogous.
\end{proof}

\subsection{Partial Classification of Non-Singular One-Point LQDs}\label{subsec:BasicOnePTLQDEx}
A classic result due to Aharonov \& Shapiro \cite{AharonovShapiro} states that a finitely connected bounded domain is a one-point classical QD if and only if it is a disk. We prove the following analog of this result for LQDs: A simply connected bounded domain not containing zero is a one-point LQD if and only if it is biholomorphic to a disk via the exponential map $w\mapsto e^{w}$. In particular, we show the following:
\begin{theorem}\label{thm:NZOnePtLQD}
Fix $\alpha,w_0\in\C\setminus\{0\}$. If $\Omega$ is a simply connected bounded domain not containing zero, then $\Omega\in\QD_0\left(\frac{\alpha}{w-w_0}\right)$ if and only if $0<\alpha\leq\pi^2$ and $\Omega=\{w\in\C:|\ln(ww_0^{-1})|^2<\alpha\}$.
\end{theorem}
Moreover, $\partial\Omega$ develops a double point at $\alpha=\pi^2$. Here and throughout the paper, $\ln$ denotes the principal branch of the natural logarithm, with branch cut along the negative real axis.

\begin{proof}[Proof of Theorem \ref{thm:NZOnePtLQD}]
Beginning with the reverse direction, consider the bounded domain
$$\Omega=\{w\in\C:\left|\ln(ww_0^{-1})\right|^2<\alpha\}$$
for some $w_0\neq0$ and $\alpha>0$. To solve the direct problem for $\Omega$, we would like to show that $\Omega\in\QD_0$ and determine its quadrature function. We proceed by direct computation of the quadrature identity. Note that if $w\in\partial\Omega$, then $\left|\ln(ww_0^{-1})\right|^2=\alpha$, so
$$\dfrac{\ln|w|^2}{w}=\dfrac{\ln(\overline{ww_0^{-1}})+\ln(w\overline{w_0})}{w}=\dfrac{1}{w}\left(\dfrac{\alpha}{\ln(ww_0^{-1})}+\ln(w\overline{w_0})\right)=\dfrac{1}{w}\dfrac{\alpha w_0}{w-w_0}+G(w)=:S_0(w)$$
for some $G\in\A(\Omega)$. $S_0$ is meromorphic in $\Omega$ and extends continuously to the boundary with boundary values $\frac{\ln|w|^2}{w}$. Hence $S_0$ is a generalized Schwarz function for $\Omega$ so, $\Omega\in\QD_0$ by Theorem \ref{thm:LQDSFEEquiv}. 

We will now compute the associated quadrature function $h$ directly. Note that $0\notin\Omega$ so, by Green's theorem and Cauchy's integral theorem, for each $f\in L_a^1(\Omega;\rho_0)$,
\begin{align*}
    \int_{\Omega}\dfrac{f(w)}{|w|^2}dA(w)&=\oint_{\partial\Omega}f(w)\dfrac{\ln|w|^2}{w}dw\\
    &=\oint_{\partial\Omega}f(w)\left(\dfrac{1}{w}\dfrac{\alpha w_0}{w-w_0}+G(w)\right)dw\\
    &=\oint_{\partial\Omega}f(w)\dfrac{\alpha}{w-w_0}dw+\oint_{\partial\Omega}f(w)\left(G(w)-\dfrac{\alpha}{w}\right)dw\\
    &=\oint_{\partial\Omega}f(w)\dfrac{\alpha}{w-w_0}dw.
\end{align*}
Hence, $\Omega\in\QD_0\left(\frac{\alpha}{w-w_0}\right)$.\\

For the forward direction, suppose that $\Omega\in\QD_0\left(\frac{\alpha}{w-w_0}\right)$ for some $\alpha,w_0\in\C\setminus\{0\}$, and that $\Omega$ is a bounded simply connected domain not containing zero. We will show that $\alpha>0$ and $\partial\Omega=\{w\in\C:|\ln(ww_0^{-1})|^2=\alpha\}$. Firstly, the fact that $\alpha>0$ follows immediately from applying the quadrature identity to $1\in L_a^1(\Omega;\rho_0)$:
\begin{align*}
    0<\int_{\Omega}\dfrac{dA(w)}{|w|^2}&=\oint_{\partial\Omega}\dfrac{\alpha}{w-w_0}dw=\alpha.
\end{align*}

By Theorem \ref{thm:ScalingLaw}, $\Omega\in\QD_0\left(\frac{\alpha}{w-w_0}\right)$ if and only if $w_0^{-1}\Omega\in\QD_0\left(\frac{\alpha}{w-1}\right)$. Hence, we can wlog consider the case in which $w_0=1$. Next, by Theorem \ref{thm:LQDCENoZero}, there exists $G\in\A(\Omega)$ such that $\frac{\ln|w|^2}{w}\dEquals\frac{\alpha}{w-1}+G(w)$. Hence,
\begin{align*}
    \ln|w|^2&\dEquals\dfrac{w\alpha}{w-1}+wG(w)
\end{align*}
As $0\notin\Omega$ is bounded and simply connected, we know that $\Omega\IntComp$ is connected and contains both $0$ and $\infty$. Hence, isolating the simple pole of $\frac{1}{\ln(w)}$ at $w=1$, we find that $\frac{w\alpha}{w-1}=\frac{\alpha}{\ln(w)}+\widetilde{G}(w)$ for some $\widetilde{G}\in\A(\Omega)$. Therefore
\begin{align*}
    \ln(\overline{w})+\ln(w)&\dEquals \frac{\alpha}{\ln(w)}+\widetilde{G}(w)+wG(w)
\end{align*}
Finally, rearranging, we obtain
\begin{align*}
    \overline{\ln(w)}&\dEquals \frac{\alpha}{\ln(w)}+\widetilde{G}(w)+\ln(w)+wG(w).
\end{align*}
Hence, $\widehat{G}(w):=\alpha+\ln(w)(\widetilde{G}(w)+\ln(w)+wG(w))$ is a function analytic in $\Omega$ for which
\begin{align*}
    |\ln(w)|^2&\dEquals \widehat{G}(w).
\end{align*}
Therefore $\Im(\widehat{G})=0$ so, by the maximum principle, $\widehat{G}$ is real-valued and analytic in $\Omega$ and hence constant on $\Omega$. As $\Re(\widehat{G})\geq0$ on $\partial\Omega$, we conclude that $\widehat{G}$ is a non-negative real constant on $\Omega$. In particular, $\exists c\geq0$ such that $|\ln(w)|^2\dEquals c$. Then, applying the quadrature identity calculation from the first part, we find that $c=\alpha$. In particular, this demonstrates that if $\Omega\in\QD_0\left(\frac{\alpha}{w-1}\right)$, then $\partial\Omega\subseteq\{w\in\C:|\ln(w)|^2=\alpha\}=:\Gamma_\alpha$. 

Finally, we must determine which connected component of $\Gamma_\alpha^c$ is $\Omega$. Note that $1\in\Omega$ because $1$ is the location of the unique quadrature node. Furthermore, $|\ln(1)|^2-\alpha=-\alpha<0$, so $\Omega$ must be the connected component of $\Gamma_\alpha^c$ on which $|\ln(w)|^2<\alpha$. This concludes the proof of Theorem \ref{thm:NZOnePtLQD}.
\end{proof}

The Riemann map $\varphi:\D\rightarrow\C$ associated to $\Omega$ such that $\varphi(0)=w_0$ and $\varphi'(0)>0$ is given by $\varphi(z)=w_0e^{z\sqrt{\alpha}}$. We note that $\varphi$ is univalent in $\D$ if and only if $0<\alpha\leq\pi^2$, which is precisely the set of $\alpha$ for which $|\ln(ww_0^{-1})|^2<\alpha$ is simply connected.

\section{Simply Connected Log-Weighted Quadrature Domains}\label{sec:SimplyConnectedLQDs}
This section is primarily concerned with solving the inverse and direct problems for simply connected log-weighted quadrature domains. Recall that the {\it inverse problem} for log-weighted quadrature domains is concerned with recovering an LQD from its quadrature function. Conversely, the {\it direct problem} for LQDs is concerned with determining whether a given domain is an LQD and, if so, identifying its quadrature function. 

In the simply connected setting, domains are represented by their associated Riemann map. Hence, we may reduce the inverse and direct problems to those of characterizing the Riemann map in terms of the quadrature function and vice-versa. In the case of classical quadrature domains, this scheme is aided by the fact that the associated Riemann map is rational: the inverse and direct problems reduce to the determination of the relationship between the finitely many coefficients of the quadrature function and the finitely many coefficients of the Riemann map.

The situation is more complicated for log-weighted quadrature domains. While the Riemann map associated to a simply connected LQD is not necessarily rational, it turns out that it can still be represented in terms of rational functions. In particular, if $\varphi$ is the Riemann map associated to a simply connected domain $\Omega$ and has the {\it inner-outer factorization} $\varphi(z)=\varphi_{\rm in}\varphi_{\rm out}$ (\S\ref{subsec:InnerOuterFactorization}) then
\begin{theorem}\label{thm:SCLQDCharacterization}
$\Omega\in\QD_0$ if and only if $\varphi_{\rm out}$ extends to the exponential of a rational function.
\end{theorem}
The proof is in \S\ref{subsec:InnerOuterFactorization}, following a more formal discussion of the inner/outer factorization. As we shall see, $\varphi_{\rm in}$ is a product of Blaschke factors, so it is also rational. Combining these facts reduces the inverse and direct problems to the determination of the relationship between the finitely many poles and coefficients of the quadrature function and those of the rational functions comprising the Riemann map.

\subsection{Inner/Outer Function Factorization}\label{subsec:InnerOuterFactorization}
We use only the simplest form of the inner–outer factorization, namely the case of bounded analytic functions that extend continuously to the boundary and have no boundary zeros. In this setting no singular inner factor appears: the inner part is a finite Blaschke product determined by the zeros, and the outer part is given by the Poisson integral formula. We record the corresponding forms on $\D$ and $\D\IntComp$.

\begin{lemma}[Inner/Outer Factorization]\label{lemma:InnerOuterFactorization}\;
\begin{enumerate}[label=(\alph*)]
    \item Suppose $f\in H^\infty(\D)\cap C^0(\Cl(\D))$ is non-vanishing on $\partial\D$. Then $f$ admits a unique factorization $f=f_{\rm in}f_{\rm out}$, where $f_{\rm in}$ is the finite Blaschke product consisting of the roots of $f$ in $\D$, and
    \begin{equation}\label{eqn:InnerOuterFactorizationOuterFormula}
    f_{\rm out}(z)=\exp\left(\dfrac{1}{2}\oint_{\partial\D}\dfrac{\xi+z}{\xi-z}\dfrac{\ln|f(\xi)|^2}{\xi}d\xi\right),\;\;\;z\in\D.
\end{equation}
    \item Suppose $f\in H^\infty(\D\IntComp)\cap C^0(\Cl(\D\IntComp))$ is non-vanishing on $\partial\D\IntComp$. Then $f$ admits a unique factorization $f(z)=f_{\rm in}f_{\rm out}$, where $f_{\rm in}$ is the finite Blaschke product consisting of the roots of $f$ in $\D\IntComp$, and
    \begin{equation}\label{eqn:InnerOuterFactorizationOuterFormulaUnbounded}
    f_{\rm out}(z)=\exp\left(\dfrac{1}{2}\oint_{\partial\D\IntComp}\dfrac{\xi+z}{\xi-z}\dfrac{\ln|f(\xi)|^2}{\xi}d\xi\right),\;\;\;z\in\D\IntComp
    \end{equation}
\end{enumerate}
\end{lemma}
Part (a) is the standard inner/outer factorization in the absence of boundary zeros, so the singular inner factor is absent, and the inner part reduces to a finite Blaschke product (see \cite{MashreghiInnnerFunctions2013}, Theorem 1.12 for a more detailed discussion). Part (b) follows by applying part (a) to the reflection of $f$, $f^{\#}$ on $\D$, then reflecting back.\\

For $\lambda\in\D$, let $$b_{\lambda}(z):=\frac{\overline{\lambda}}{|\lambda|}\frac{z-\lambda}{\overline{\lambda}z-1},$$
and set $b_0(z):=z$. We will use the following elementary properties of $b_\lambda$ repeatedly: $b_\lambda$ has a simple root at $\lambda$, $|b_\lambda|=1$ on $\partial\D$, and $b_\lambda^{\#}(z)=b_\lambda(z)^{-1}$.
\\

We now consider the consequences of this representation when $f$ is a Riemann map. The boundary regularity result established in \S\ref{sec:TheoryOfLQDs} means Lemma 4.1 applies directly. In this case, the preceding integral representation can be rewritten in terms of the Faber transform.
\begin{theorem}\label{thm:RiemannMapInnerOuterRepresentation}
If $\varphi$ is the Riemann map associated to a simply connected domain $\Omega$ with compact piecewise $C^1$ boundary such that $0\notin\partial\Omega$, then $\varphi(z)=\varphi_{\rm in}\varphi_{\rm out}$, where
\begin{equation}\label{eqn:RiemannMapInnerOuterRepresentation}
    \varphi_{\rm out}(z)=C\exp\left(\Phi_{\varphi}^{-1}\left(\AnalyticIn{\ln|w|^2}{\Omega\IntComp}\right)^{\#}(z)\right),
\end{equation}
for some $C\in\C$. If $\Omega$ is bounded, then $\varphi_{\rm in}(z)=1$ when $0\notin\Omega$, and $\varphi_{\rm in}(z)=b_{z_0}(z)$ when $0\in\Omega$. If $\Omega$ is unbounded, then $\varphi_{\rm in}(z)=z$ when $0\notin\Omega$, and $\varphi_{\rm in}(z)=zb_{z_0}(z)$ when $0\in\Omega$. $z_0$ is the unique root of $\varphi$.
\end{theorem}
Because $\varphi$ has a simple pole at infinity in the unbounded case, it does not belong to $H^\infty(\D\IntComp)$, so the classical inner-outer factorization of Lemma \ref{lemma:InnerOuterFactorization} does not apply directly. However, the normalized map $f(z) := \varphi(z)/z$ is bounded and non-vanishing on the boundary, so $f \in H^\infty(\D\IntComp) \cap C^0(\Cl(\D\IntComp))$. Therefore, we define the factorization of $\varphi$ by decomposing $f = f_{\rm in}f_{\rm out}$ via Lemma \ref{lemma:InnerOuterFactorization}. We define the \emph{outer factor} of the Riemann map as $\varphi_{\rm out} := f_{\rm out}$ and the \emph{extended inner factor} as $\varphi_{\rm in}(z) := zf_{\rm in}(z)$. The proof of Theorem \ref{thm:RiemannMapInnerOuterRepresentation} is in Appendix \ref{proof:RiemannMapInnerOuterRepresentation}.\\

\noindent We are now prepared to prove Theorem \ref{thm:SCLQDCharacterization}, the central result of this paper.
\begin{proof}[Proof of Theorem \ref{thm:SCLQDCharacterization}]\;\\
We will begin by considering the bounded case: Suppose $\Omega\in\QD_0(h)$ is a simply connected bounded domain with Riemann map $\varphi:\D\rightarrow\Omega$. By the definition of $\QD_0$, $0\notin\partial\Omega$. Hence, by Theorem \ref{thm:RiemannMapInnerOuterRepresentation}, $\varphi=\varphi_{\rm in}\varphi_{\rm out}$, where $\varphi_{\rm in}=1$ when $0\notin\Omega$, $\varphi_{\rm in}=b_{z_0}$ when $0\in\Omega$ ($z_0$ the unique root of $\varphi$ in $\D$), and $\varphi_{\rm out}=Ce^{r^{\#}}$, where
\begin{align*}
r(z)&=\Phi_{\varphi}^{-1}\left(\AnalyticIn{\ln|w|^2}{\Omega\IntComp}\right)(z).
\end{align*}
By the generalized coincidence equation (\ref{eqn:LQDCENoZero}, \ref{eqn:LQDCEZero}), there exists $G\in\A(\Omega)$ and $q\in\C$ ($q=0$ if $0\notin\Omega$) such that $\ln|w|^2\dEquals wh(w)+wG(w)+q$. Substituting this into our formula for $r$, we obtain
\begin{align*}
r&=\Phi_{\varphi}^{-1}\left(\AnalyticIn{wh(w)+wG(w)+q}{\Omega\IntComp}\right).
\end{align*}
Next, $wG(w)+q\in\A(\Omega)$, so $\AnalyticIn{wG(w)+q}{\Omega\IntComp}=0$. It follows from calculus of residues that
$$\AnalyticIn{wh(w)}{\Omega\IntComp}=wh(w)+\Res{\infty}h,$$
Hence,
\begin{align*}
r&=\Phi_{\varphi}^{-1}\left(wh(w)+\Res{\infty}h\right).
\end{align*}
We conclude that $r$ is rational, as the Faber transform of a rational function.

For the reverse direction, suppose $\varphi_{\rm out}=e^{r^{\#}}$ with $r$ rational. Then, as $|\varphi_{\rm in}|\dEquals1$, we find that
$$\ln\left|\varphi\right|^{2}\dEquals r+r^{\#}.$$
Hence,
\begin{align*}
    S_0(w):=\dfrac{\left(r+r^{\#}\right)\circ\psi(w)}{w}
\end{align*}
is meromorphic in $\Omega$, extends continuously to $\partial\Omega$, and satisfies
$$S_0(w)\dEquals\dfrac{\ln|w|^2}{w}.$$
Thus $S_0$ is a generalized Schwarz function for $\Omega$, and Theorem \ref{thm:LQDSFEEquiv} yields $\Omega\in\QD_0$.\\

We will now consider the {\bf unbounded} case: Suppose $\Omega\in\QD_0(h)$ is a simply connected unbounded domain with Riemann map $\varphi:\D\IntComp\rightarrow\Omega$. By Theorem \ref{thm:RiemannMapInnerOuterRepresentation}, $\varphi=\varphi_{\rm in}\varphi_{\rm out}$, where $\varphi_{\rm in}(z)=z$ when $0\notin\Omega$, $\varphi_{\rm in}(z)=zb_{z_0}(z)$ when $0\in\Omega$ ($z_0$ the unique root of $\varphi$ in $\D$), and $\varphi_{\rm out}=Ce^{r^{\#}}$, where
\begin{align*}
r(z)&=\Phi_{\varphi}^{-1}\left(\AnalyticIn{\ln|w|^2}{\Omega\IntComp}\right)(z)
\end{align*}
As $\Omega\in\QD_0(h)$, Theorems \ref{thm:LQDCENoZero} and \ref{thm:LQDCEZero} tell us there exists $G\in\A_0(\Omega)$ and $q\in\C$ ($q=0$ if $0\notin\Omega$) such that $\ln|w|^2\dEquals wh(w)+wG(w)+q$. Substituting this into our formula for $r$, we obtain
\begin{align*}
r&=\Phi_{\varphi}^{-1}\left(\AnalyticIn{wh(w)+wG(w)+q}{\Omega\IntComp}\right).
\end{align*}
By calculus of residues, $\AnalyticIn{wG(w)+q}{\Omega\IntComp}=q-\Res{\infty}G=:q'$. As $wh(w)\in\A(\Omega\IntComp)$, $\AnalyticIn{wh(w)}{\Omega\IntComp}=wh(w)$. We find that
\begin{align*}
r&=\Phi_{\varphi}^{-1}\left(wh(w)\right)+q'.
\end{align*}
Hence $r$ is rational, as the Faber transform of a rational function. The argument for the reverse direction is entirely analogous to the argument in the bounded case. 
\end{proof}

In view of Theorem \ref{thm:SCLQDCharacterization}, it is convenient to adopt the following normalized factorizations of $\varphi$. Under this normalization, $\varphi_{\rm out}$ is independent of whether $\Omega\in\QD_0(h)$ is singular, and is encoded by a single rational function $r$.
\begin{equation}\label{eqn:LQDBoundedInnerOuterFactorization}
\begin{alignedat}{2}
    \varphi(z)&=w_0e^{r^{\#}(z)},\quad&&\text{when }0\notin\Omega,\\
    \varphi(z)&=\dfrac{w_0}{|z_0|}b_{z_0}(z)e^{r^{\#}(z)},\quad&&\text{when }0\in\Omega,
\end{alignedat}
\end{equation}
where the constant out front is chosen for the normalization $\varphi(0)=w_0\neq0$, and $r\in\A_0(\D\IntComp)$. When $\varphi(0)=0$, we instead take $\varphi(z)=cze^{r^{\#}(z)}$, where $\varphi'(0)=c>0$, to the same effect. On the other hand, when $\Omega$ is unbounded, we write
\begin{equation}\label{eqn:LQDUnboundedInnerOuterFactorization}
\begin{alignedat}{2}
    \varphi(z)&=cze^{r^{\#}(z)},\quad&&\text{when }0\notin\Omega,\\
    \varphi(z)&=c|z_0|zb_{z_0}(z)e^{r^{\#}(z)},\quad&&\text{when }0\in\Omega,
\end{alignedat}
\end{equation}
where the constant out front is chosen for the normalization $\varphi'(\infty)=c>0$ (the conformal radius of $\Omega$), and $r\in\A(\D)$, with $r(0)=0$.

\subsection{Faber Transform Formulae for Log-Weighted Quadrature Domains}
Under the normalizations (\ref{eqn:LQDBoundedInnerOuterFactorization}, \ref{eqn:LQDUnboundedInnerOuterFactorization}), Theorem 4.1 reduces the simply connected inverse and direct problems to a single question: how is the rational function $r$ related to the quadrature function $h$? The next theorem answers this explicitly via the Faber transform.

\begin{theorem}\label{thm:LQDFTInverseProb}
Let $\Omega\in\QD_0(h)$ be a simply connected domain with Riemann map $\varphi$. 
\begin{enumerate}
    \item If $\Omega$ is bounded, then $\varphi$ admits the factorization \ref{eqn:LQDBoundedInnerOuterFactorization}, where $r$ is given by
    \begin{equation}\label{eqn:LQDBoundedFTInverseProbFormula}
        r(z)=\Phi_{\varphi}^{-1}\left(wh(w)+\Res{\infty}h\right)(z),
    \end{equation}
    \item If $\Omega$ is unbounded, then $\varphi$ admits the factorization \ref{eqn:LQDUnboundedInnerOuterFactorization}, where $r$ is given by
    \begin{equation}\label{eqn:LQDUnboundedFTInverseProbFormula}
        r(z)=\Phi_{\varphi}^{-1}\left(wh(w)\right)(z)-\Phi_{\varphi}^{-1}\left(wh(w)\right)(0).
    \end{equation}
\end{enumerate}
\end{theorem}
\noindent Inverting our formulae for $r$, we obtain a formula for the quadrature function $h$ in terms of $r$:
\begin{equation}\label{eqn:LQDhFormula}
    h(w)=\dfrac{\Phi_{\varphi}\left(r\right)(w)-C}{w}
\end{equation}
for some $C\in\C$. Note that if $0\notin\Omega$, then $C$ is determined by the requirement that $h$ has no pole at $0$. On the other hand if $0\in\Omega$ then $C$ is arbitrary (Corollary \ref{cor:LQDQuadFuncUniquenessZero}).\\

Theorem \ref{thm:LQDFTInverseProb} gives the explicit relationship between the Riemann map (via $r$) and the quadrature function $h$. In particular, once $h$ is known, the formulae above recover $r$, and hence $\varphi$.
\begin{proof}[Proof of Theorem \ref{thm:LQDFTInverseProb}]\;\\
{\bf Bounded case:} Suppose $\Omega\in\QD_0(h)$ is bounded with Riemann map $\varphi:\D\rightarrow\Omega$. By Theorem \ref{thm:RiemannMapInnerOuterRepresentation} and the proof of Theorem \ref{thm:SCLQDCharacterization}, there exists $C\in\C$, and a rational function $r$ satisfying
$$r=\Phi_{\varphi}^{-1}\left(wh(w)+\Res{\infty}h\right)$$
such that
$$\varphi(z)=C\varphi_{\rm in}(z)e^{r^{\#}(z)},$$
where $\varphi_{\rm in}$ is either $1$, a Blaschke factor $b_{z_0}$, or $z$, when $0\notin\Omega$ $\varphi(0)\neq0$, and $\varphi(0)=0$ respectively. Since
$$\Phi_{\varphi}^{-1}:\A_0(\Omega\IntComp)\rightarrow\A_0(\D\IntComp)$$
in the bounded case, we have that $r\in\A_0(\D\IntComp)$, hence $r^{\#}\in\A(\D)$ and $r^{\#}(0)=0$. Hence the constant $C$ is determined by the normalization at $z=0$: If $\varphi(0)=w_0\neq0$, then
$$C=w_0,\;\;\;\text{ or }\;\;\; C=\dfrac{w_0}{b_{z_0}(0)}=\dfrac{w_0}{|z_0|},$$
when $0\notin\Omega$ and $0\in\Omega$ respectively. If $\varphi(0)=0$, then the normalization implies
$$\varphi(z)=cze^{r^{\#}(z)},$$
where $c=\varphi'(0)>0$. Hence $\varphi$ takes the normalized form given in Equation \ref{eqn:LQDBoundedInnerOuterFactorization}.\\

{\bf Unbounded case:} Suppose $\Omega\in\QD_0(h)$ is unbounded with Riemann map $\varphi:\D\IntComp\rightarrow\Omega$ normalized by $\varphi(\infty)=\infty$ and $\varphi'(\infty)=c>0$. By Theorem \ref{thm:RiemannMapInnerOuterRepresentation} and the proof of Theorem \ref{thm:SCLQDCharacterization}, there exists $C\in\C$, and a rational function $\widetilde{r}$ such that
$$\widetilde{r}=\Phi_{\varphi}^{-1}\left(wh(w)\right)+q'$$
for some $q'\in\C$, and
$$\varphi(z)=C\varphi_{\rm in}(z)e^{\widetilde{r}^{\#}(z)},$$
where $\varphi_{\rm in}$ is either $z$ or $zb_{z_0}(z)$, when $0\notin\Omega$ and $0\in\Omega$ respectively. Since
$$\Phi_{\varphi}^{-1}:\A(\Omega\IntComp)\rightarrow\A(\D)$$
in the unbounded case, we have that $\widetilde{r}\in\A(\D)$. Subtracting the constant $\widetilde{r}(0)$, we obtain a rational function 
$$r(z)=\widetilde{r}(z)-\widetilde{r}(0)=\Phi_{\varphi}^{-1}\left(wh(w)\right)(z)-\Phi_{\varphi}^{-1}\left(wh(w)\right)(0),$$
so that $r(0)=0$. Absorbing the corresponding multiplicative constant in $C$ and imposing the normalization $\varphi'(\infty)=c>0$, we obtain the precise forms of Equation \ref{eqn:LQDUnboundedInnerOuterFactorization}.
\end{proof}

Theorem \ref{thm:LQDFTInverseProb} and Equations \ref{eqn:LQDBoundedFTInverseProbFormula}, \ref{eqn:LQDUnboundedFTInverseProbFormula}, and \ref{eqn:LQDhFormula} will be used extensively throughout the remainder of the manuscript. In particular, we will apply these formulae to address the inverse and direct problems for several different classes of simply connected LQDs.

\subsection{Example: Constant Log-Weighted Quadrature Domains}\label{subsec:ConstantLQDExample}
In this section, we provide a classification of simply connected ``constant'' LQDs: $\Omega\in\QD_0(\alpha)$, $\alpha\in\C$. Note that $\Omega$ must be unbounded because $\infty$ is a quadrature node, so there are two cases to consider: (1) $0\notin\Omega$ and (2) $0\in\Omega$. Let $\varphi:\D\IntComp\rightarrow\Omega$ be the associated Riemann map. Graven and Makarov \cite{GravenMakarov2025} demonstrated that if $0\notin\Omega$, then $\varphi(z)=cze^{\overline{\alpha}cz^{-1}}$, where $0<c<|\alpha|^{-1}$ is the conformal radius of $\Omega$. Thus we will restrict our attention to the case in which $0\in\Omega$. 

\begin{theorem}
Fix $\alpha\neq0$ and a simply connected domain $\Omega$ containing zero. Then $\Omega\in\QD_0(\alpha)$ if and only if it admits a Riemann map of the form
\begin{equation}\label{eqn:ZConstantLQDRiemannMap}
    \varphi(z)=c|z_0|zb_{z_0}(z)e^{\overline{\alpha}cz^{-1}}.
\end{equation}
\end{theorem}

\begin{proof}
By Equation \ref{eqn:LQDUnboundedInnerOuterFactorization}, we know that $\varphi(z)=c|z_0|zb_{z_0}(z)e^{r^{\#}(z)}$, where $z_0$ is the unique root of $\varphi$. Moreover, by Equation \ref{eqn:LQDUnboundedFTInverseProbFormula},
$r(z)=\Phi_{\varphi}^{-1}\left(w\alpha\right)(z)-\Phi_{\varphi}^{-1}\left(w\alpha\right)(0)$. By the linearity of the Faber transform and Equation \ref{eqn:FaberPolyFormulae}, we obtain
$$r(z)=\alpha W_1(z)-\alpha W_1(0)=\alpha cz+\alpha f_0-\alpha f_0,$$
where $W_j(z)=\Phi_{\varphi}^{-1}(w^j)$ is the $j$th inverse Faber polynomial (\S\ref{subsec:FaberTransform}). Hence,
$$\varphi(z)=c|z_0|zb_{z_0}(z)e^{\overline{\alpha}cz^{-1}}.$$

In particular we have shown that if $0\in\Omega\in\QD_0(\alpha)$ is simply connected, then it admits a Riemann map of the above form (Equation \ref{eqn:ZConstantLQDRiemannMap}). On the other hand, by Theorem \ref{thm:SCLQDCharacterization}, we know that the image $\Omega$ of the above $\varphi$ is in $\QD_0(h)$ for some $h$. Hence, as $\Omega$ is unbounded, Equation \ref{eqn:LQDhFormula} tells us there exists $C\in\C$ such that
\begin{align*}
    h(w)&=\dfrac{\Phi_{\varphi}\left(\alpha cz\right)(w)-C}{w}=\dfrac{\alpha c\left(\frac{w}{c}-\frac{f_0}{c}\right)-C}{w}=\alpha-\dfrac{\alpha f_0+C}{w}
\end{align*}
Thus $\Omega\in\QD_0\left(\alpha+\frac{q}{w}\right)$ for some $q\in\C$. We conclude by Corollary \ref{cor:LQDQuadFuncUniquenessZero} that $\Omega\in\QD_0(\alpha)$.
\end{proof}

\noindent This is only a special case of the class of \emph{monomial} LQDs, which is addressed in the next section.

\section{Monomial Log-Weighted Quadrature Domains}\label{sec:MonomialLQDs}
In this section, we consider simply connected LQDs with quadrature function $h(w)=\alpha k w^{k-1}$ for $k\in\Z_{+}$ and $\alpha\in\C\setminus\{0\}$. Note that all such domains are unbounded, as $\infty$ is a quadrature node.  When $\Omega$ is non-singular, we are able to construct simple solutions by exploiting the symmetry of the quadrature function. As we shall see, no such symmetry is possible when $\Omega$ is singular, resulting in a much larger space of possible solutions. Hence, we only construct explicit solutions in the singular case for $k=2$ (The $k=1$ case was covered in \S\ref{subsec:ConstantLQDExample}).

\subsection{Symmetric Non-Singular Monomial LQDs}
Suppose that $\Omega\in\QD_0(\alpha k w^{k-1})$ is a simply connected domain not containing zero. As $\Omega$ is unbounded, Theorem \ref{thm:LQDFTInverseProb} tells us that the associated Riemann map $\varphi:\D\IntComp\rightarrow\Omega$ is of the form $\varphi(z)=cze^{r^{\#}(z)}$, where $c>0$ is the conformal radius of $\Omega$ and
\begin{align*}
    r(z)&=\Phi_{\varphi}^{-1}\left(\alpha k w^{k}\right)(z)-\Phi_{\varphi}^{-1}\left(\alpha k w^{k}\right)(0).
\end{align*}
We recognize $\Phi_{\varphi}^{-1}\left(\alpha k w^{k}\right)(z)=\alpha k \Phi_{\varphi}^{-1}\left(w^{k}\right)(z)=\alpha k W_k(z)$, where $W_k$ is the $k$th inverse Faber polynomial (\S\ref{subsec:FaberTransform}). Hence, $r(z)=\alpha k\left(W_k(z)-W_k(0)\right)$, and we find
$$\varphi(z)=cze^{\overline{\alpha}k\left(W_k^{\#}(z)-\overline{W_k(0)}\right)}.$$

We now consider the implications of the symmetry of the quadrature function. By Theorem \ref{thm:ScalingLaw}, for each $j\in\{1,\hdots,k\}$, we have $\Omega\in\QD_0(\alpha k w^{k-1})$ if and only if $e^{-2\pi i\frac{j}{k}}\Omega\in\QD_0(\alpha k w^{k-1})$. That is, $\Omega\in\QD_0(\alpha kw^{k-1})$ implies each rotation of $\Omega$ by a multiple of $\frac{2\pi}{k}$ radians is an LQD with the same quadrature function. Hence, if $\Omega\in\QD_0(\alpha kw^{k-1})$ is the unique such domain of its conformal radius, then it has $k-$fold rotational symmetry about the origin. It follows via inspection of the Laurent expansion that if a simply connected domain has $k-$fold rotational symmetry about the origin, then so does its Riemann map: $\varphi\left(ze^{\frac{2\pi i}{k}}\right)=e^{\frac{2\pi i}{k}}\varphi\left(z\right)$. Hence, $W_k^{\#}\left(ze^{\frac{2\pi i}{k}}\right)-\overline{W_k(0)}=e^{\frac{2\pi i}{k}}\left(W_k^{\#}(z)-\overline{W_k(0)}\right)$. Iterating this identity, we find that each of the coefficients of $W_k(z)$ must equal zero except for the term of order $k$, $c^kz^k$. That is,
\begin{equation}\label{eqn:NonSingularMonomialLQD}
    \varphi(z)=cze^{\overline{\alpha}kc^kz^{-k}}
\end{equation}
In summary,
\begin{theorem}\label{thm:NonSingularMonomialLQD}
    If $\Omega\in\QD_0(\alpha k w^{k-1})$ is a simply connected domain not containing zero which is $k-$fold rotationally symmetric about the origin, then its associated Riemann map is given by Equation \ref{eqn:NonSingularMonomialLQD}.
\end{theorem}
Furthermore, note that this family is univalent precisely when $0<c<|\alpha k^2|^{-\frac{1}{k}}$. A similar analysis is carried out in Graven \& Makarov \cite{GravenMakarov2025}.

\subsection{Singular Monomial LQDs}
Now suppose that $\Omega\in\QD_0(\alpha k w^{k-1})$ is a simply connected domain containing zero. As $\Omega$ is unbounded, Theorem \ref{thm:LQDFTInverseProb} tells us that the associated Riemann map $\varphi:\D\IntComp\rightarrow\Omega$ is of the form $\varphi(z)=c|z_0|zb_{z_0}(z)e^{r^{\#}(z)}$, where $c>0$ is the conformal radius of $\Omega$, $z_0\in\D\IntComp$ is the unique root of $\varphi$, and $r$ is given by Equation \ref{eqn:LQDUnboundedFTInverseProbFormula}. By the same argument as in the non-singular case, we find that
\begin{equation}\label{eqn:SingularMonomialLQD}
    \varphi(z)=c|z_0|zb_{z_0}(z)e^{\overline{\alpha}k\left(W_k^{\#}(z)-\overline{W_k(0)}\right)}.
\end{equation}
We note that $\Omega$ cannot be $k-$fold rotationally symmetric: If $\Omega$ were, then it must either be multiply connected, or equal the whole plane, because it contains an open neighborhood of both zero and infinity. As we assumed that $\Omega$ is connected, we would conclude that $\Omega=\Ch$, which is not conformally equivalent to $\D\IntComp$.

It remains to determine a relation involving the charge $q$ at the origin under the specification $\Omega\in\QD_0\left(\alpha k w^{k-1};q\right)$. We begin by computing the generalized Schwarz function. Note that $\frac{\ln|w|^2}{w}\dEquals \frac{\ln(\varphi\varphi^{\#})\circ\psi(w)}{w}$, so
\begin{align*}
    S_0(w)&:=\dfrac{\ln(\varphi\varphi^{\#})\circ\psi(w)}{w}=\dfrac{\ln|cz_0|^2+\overline{\alpha}k\left(W_k^{\#}\circ\psi(w)-\overline{W_k(0)}\right)+\alpha k\left(W_k\circ\psi(w)-W_k(0)\right)}{w}
\end{align*}
Applying the fact that $W_k\circ\psi(w)=w^k+O(1)$, we find that
\begin{align*}
S_0(w)&=\alpha kw^{k-1}+\dfrac{\ln|cz_0|^2-\alpha kW_k(0)}{w}+O(w^{-1}).
\end{align*}
Hence, $q=\ln|cz_0|^2-\alpha kW_k(0)$.\\

\subsubsection{Example: The case of k=2}
We consider the $k=2$ case as a basic example. As $W_2(z)$ is a polynomial of degree $2$ with leading term $c^2z^2$, we find that
$$\varphi(z)=c|z_0|zb_{z_0}(z)e^{2\overline{\alpha}\left(c^2z^{-2}+\beta z^{-1}\right)},$$
for some $\beta\in\C$. Applying Equation \ref{eqn:LQDhFormula} yields
\begin{align*}
    2\alpha w&=\dfrac{\Phi_{\varphi}\left(2\alpha(c^2z^2+\overline{\beta} z)\right)(w)-C}{w}\\
    &=\dfrac{2\alpha c^2}{w}F_2(w)+\dfrac{2\alpha \overline{\beta}}{w}F_1(w)-\dfrac{C}{w},
\end{align*}
where $F_1$ and $F_2$ are the first and second Faber polynomials associated to $\varphi$. Computing these, and substituting, we obtain
\begin{align*}
    2\alpha w&=2\alpha w+\dfrac{2\alpha}{c\overline{z_0}}\left(\overline{\beta}\overline{z_0}-4 c^2\overline{\alpha}\beta\overline{z_0}+2c^2(|z_0|^2-1)\right)+O(1)
\end{align*}
As the second term must $=0$, we find that
$$\beta=\dfrac{2 c^2 \left(|z_0|^2-1\right) \left(4 c^2 \alpha z_0+\overline{z_0}\right)}{|z_0|^2\left(16 c^4|\alpha|^2-1\right)}.$$
Figure \ref{fig:LQDMonomialFamSingular} shows two different families of solutions obtained using this method.

\begin{figure}[ht]
  \centering
\includegraphics[height=0.25\linewidth,width=.55\linewidth]{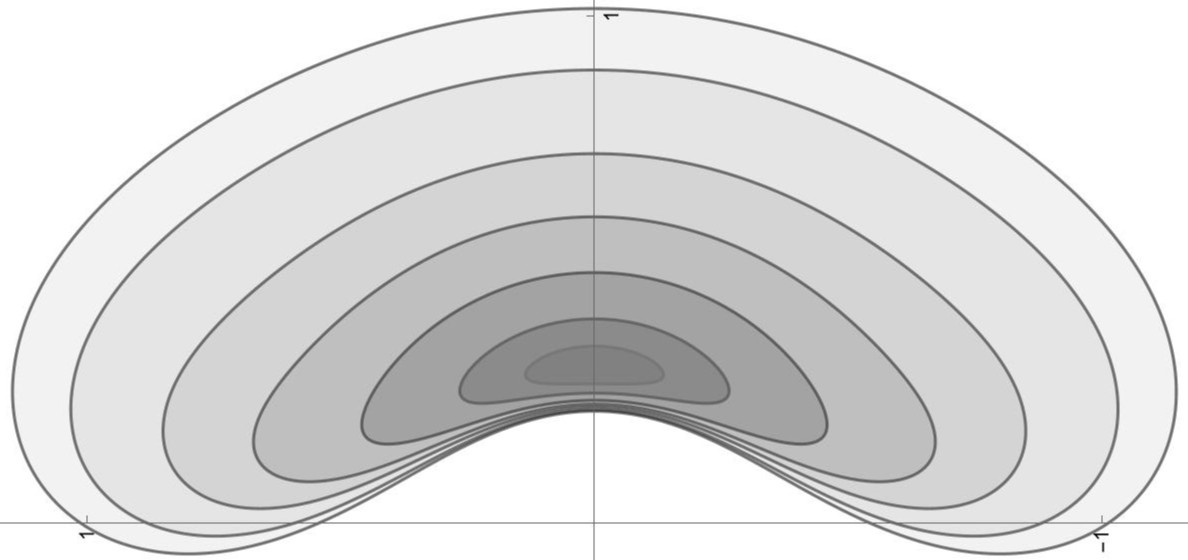}\;\;\includegraphics[height=0.25\linewidth,width=.45\linewidth]{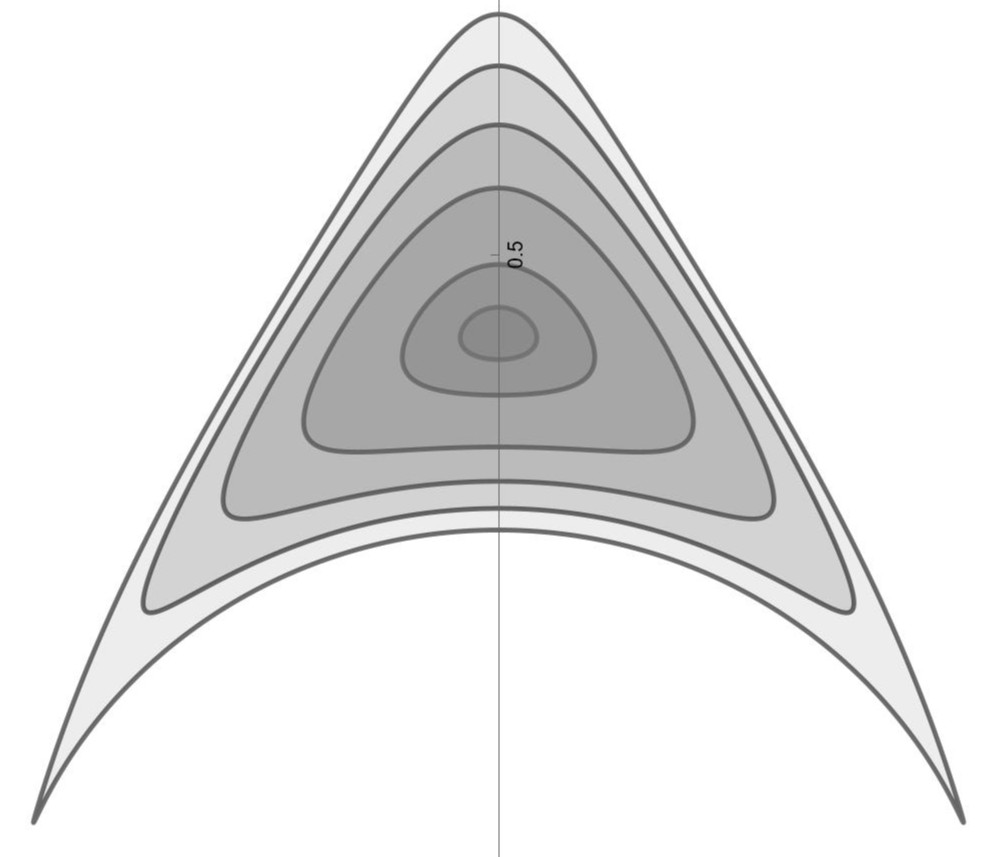}
\vspace{.5em}
\caption{Family of complements of singular monomial LQDs in $\QD_0(2\alpha w)$. For $\alpha=-2$, $q=-2$, $0<c\leq.3$ (left); and $\alpha=1$, $q=-2$, $0<c<.25$ (right). The plot is rotated so the positive real axis points upwards.}\label{fig:LQDMonomialFamSingular}\vspace{1.5em}
\end{figure}

\section{One-Point Log-Weighted Quadrature Domains}\label{sec:OnePointLQDs}

Let $\Omega$ be a simply connected LQD with quadrature function $\frac{\alpha}{w-w_0}$ for some $\alpha\in\C\setminus\{0\}$ and $w_0\in\C$; that is, $\Omega\in\QD_0\left(\frac{\alpha}{w-w_0}\right)$. If $w_0=0$ then $\Omega$ is a disk centered at the origin. This follows from Corollary \ref{cor:LQDQuadFuncUniquenessZero}, which tells us that $\QD_0\left(\frac{\alpha}{w}\right)=\left\{\Omega\in\QD_0(0):0\in\Omega\right\}$, and Theorem \ref{thm:NullLQDClass}, which tells us that the only null quadrature domains containing $0$ are disks centered at the origin. Hence, we will assume that $w_0\neq0$. Furthermore, by Theorem \ref{thm:ScalingLaw}, $\Omega\in\left(\frac{\alpha}{w-w_0}\right)$ if and only if $w_0^{-1}\Omega\in\left(\frac{\alpha}{w-1}\right)$. Hence, we may wlog assume that $w_0=1$.

The behavior of a given family of solutions depends crucially on the singularities of $\rho_0$ therein. If $0\in\Omega$, then the choice of charge $q\in\C$ at $0$ serves as a free parameter. Moreover, when $\infty\in\Omega$, the conformal radius $c$ of the domain serves as an additional free parameter. That is, for a given family of solutions, {\it there are as many free parameters as there are singularities} of $\rho_0$ in $\Omega$. Thus it is natural to break up the classification of simply connected $\Omega\in\QD_0\left(\frac{\alpha}{w-w_0}\right)$ into four cases:
\begin{enumerate}
    \item Non-Singular ($0\notin\Omega$):
    \begin{enumerate}
        \item $\Omega$ bounded.
        \item $\Omega$ unbounded.
    \end{enumerate}
    \item Singular ($0\in\Omega$):
    \begin{enumerate}
        \item $\Omega$ bounded.
        \item $\Omega$ unbounded.
    \end{enumerate}
\end{enumerate}

\subsection{Non-Singular One-Point LQDs}\label{subsec:LOGWQDOnePtNZ}
Let us begin by considering the more tractable case of non-singular one-point LQDs. In particular, we consider simply connected domains $\Omega\in\QD_0\left(\frac{\alpha}{w-w_0}\right)$ for $\alpha,w_0\in\C\setminus\{0\}$ not containing zero. The bounded case was addressed in \S\ref{subsec:BasicOnePTLQDEx}, so we shall restrict our consideration to unbounded $\Omega$, the results of which are summarized in Theorem \ref{thm:UnboundedNonSingularOnePtLQD}.

\begin{theorem}\label{thm:UnboundedNonSingularOnePtLQD}
Take $w_0,\alpha\in\C\setminus\{0\}$. There exists an unbounded simply connected domain $\Omega$ not containing zero for which $\Omega\in\QD_0\left(\frac{\alpha}{w-w_0}\right)$ if and only if there exist $c>0$, $\lambda\in\C$, and $z_1\in\D\IntComp$ for which $\Omega=\varphi(\D\IntComp)$, where $\varphi$ is univalent and given by
\begin{equation}\label{eqn:UnboundedNonSingularOnePtLQD}
    \varphi(z)=cze^{\frac{\lambda}{1-z\overline{z_1}}},
\end{equation}
where $\varphi(z_1)=w_0$, and the coefficients satisfy the relation $\lambda=\frac{\overline{\alpha}\overline{w_0}}{\overline{z_1}\overline{\varphi'(z_1)}}$.
\end{theorem}
\begin{proof}[Proof of Theorem \ref{thm:UnboundedNonSingularOnePtLQD}]\label{proof:UnboundedNonSingularOnePtLQD}
For the forward direction, suppose that $0\notin\Omega\in\QD_0\left(\frac{\alpha}{w-1}\right)$ with $\alpha\in\C\setminus\{0\}$ is unbounded and simply connected with Riemann map $\varphi:\D\IntComp\rightarrow\Omega$, with $\varphi(z_1)=1$, and $\varphi'(\infty)=c>0$. By Theorem \ref{thm:LQDFTInverseProb}, $\varphi(z)=cze^{r^{\#}(z)}$, where $r(z)=\Phi_{\varphi}^{-1}\left(\frac{\alpha w}{w-1}\right)(z)-\Phi_{\varphi}^{-1}\left(\frac{\alpha w}{w-1}\right)(0)$. Computing the Faber transform via Equation \ref{eqn:FaberPolyFormulae}, we find
\begin{align*}
    \Phi_{\varphi}^{-1}\left(\dfrac{\alpha w}{w-1}\right)(z)&=\alpha+\alpha\Phi_{\varphi}^{-1}\left(\dfrac{1}{w-1}\right)(z)=\alpha\left(1+\dfrac{\psi'(1)}{z-\psi(1)}\right)
\end{align*}
so $r(z)=\frac{\alpha}{z_1\varphi'(z_1)}\frac{z}{z-z_1}$, and we recover Equation \ref{eqn:UnboundedNonSingularOnePtLQD} by setting $\lambda=\frac{\overline{\alpha}}{\overline{z_1}\overline{\varphi'(z_1)}}$. Also, $z_1\in\D\IntComp$ because otherwise $\varphi$ wouldn't be analytic.

For the reverse direction, suppose that there exist $c>0$, $\lambda\in\C$, and $z_1\in\D\IntComp$ for which $\Omega=\varphi(\D\IntComp)$, where $\varphi$ is univalent and given by Equation \ref{eqn:UnboundedNonSingularOnePtLQD}. Then, by Theorem \ref{thm:SCLQDCharacterization}, $\Omega\in\QD_0(h)$ for some $h$. We will now compute $h$ using Theorem \ref{thm:LQDFTInverseProb}, which tells us that there exists $C\in\C$ such that
\begin{align*}
h(w)&=\dfrac{\Phi_{\varphi}\left(\frac{\overline{\lambda}z}{z-z_1}\right)(w)-C}{w}=\dfrac{z_1\overline{\lambda}}{w}\Phi_{\varphi}\left(\frac{1}{z-z_1}\right)(w)+\dfrac{\overline{\lambda}-C}{w}.
\end{align*}
Computing the Faber transform via Equation \ref{eqn:FaberPolyFormulae} and substituting the formula for $\lambda$, we find
\begin{align*}
h(w)&=\dfrac{z_1\overline{\lambda}}{w}\dfrac{\varphi'(z_1)}{w-\varphi(z_1)}+\dfrac{\overline{\lambda}-C}{w}=\dfrac{\alpha}{w-1}+\dfrac{\overline{\lambda}-C-\alpha}{w}.
\end{align*}
Finally, as $h\in\A(\Omega\IntComp)$ and $0\in\Omega\IntComp$ ($\varphi(z)=0$ has no solutions in $\D\IntComp$, so $0\notin\Omega$), we find that $C=\overline{\lambda}-\alpha$, so $\Omega\in\QD_0\left(\frac{\alpha}{w-1}\right)$.
\end{proof}

\subsection{Singular One-Point LQDs}\label{subsec:LOGWQDOnePtZ}
We now turn our consideration to the more exotic case of singular one-point LQDs. In particular, we consider simply connected domains $\Omega\in\QD_0\left(\frac{\alpha}{w-w_0}\right)$ which contain zero (recall that we can wlog take $w_0=1$).

\subsubsection{Singular Bounded One-Point LQDs}\label{subsubsec:LOGWQDBoundedOnePtZ}

\begin{theorem}\label{thm:BoundedSingularOnePtLQD}
Take $w_0,\alpha\in\C\setminus\{0\}$. There exists a bounded simply connected domain $\Omega$ containing zero for which $\Omega\in\QD_0\left(\frac{\alpha}{w-w_0}\right)$ if and only if there are $\lambda\in\C$ and $z_0\in\D$ for which $\Omega=\varphi(\D)$, where $\varphi$ is univalent and given by
\begin{equation}\label{eqn:BoundedSingularOnePtLQD}
    \varphi(z)=\dfrac{w_0}{|z_0|}b_{z_0}(z)e^{\lambda z},
\end{equation}
where $\lambda=\frac{\overline{w_0}\overline{\alpha}}{\varphi'(0)}$. Moreover, if we specify that $\Omega\in\QD_0\left(\frac{\alpha}{w-w_0};q\right)$, then $\lambda=\overline{w_0}\overline{\alpha}\frac{q+\ln|z_0|^2}{\overline{\alpha}z_0+\alpha z_0^{-1}}$.
\end{theorem}
Here, $\varphi$ is normalized so that $\varphi(0)=w_0$ and $\varphi'(0)>0$.

\begin{proof}[Proof of Theorem \ref{thm:BoundedSingularOnePtLQD}]\label{proof:BoundedSingularOnePtLQD}
For the forward direction, suppose that $0\in\Omega\in\QD_0\left(\frac{\alpha}{w-1}\right)$ with $\alpha\in\C\setminus\{0\}$ is bounded and simply connected with Riemann map $\varphi:\D\rightarrow\Omega$, normalized such that $\varphi(z_0)=0$, $\varphi(0)=1$, and $\varphi'(0)>0$. Note that, unlike in the non-singular case (\S\ref{subsec:BasicOnePTLQDEx}), we cannot assume $\alpha>0$ because $1\notin L^1_a(\Omega;\rho_0)$ when $0\in\Omega$.

By Theorem \ref{thm:LQDFTInverseProb}, $\varphi(z)=\frac{1}{|z_0|}b_{z_0}(z)e^{r^{\#}(z)}$, where $r(z)=\Phi_{\varphi}^{-1}\left(\frac{\alpha w}{w-1}-\alpha\right)(z)=\alpha\Phi_{\varphi}^{-1}\left(\frac{1}{w-1}\right)(z)$. Computing the Faber transform using Equation \ref{eqn:FaberPolyFormulae}, we obtain $r(z)=\alpha\frac{\psi'(1)}{z-\psi(1)}=z^{-1}\frac{\alpha}{\varphi'(0)}$. Hence, setting $\lambda=\frac{\overline{\alpha}}{\overline{\varphi'(0)}}$, we obtain
$$\varphi(z)=\frac{1}{|z_0|}b_{z_0}(z)e^{\lambda z}.$$

For the reverse direction, suppose there exist $\lambda\in\C$ and $z_0\in\D$ for which $\Omega=\varphi(\D)$, where $\varphi$ is univalent and given by Equation \ref{eqn:BoundedSingularOnePtLQD} (with $w_0=1$). Then, by Theorem \ref{thm:SCLQDCharacterization}, $\Omega\in\QD_0(h)$ for some $h$. We will now compute $h$ using Theorem \ref{thm:LQDFTInverseProb}, which tells us that there exists $C\in\C$ such that
\begin{align*}
h(w)&=\dfrac{\Phi_{\varphi}\left(\overline{\lambda}z^{-1}\right)(w)-C}{w}=\dfrac{\overline{\lambda}}{w}\Phi_{\varphi}\left(z^{-1}\right)(w)-\dfrac{C}{w}
\end{align*}
Computing the Faber transform via Equation \ref{eqn:FaberPolyFormulae} and substituting the formula for $\lambda$, we find
\begin{align*}
h(w)&=\dfrac{\overline{\lambda}}{w}\dfrac{\varphi'(0)}{w-\varphi(0)}-\dfrac{C}{w}=\dfrac{\alpha}{w-1}+\dfrac{\alpha-C}{w}.
\end{align*}
Hence $\Omega\in\QD_0\left(\frac{\alpha}{w-1}\right)$ by Corollary \ref{cor:LQDQuadFuncUniquenessZero}.

It remains to show the formula for $\lambda$ in terms of $q$ under the specification $\Omega\in\QD_0\left(\frac{\alpha}{w-1};q\right)$. We begin by computing the generalized Schwarz function. Note that $\frac{\ln|w|^2}{w}\dEquals \frac{\ln(\varphi\varphi^{\#})\circ\psi(w)}{w}$, so
\begin{align*}
    S_0(w)&:=\dfrac{\ln(\varphi\varphi^{\#})\circ\psi(w)}{w}=-\dfrac{\ln|z_0|^2}{w}+\frac{1}{\varphi'(0)}\dfrac{\overline{\alpha}\psi(w)+\alpha\psi(w)^{-1}}{w}
\end{align*}
is a generalized Schwarz function for $\Omega$. Noting the singularities at $w=0$ and $w=1$, computing their residues, and applying Mittag-Leffler's theorem, we find that there exists $G\in\A(\Omega)$ such that
\begin{align*}
    S_0(w)&=\dfrac{\frac{\alpha}{\varphi'(0)^2z_0}(z_0\overline{\alpha}+\varphi'(0)(|z_0|^2-1))}{w-1}+\dfrac{\frac{\overline{\alpha}z_0+\alpha z_0^{-1}}{\varphi'(0)}-\ln|z_0|^2}{w}+G(w).
\end{align*}
Hence, $q=\frac{\overline{\alpha}z_0+\alpha z_0^{-1}}{\varphi'(0)}-\ln|z_0|^2$. We conclude that $\lambda=\overline{\alpha}\frac{q+\ln|z_0|^2}{\overline{\alpha}z_0+\alpha z_0^{-1}}$.

We then obtain the formulae for arbitrary $w_0$ by replacing instances of $\varphi$ and $\varphi'$ with $w_0^{-1}\varphi$ and $w_0^{-1}\varphi'$ respectively.
\end{proof}
Hence, determining the region $\Omega\in\QD_0\left(\frac{\alpha}{w-1};q\right)$ for each value of $q$ has been reduced to determining $z_0$. This can be obtained directly via the relation $\varphi'(0)=\frac{\overline{\alpha}}{\lambda}$, which yields after simplification $(1-|z_0|^2)\lambda=(|\lambda|^2-\overline{\alpha})z_0$.

\paragraph{Univalence criteria}
We would like to characterize the set of values of $z_0$ and $\lambda$ for which $\varphi(z)=\frac{1}{|z_0|}b_{z_0}(z)e^{\lambda z}$ is univalent in $\D$. Notice that if we set $f(z):=\varphi\left(\frac{|\lambda|}{\lambda}z\right)=\frac{1}{|z_1|}b_{z_1}(z)e^{|\lambda|z}$, where $z_1=\frac{\lambda}{|\lambda|}z_0$, then $\varphi$ is univalent in $\D$ if and only if $f$ is univalent in $\D$. In particular, this shows that it is sufficient to consider $\varphi$ for $\lambda>0$.

Note that $|z_1|<1$ is necessary because otherwise $f$ is not analytic in $\D$. This is also a sufficient condition for $f$ to be analytic in $\D$. Hence, it is necessary and sufficient that $|z_1|<1$ and $f(\partial\D)$ has no self-intersections. Suppose there existed $z\neq w\in\partial\D$ such that $f(z)=f(w)$. Taking the modulus of both sides and recalling that $|b_{z_1}(z)|=1$, we find $e^{|\lambda|\Re(z)}=e^{|\lambda|\Re(w)}$, which implies $\Re(z)=\Re(w)$. It follows that $w=z^{-1}$, so there exist $z\neq w\in\partial\D$ such that $f(z)=f(w)$ if and only if there exists $z\in\partial\D\cap\{z\in\C:\Im(z)>0\}$ such that $f(z)=f(z^{-1})$. Rearranging, we find $f$ is univalent in $\D$ if and only if the following equation has no solutions $z\in\partial\D\cap\{z\in\C:\Im(z)>0\}$.
\begin{align*}
    e^{2i|\lambda|\Im(z)}&=\dfrac{(zz_1-1)(z\overline{z_1}-1)}{(z-\overline{z_1})(z-z_1)}.
\end{align*}
Note that $\frac{(zz_1-1)(z\overline{z_1}-1)}{(z-\overline{z_1})(z-z_1)}=\frac{\left(\overline{z^{-1}(z-\overline{z_1})(z-z_1)}\right)}{z^{-1}(z-\overline{z_1})(z-z_1)}=e^{-2i\Arg(z^{-1}(z-z_1)(z-\overline{z_1}))}$. Hence, setting $z=e^{i\theta}$ ($0<\theta<\pi$), and $z_1=\frac{\lambda}{|\lambda|}z_0=|z_0|e^{i(\Arg(z_0)+\Arg(\lambda))}$, we find $e^{2i(|\lambda|\sin(\theta)+\delta(\theta))}=1$, where
\begin{equation}\label{eqn:BoundedSingularOnePtLQDUnivalenceDelta}
\delta(\theta)=\Arg\left((1+|z_0|^2)\cos(\theta)-2|z_0|\cos(\Arg(z_0)+\Arg(\lambda))+i(1-|z_0|^2)\sin(\theta)\right)
\end{equation}
As $\delta(\theta)\in(0,\pi)$ and $|\lambda|\sin(\theta)\in(0,|\lambda|)$ when $0<\theta<\pi$, we find that if there is a solution to $e^{2i(|\lambda|\sin(\theta)+\delta(\theta))}=1$, there must be a solution such that $|\lambda|\sin(\theta)+\delta(\theta)=\pi$. Hence, if we set
\begin{equation}\label{eqn:BoundedSingularOnePtLQDUnivalenceMinimization}
\lambda_{\max}=\min_{0<\theta<\pi}\dfrac{\pi-\delta(\theta)}{\sin(\theta)},
\end{equation}
then $f$ is univalent in $\D$ if and only if $0<|\lambda|\leq\lambda_{\max}$. In particular, we have shown that

\begin{corollary}
$\varphi:\D\rightarrow\C$ from Equation \ref{eqn:BoundedSingularOnePtLQD} is univalent if and only if $|z_0|<1$, and $0<|\lambda|\leq\lambda_{\max}$.
\end{corollary}

\subsubsection{Singular Unbounded One-Point LQDs}\label{subsubsec:LOGWQDUnboundedOnePtZ}

\begin{theorem}\label{thm:UnboundedSingularOnePtLQD}
Take $w_0,\alpha\in\C\setminus\{0\}$. There exists an unbounded simply connected domain $\Omega$ containing zero for which $\Omega\in\QD_0\left(\frac{\alpha}{w-w_0}\right)$ if and only if there exist $\lambda\in\C$ and $z_0,z_1\in\D\IntComp$, for which $\Omega=\varphi(\D\IntComp)$, where $\varphi$ is univalent and given by
\begin{equation}\label{eqn:UnboundedSingularOnePtLQD}
    \varphi(z)=c|z_0|zb_{z_0}(z)e^{\frac{\lambda}{1-z\overline{z_1}}},
\end{equation}
where $\varphi(z_1)=w_0$ and $\lambda=\frac{\overline{\alpha}\overline{w_0}}{\overline{z_1}\overline{\varphi'(z_1)}}$. Moreover, if $\Omega\in\QD_0\left(\frac{\alpha}{w-w_0};q\right)$, then $q=\frac{\lambda}{1-z_0\overline{z_1}}+\frac{\overline{\lambda}}{1-z_0^{-1}z_1}+\ln|cz_0|^2$.
\end{theorem}

\begin{proof}[Proof of Theorem \ref{thm:UnboundedSingularOnePtLQD}]\label{proof:UnboundedSingularOnePtLQD}
Recall that, by Theorem \ref{thm:ScalingLaw}, $\Omega\in\left(\frac{\alpha}{w-w_0}\right)$ if and only if $w_0^{-1}\Omega\in\left(\frac{\alpha}{w-1}\right)$, so it is sufficient to consider the case in which $w_0=1$.

For the forward direction, suppose that $0\in\Omega\in\QD_0\left(\frac{\alpha}{w-1}\right)$ with $\alpha\in\C\setminus\{0\}$ is unbounded and simply connected with Riemann map $\varphi:\D\IntComp\rightarrow\Omega$, normalized such that $\varphi(z_0)=0$, $\varphi(z_1)=1$, and $\varphi'(\infty)=c>0$. By Theorem \ref{thm:LQDFTInverseProb}, $\varphi(z)=c|z_0|zb_{z_0}(z)e^{r^{\#}(z)}$, where $r(z)=\Phi_{\varphi}^{-1}\left(\frac{\alpha w}{w-1}\right)(z)-\Phi_{\varphi}^{-1}\left(\frac{\alpha w}{w-1}\right)(0)$. Computing the Faber transform using Equation \ref{eqn:FaberPolyFormulae}, we obtain
$$\Phi_{\varphi}^{-1}\left(\frac{\alpha w}{w-1}\right)(z)=\alpha+\alpha\Phi_{\varphi}^{-1}\left(\frac{1}{w-1}\right)(z)=\alpha\left(1+\dfrac{\psi'(1)}{z-\psi(1)}\right)=\alpha\left(1+\dfrac{1}{\varphi'(z_1)}\dfrac{1}{z-z_1}\right).$$
Thus $r(z)=\frac{\alpha}{z_1\varphi'(z_1)}\frac{z}{z-z_1}$ and, setting $\lambda=\frac{\overline{\alpha}}{\overline{z_1}\overline{\varphi'(z_1)}}$, we find
$$\varphi(z)=c|z_0|zb_{z_0}(z)e^{\frac{\lambda}{1-z\overline{z_1}}}$$

For the reverse direction, suppose there exist $\lambda\in\C$ and $z_0,z_1\in\D\IntComp$ for which $\Omega=\varphi(\D\IntComp)$, where $\varphi$ is univalent and given by Equation \ref{eqn:UnboundedSingularOnePtLQD} (with $w_0=1$). Then, by Theorem \ref{thm:SCLQDCharacterization}, $\Omega\in\QD_0(h)$ for some $h$. We will now compute $h$ using Theorem \ref{thm:LQDFTInverseProb}, which tells us that there exists $C\in\C$ such that
\begin{align*}
h(w)&=\dfrac{\Phi_{\varphi}\left(\frac{\overline{\lambda}z}{z-z_1}\right)(w)-C}{w}=\dfrac{\overline{\lambda}z_1}{w}\Phi_{\varphi}\left(\frac{1}{z-z_1}\right)(w)+\dfrac{\overline{\lambda}-C}{w}
\end{align*}
Computing the Faber transform via Equation \ref{eqn:FaberPolyFormulae} and substituting the formula for $\lambda$, we find
\begin{align*}
h(w)&=\dfrac{\overline{\lambda}z_1}{w}\dfrac{\varphi'(z_1)}{w-\varphi(z_1)}+\dfrac{\overline{\lambda}-C}{w}=\dfrac{\alpha}{w-1}+\dfrac{\overline{\lambda}-C-\alpha}{w}.
\end{align*}
Hence $\Omega\in\QD_0\left(\frac{\alpha}{w-1}\right)$ by Corollary \ref{cor:LQDQuadFuncUniquenessZero}.

It remains to show the formula for $q$ under the specification $\Omega\in\QD_0\left(\frac{\alpha}{w-1};q\right)$. We begin by computing the generalized Schwarz function. Note that $\frac{\ln|w|^2}{w}\dEquals \frac{\ln(\varphi\varphi^{\#})\circ\psi(w)}{w}$, so
\begin{align*}
    S_0(w)&:=\dfrac{\ln(\varphi\varphi^{\#})\circ\psi(w)}{w}=\dfrac{\ln|cz_0|^2+\left(\frac{\lambda}{1-\psi(w)\overline{z_1}}+\frac{\overline{\lambda}\psi(w)}{\psi(w)-z_1}\right)}{w}
\end{align*}
is a generalized Schwarz function for $\Omega$. Noting the singularities at $w=0$ and $w=1$, computing the residues, and applying Mittag-Leffler's theorem, we find that there exist $\widetilde{C}\in\C$ and $G\in\A_0(\Omega)$ such that
\begin{align*}
    S_0(w)&=\dfrac{\widetilde{C}}{w-1}+\dfrac{\frac{\lambda}{1-z_0\overline{z_1}}+\frac{\overline{\lambda}}{1-z_0^{-1}z_1}+\ln|cz_0|^2}{w}+G(w).
\end{align*}
Hence, $q=\frac{\lambda}{1-z_0\overline{z_1}}+\frac{\overline{\lambda}}{1-z_0^{-1}z_1}+\ln|cz_0|^2$.

We then recover formulae for arbitrary $w_0$ by replacing instances of $\varphi$ and $\varphi'$ with $w_0^{-1}\varphi$ and $w_0^{-1}\varphi'$ respectively.
\end{proof}

Note that both $\Omega$ and $\Omega^{-1}$ are unbounded domains containing zero. Moreover, by Corollary \ref{cor:LQDInversion}, $\Omega\in\QD_0\left(\frac{\alpha}{w-w_0};q\right)$ if and only if $\Omega^{-1}\in\QD_0\left(\frac{\alpha}{w-w_0^{-1}}-\frac{\alpha}{w};-q-A_{\rho_0}(\Omega\IntComp)\right)$.
$$\Omega^{-1}\in\QD_0\left(\frac{\alpha}{w-w_0^{-1}};-q-\alpha-A_{\rho_0}(\Omega\IntComp)\right).$$

Setting $w_0=1$, we find that $\Omega$ and $\Omega^{-1}$ have the same quadrature function, simply for different values of $q$.

\begin{figure}[ht]
  \centering
\begin{table}[H]
\centering
\begin{tblr}{
  cells={valign=m,halign=c},
  row{1}={bg=lightgray,font=\bfseries,rowsep=8pt},
  column{1}={bg=lightgray,font=\bfseries},
  colspec={QQQ},
  hlines,
  vlines
}
 & $\Omega$ Bounded & $\Omega$ Unbounded\\ 
$0\notin\Omega$ & \includegraphics[height=0.28\textwidth,width=.28\textwidth,valign=c
]{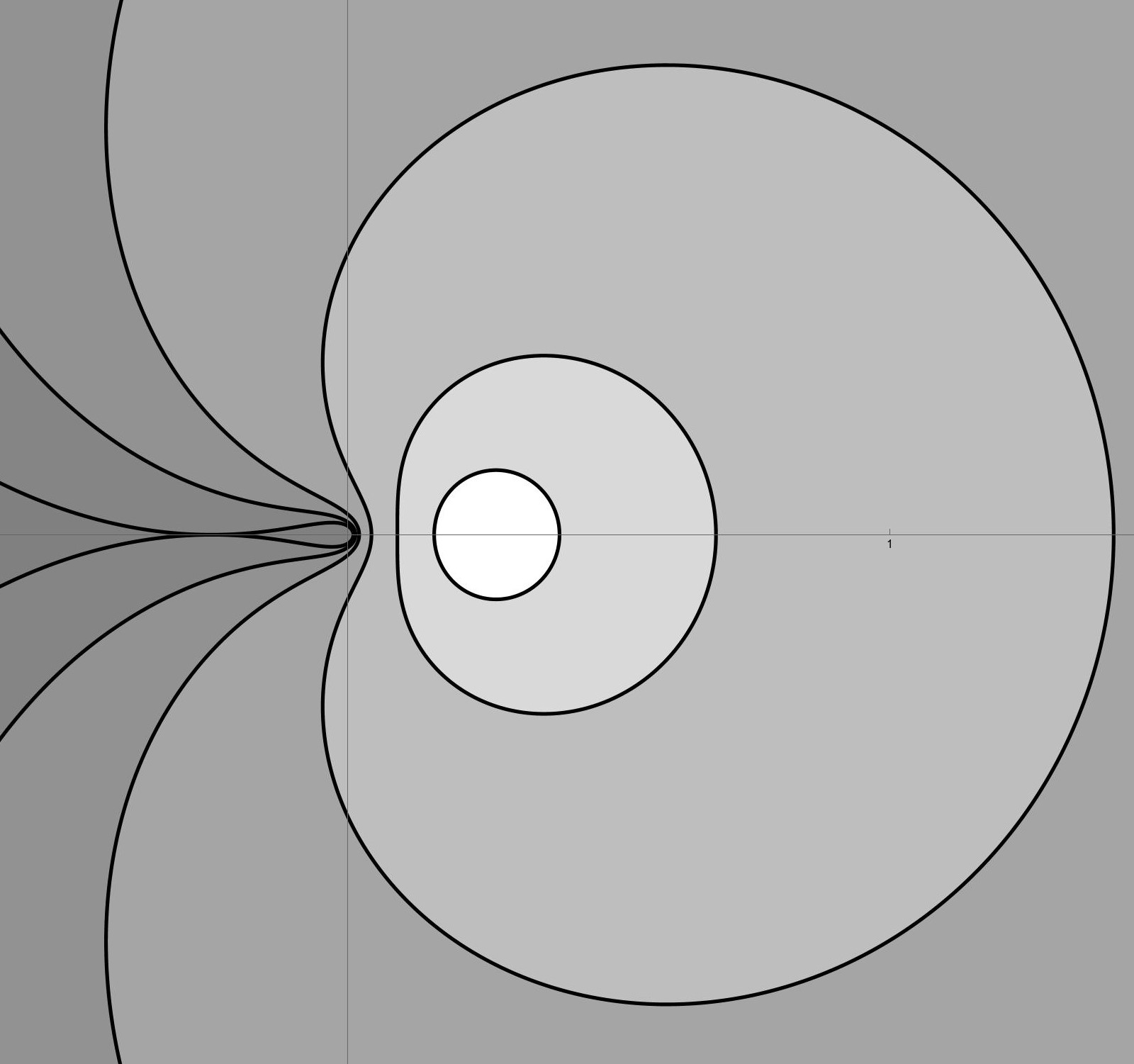} & \includegraphics[height=0.28\textwidth,width=.28\textwidth,valign=c
]{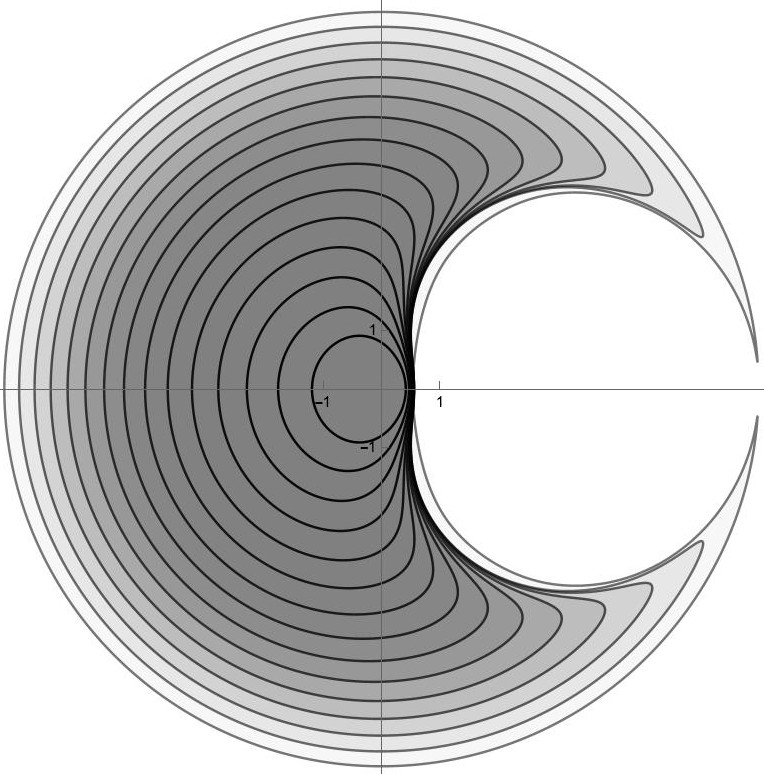}\;\;\includegraphics[height=0.28\textwidth,width=.28\textwidth,valign=c]{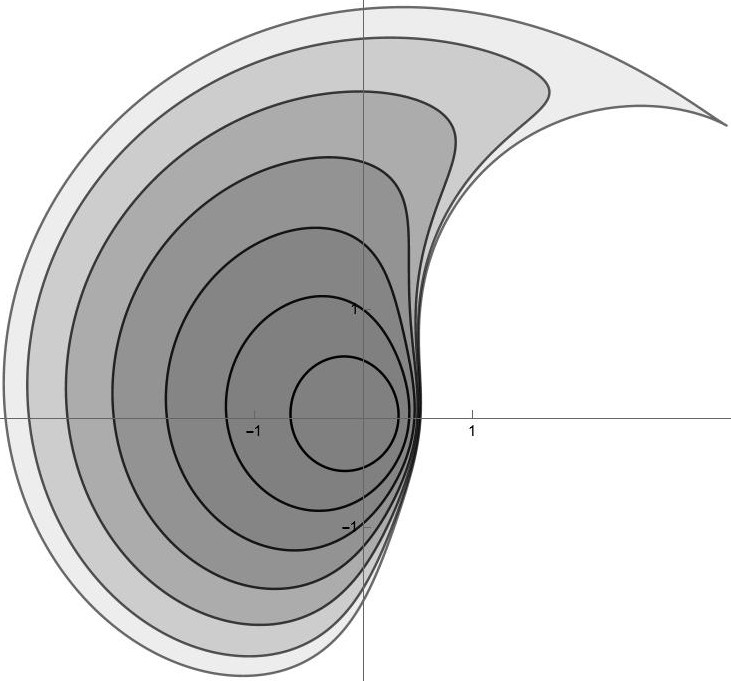}\\ 
$0\in\Omega$ & \includegraphics[height=0.28\textwidth,width=.28\textwidth,valign=c
]{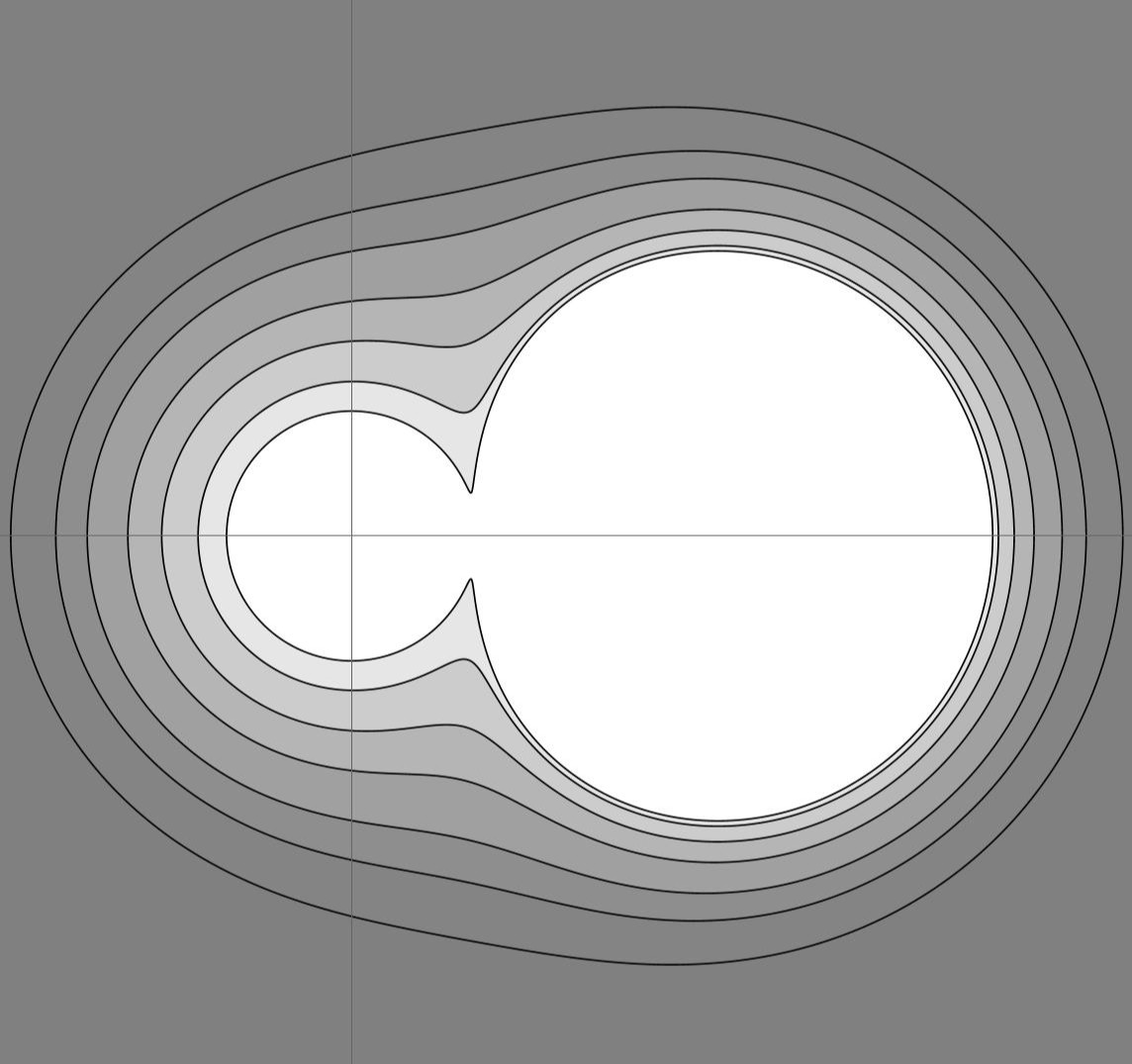} & \includegraphics[height=0.28\textwidth,width=.28\textwidth,valign=c
]{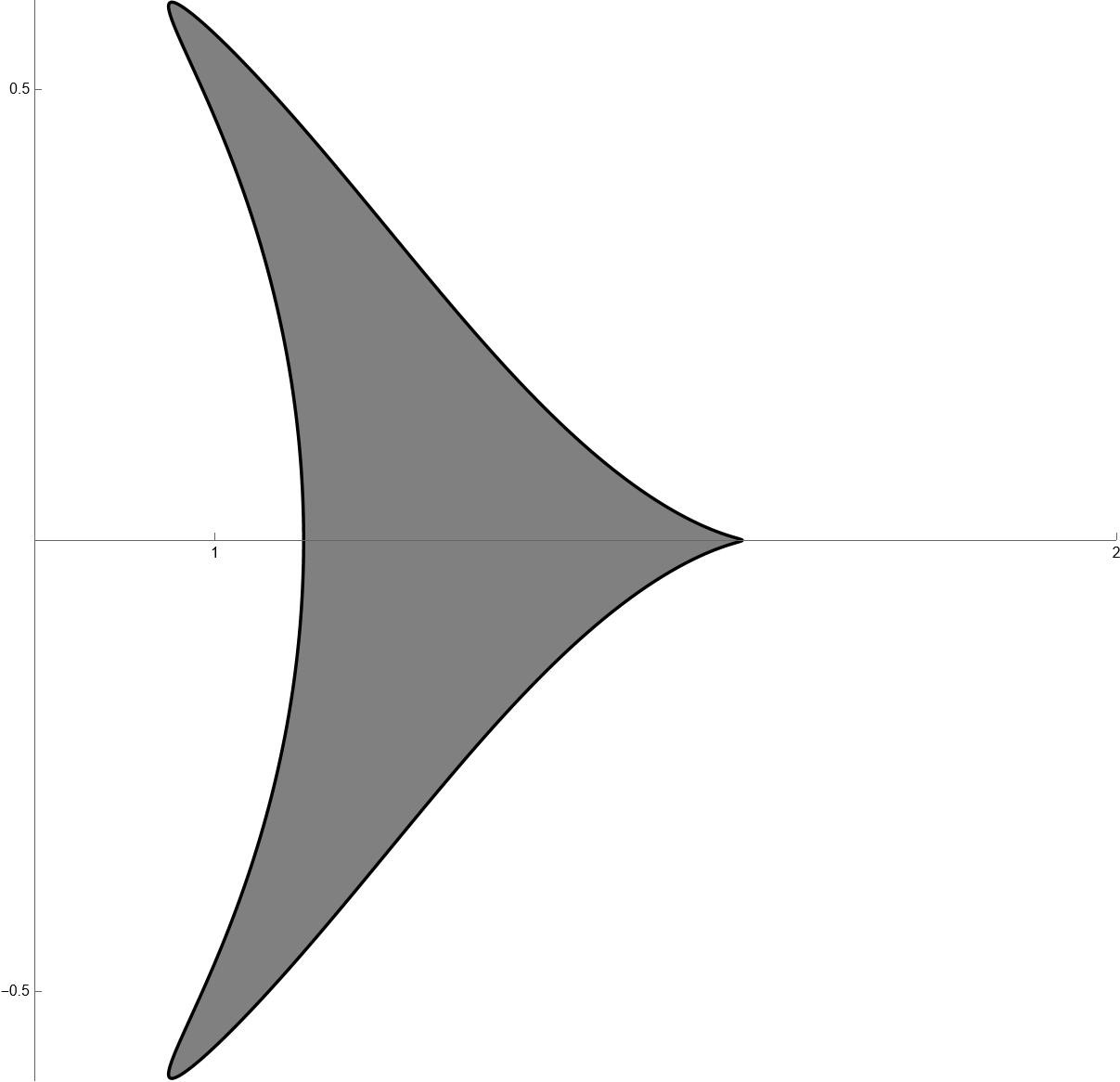}\\
\end{tblr}
\end{table}\vspace{.5em}
\caption{The families of (shaded complements of) one-point LQDs $\Omega\in\QD_0\left(\frac{\alpha}{w-w_0}\right)$ with}
\vspace{-0em}
\begin{itemize}
\itemindent=27pt
    \item $w_0=.25$, $2\leq \alpha\leq\pi^2$ (top left, \S\ref{subsec:BasicOnePTLQDEx});
    \item $w_0=1$, $\alpha=.7$, $q\in(-1.6,2.5)$ (bottom left, \S\ref{subsubsec:LOGWQDBoundedOnePtZ});
    \item $w_0=1.9$, $\alpha=1.5,1.5+.3i$ (top right, \S\ref{subsec:LOGWQDOnePtNZ});
    \item $w_0=2$, $\alpha=-.15$, $c=.389$, $z_0=-3.13$, $z_1=2.28$ (bottom right, \S\ref{subsubsec:LOGWQDUnboundedOnePtZ}).
\end{itemize}
\label{fig:LOGWQDOnePt}\vspace{1.5em}
\end{figure}

\begin{figure}[ht]
  \centering
\includegraphics[height=0.4\linewidth,width=.4\linewidth]{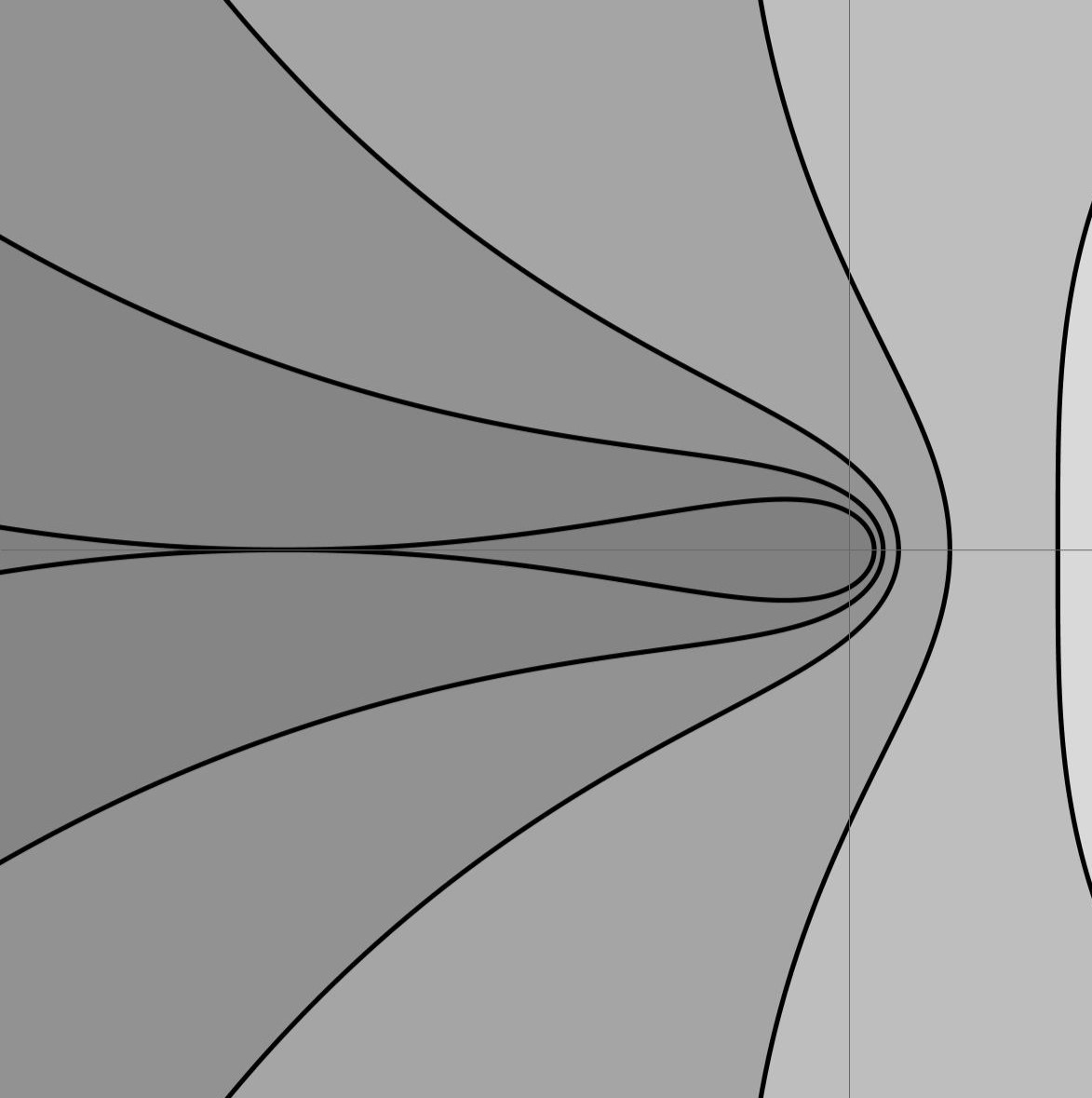}
\vspace{.5em}
\caption{Zoom of double-point formation in Figure \ref{fig:LOGWQDOnePt}, $\Omega\in\QD_0\left(\frac{\alpha}{w-w_0}\right)$ with $w_0=.25$ and $0<\alpha\leq\pi^2$}\label{fig:LOGWQDBoundedNoZeroOnePtZoom}\vspace{1.5em}
\end{figure}

\FloatBarrier

\subsection{Example: A Family of Two-Point LQDs}
In this section, we consider a doubly symmetric family of simply connected LQDs with quadrature function $h(w)=\frac{\alpha}{w-1}+\frac{\alpha}{w+1}$ for $\alpha>0$.

\begin{theorem}\label{thm:2PtSymmetricSingularLQDTheorem}
Fix $\alpha>0$ and $q\in\C$. If $\Omega\in\QD_0\left(\frac{\alpha}{w-1}+\frac{\alpha}{w+1};q\right)$ is a simply connected singular LQD symmetric about the real and imaginary axes, then $\Omega=\varphi(\D)$, where
\begin{equation}\label{eqn:2PtSymmetricLQDFormula2}
    \varphi(z)=\dfrac{z}{z_{+}}e^{\lambda\frac{z^2-z_+^2}{(zz_+)^2-1}}.
\end{equation}
Moreover, $\varphi(z_+)=1$, $\lambda=\frac{q+\ln(z_+^2)}{z_+^2+z_+^{-2}}$, and $\alpha=\lambda^2+\lambda\frac{z_+^4-1}{2z_{+}^2}$.
\end{theorem}

\begin{figure}[ht]
  \centering
\includegraphics[height=0.37\linewidth,width=.6\linewidth]{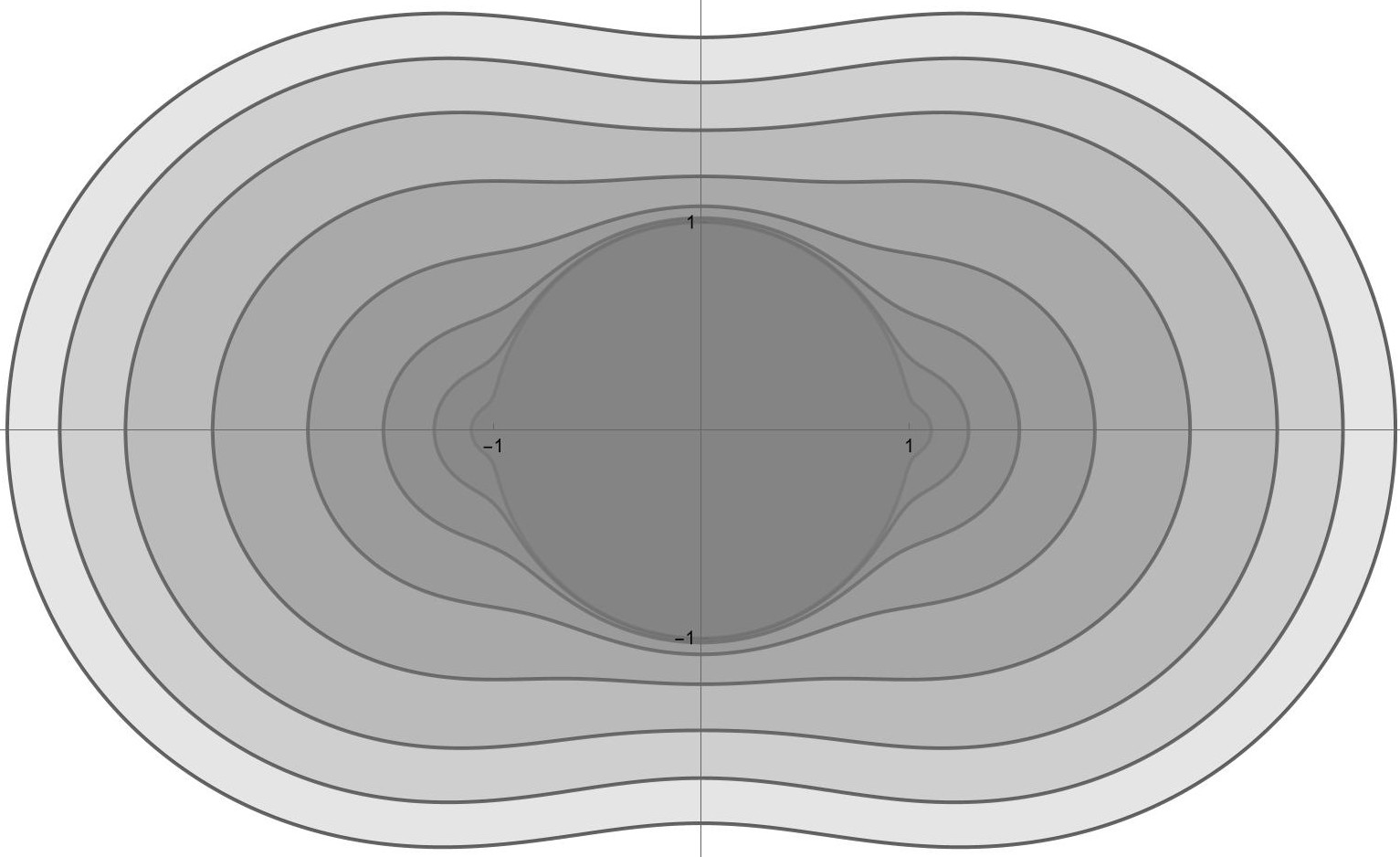}
\vspace{.5em}
\caption{Family of LQDs in $\QD_0\left(\frac{\alpha}{w-1}+\frac{\alpha}{w+1};q\right)$, for $q=0$ and $0<\alpha<1$.}\label{fig:TwoPointLOGWQDZero}
\end{figure}

\begin{proof}[Proof of Theorem \ref{thm:2PtSymmetricSingularLQDTheorem}]
Suppose that $\Omega\in\QD_0\left(\frac{\alpha}{w-1}+\frac{\alpha}{w+1}\right)$ for some $\alpha>0$. If the Riemann map associated to $\Omega$, $\varphi:\D\rightarrow\Omega$ takes $0$ to $0$, then $\varphi(z)=cze^{r^{\#}(z)}$, where $c=\varphi'(0)>0$. Also choose $z_+,z_-\in\D$ such that $\varphi(z_{\pm})=\pm1$. By Theorem \ref{thm:LQDFTInverseProb}, $r(z)=\Phi_{\varphi}^{-1}\left(\frac{\alpha}{w-1}+\frac{\alpha}{w+1}-2\alpha\right)(z)$. Applying the Faber transform formulae of Equation \ref{eqn:FaberPolyFormulae}, we find
\begin{align*}
    r(z)&=\alpha\Phi_{\varphi}^{-1}\left(\dfrac{1}{w-1}\right)(z)-\alpha\Phi_{\varphi}^{-1}\left(\dfrac{1}{w+1}\right)(z)\\
    &=\dfrac{\alpha}{\varphi'(z_{+})}\dfrac{1}{z-z_{+}}-\dfrac{\alpha}{\varphi'(z_{-})}\dfrac{1}{z-z_{-}}.
\end{align*}
Assuming the domain is symmetric about the real and imaginary axes, we find that $\varphi(-z)=-\varphi(z)$, $\varphi'(-z)=\varphi'(z)$, and $z_{-}=-z_{+}\in\R$, so this simplifies to $r(z)=\frac{2z_{+}\alpha}{\varphi'(z_{+})}\frac{1}{z^2-z_{+}^2}$, and we obtain
\begin{equation}\label{eqn:2PtSymmetricLQDFormula1}
    \varphi(z)=cze^{\frac{2z_{+}\alpha}{\varphi'(z_{+})}\frac{z^2}{1-z^2z_{+}^2}}.
\end{equation}
Using the relations for $\varphi(z_+)$ and $\varphi'(z_+)$, we find that
\begin{align*}
    \varphi(z)&=\dfrac{z}{z_{+}}e^{\frac{\ln(cz_{+})}{z_{+}^2}\frac{z^2-z_+^2}{(zz_+)^2-1}}.
\end{align*}
We now compute $q$ - that is, we know $\Omega\in\QD_0(h;q)$, and we would like to determine a relationship between the coefficients of $\varphi$ and the charge $q$ at $0$. We compute the Schwarz function,
\begin{align*}
    \dfrac{\ln|w|^2}{w}&\dEquals\dfrac{\ln(\varphi\varphi^{\#})\circ\psi(w)}{w}\\
    &=\dfrac{\dfrac{\ln(cz_{+})}{z_+^2}\left(\frac{z^2-z_+^2}{(zz_+)^2-1}+\frac{1-(zz_+)^2}{z_+^2-z^2}\right)\circ\psi(w)-\ln(z_+^2)}{w}\\
    &=\dfrac{\ln\left(\frac{wz_{+}}{\psi(w)}\right)+\left(\frac{\ln(cz_{+})}{z_+^2}\right)^2\frac{1}{\ln\left(\frac{wz_{+}}{\psi(w)}\right)}}{w}-\dfrac{\ln(z_+^2)}{w}
\end{align*}
By L'Hopital's rule,
$$\lim_{w\to0}\dfrac{wz_{+}}{\psi(w)}=\dfrac{z_{+}}{\psi'(0)}=z_{+}\varphi'(0)=cz_{+}.$$
So, expanding about $w=0$, we obtain
\begin{align*}
    \dfrac{\ln|w|^2}{w}&\dEquals\dfrac{\ln(cz_{+})(z_+^{-4}+1)-\ln(z_+^2)}{w}+O(w)
\end{align*}
Hence, $q=\ln(cz_{+})(z_+^{-4}+1)-\ln(z_+^2)$. Setting $\lambda=\frac{q+\ln(z_+^2)}{z_+^2+z_+^{-2}}$, we recover Equation \ref{eqn:2PtSymmetricLQDFormula2}. Finally, combining Equations \ref{eqn:2PtSymmetricLQDFormula1} and \ref{eqn:2PtSymmetricLQDFormula2}, we obtain the equation for $\alpha$ in the theorem.
\end{proof}


\section{Conclusion}

This paper develops the theory of log-weighted quadrature domains - plane domains satisfying a quadrature identity with respect to the weight $\rho_0(w)=|w|^{-2}$. The presence of the singularity at the origin leads to several new features absent from the classical theory. Principal among these is the fact that, when $0\in\Omega$, the quadrature function is no longer uniquely determined by the domain, but only up to the addition of a point charge at the origin. At the structural level, the paper establishes a generalized Schwarz function characterization of LQDs (Theorem \ref{thm:LQDSFEEquiv}) and derives a corresponding boundary regularity result (Theorem \ref{thm:LQDBoundaryRegularity}), parallel to the classical theory of quadrature domains. 

The main result in the simply connected setting is that a domain is an LQD if and only if the outer factor of its Riemann map extends to the exponential of a rational function (Theorem \ref{thm:SCLQDCharacterization}). This gives a natural analogue of the Riemann map characterization of classical QDs, while showing that in the log-weighted setting the correct object is not the full Riemann map itself, but rather its rational inner part together with the rational data encoded by the outer factor. From this characterization one obtains explicit Faber transform formulae relating the quadrature function to the Riemann map (Theorem \ref{thm:LQDFTInverseProb}) which, in turn, renders the inverse and direct problems effectively computable in a number of cases. These general results are then applied to classify null LQDs (Theorem \ref{thm:NullLQDClass}), to obtain a partial classification of one-point LQDs (\S\ref{sec:OnePointLQDs}), and to analyze monomial and one-point families in both the singular and non-singular settings.

Several natural directions remain open. First, while the simply connected case admits a clean description in terms of the inner/outer factorization of the Riemann map, it would be of interest to generalize this to the multiply connected setting. One promising direction is the work of Crowdy and Marshall \cite{CrowdyMultipleQDs2004}, which uses Schottky–Klein prime functions to address the multiply connected inverse problem for classical quadrature domains. Second, the families treated here suggest the possibility of sharper classification theorems for singular one-point and monomial LQDs, especially in the absence of symmetry. More broadly, this study of LQDs should be viewed as part of a larger program of understanding \emph{weighted quadrature domains} - especially those with singular/degenerate weights. It would be interesting to determine to what extent the methods developed here extend to other weighted theories. My hope is that the present work provides a useful framework for such problems by demonstrating that much of the classical QD machinery survives in this weighted setting, albeit in a modified form.

\newpage
\printbibliography

\newpage
\section{Appendix 1: The Cauchy Projection and the Faber Transform}\label{subsec:CauchyTransformFaberTransform}

\subsection{The Cauchy Projection}\label{subsec:CauchyTransformV2}
If $\Omega$ is a rectifiable domain, then there is a natural linear operator $\AnalyticInNoBracket{\cdot}{\Omega}:C^0(\partial\Omega)\rightarrow H^{\infty}(\Omega)$ called the {\it Cauchy projection},
\begin{equation}\label{eqn:analyticprojection}
\AnalyticIn{f}{\Omega}(w):=\oint_{\partial\Omega}\dfrac{f(\xi)}{\xi-w}d\xi,
\end{equation}

The Cauchy projection has a number of useful properties, summarized in the following lemma.
\begin{lemma}\label{lemma:AnalyticDecompositionLemmaV2}
Fix a bounded domain $\Omega$ with a piecewise $C^1$ boundary.
    \begin{enumerate}
        \item If $f$ is analytic on $\partial\Omega$,\footnote{meaning analytic in an open nbhd of $\partial\Omega$} then
        \begin{itemize}
            \item $\AnalyticInNoBracket{f}{\Omega}\in\A(\Omega)$,
            \item $\AnalyticInNoBracket{f}{\Omega\IntComp}\in\A_0(\Omega\IntComp)$,
            \item $f\dEquals\AnalyticInNoBracket{f}{\Omega}+\AnalyticInNoBracket{f}{\Omega\IntComp}$.
        \end{itemize}
        \item If $f\in\A(\Omega)$, then $\AnalyticInNoBracket{f}{\Omega}=f$ and $\AnalyticInNoBracket{f}{\Omega\IntComp}=0$.
        \item If $f\in\A_0(\Omega\IntComp)$, then $\AnalyticInNoBracket{f}{\Omega\IntComp}=f$ and $\AnalyticInNoBracket{f}{\Omega}=0$.
        \item If $f\in\M(\Omega)$ and extends continuously to $\partial\Omega$, then $\AnalyticInNoBracket{f}{\Omega\IntComp}\in\Rat_0(\Omega)$.
        \item If $f\in\M(\Omega\IntComp)$ and extends continuously to $\partial\Omega$, then $\AnalyticInNoBracket{f}{\Omega}\in\Rat(\Omega\IntComp)$.
    \end{enumerate}
    The analogous results hold for unbounded $\Omega$.
\end{lemma}
\begin{proof}[Proof of Lemma \ref{lemma:AnalyticDecompositionLemmaV2}]\;\\
(1) Follows immediately from the Sokhotski-Plemelj theorem (see e.g. \cite{Henrici1986}, Theorem 14.1c).\\
(2-3) Follow from the Cauchy integral formula and the identity theorem.\\
(4) Follows from the principal part decomposition of $f$: $f=r+G$, where $r\in\Rat_0(\Omega)$ and $G\in\A(\Omega)$. Hence, applying (2-3), we obtain
$$\AnalyticInNoBracket{f}{\Omega\IntComp}=\AnalyticInNoBracket{r}{\Omega\IntComp}+\AnalyticInNoBracket{G}{\Omega\IntComp}=r$$
The argument for (5) is entirely analogous.\footnote{Note that $f$ has finitely many poles even if its domain is unbounded because if $f$ is meromorphic at $\infty$, then $f(w^{-1})$ has at most an isolated pole at $0$.}
\end{proof}
See theorem 2.3 of \cite{lenells2014matrixriemannhilbertproblemsjumps} for a related exposition. The Cauchy projection, $\AnalyticInNoBracket{f}{\Omega}$, is also referred to as the ``analytic part'' of $f$ in $\Omega$.\\

With the definition of the Cauchy projection established, we now prepared to define the Faber transform.

\subsection{The Faber Transform}\label{subsec:FaberTransform}
The Faber transform is a linear operator between functions analytic in the disk (or exterior disk) and functions analytic in a given simply connected domain. In the present paper it serves as the main mechanism for passing between the quadrature function and the Riemann map. In particular, it converts rational data on one side into rational data on the other, and it admits explicit formulae for rational inputs. See \cite{GravenMakarov2025} \S1.5.1, and \cite{FaberTransformRationalApprox1983,FaberTransformAnalyticContinuation1988,TheFaberOperator1984}, for further details on the Faber transform, Faber polynomials, and their properties. We will begin by defining the \emph{interior Faber transform}.

\begin{definition}
Let $\Omega$ be a bounded simply connected domain with Riemann map $\varphi:\D\rightarrow\Omega$, normalized so that $\varphi(0)=w_0\in\Omega$, and $\varphi'(0)>0$. Also write $\psi:=\varphi^{-1}$. The associated interior Faber transform is the map $\Phi_{\varphi}:\A_0(\D\IntComp)\rightarrow\A_0(\Omega\IntComp)$ given by
\begin{equation}\label{eqn:InteriorTransformFormula}
\Phi_{\varphi}(f)(w):=\AnalyticIn{f\circ\psi}{\Omega\IntComp}(w)=\oint_{\partial\Omega\IntComp}\dfrac{f\circ\psi(\xi)}{\xi-w}d\xi.
\end{equation}
\end{definition}
\noindent Its inverse $\Phi_{\varphi}^{-1}:\A_0(\Omega\IntComp)\rightarrow \A_0(\D\IntComp)$ is given by $\Phi_{\varphi}^{-1}(f):=\AnalyticIn{f\circ\varphi}{\D\IntComp}$.\\

\noindent The important properties for our purposes are the following:
\begin{itemize}
    \item $\Phi_{\varphi}$ is a linear isomorphism $\A_0(\D\IntComp)\rightarrow\A_0(\Omega\IntComp)$.
    \item $\Phi_{\varphi}$ restricts to an isomorphism $\Rat_0(\D)\rightarrow\Rat_{0}(\Omega\IntComp)$.
    \item If $f\in C^0(\partial\D)$, then
    \begin{equation}\label{eqn:InteriorFaberTransformProjectionExtension}
        \AnalyticIn{f\circ\psi}{\Omega\IntComp}=\Phi_\varphi\left(\AnalyticInNoBracket{f}{\D\IntComp}\right).
    \end{equation}
    Similarly, when $\partial\Omega$ is piecewise $C^1$ and $f\in C^0(\partial\Omega)$,
    \begin{equation}\label{eqn:InteriorInverseFaberTransformProjectionExtension}
        \AnalyticIn{f\circ\varphi}{\D\IntComp}=\Phi_{\varphi}^{-1}\left(\AnalyticInNoBracket{f}{\Omega\IntComp}\right).
    \end{equation}
\end{itemize}

The discussion for the \emph{exterior Faber transform} is largely analogous to the interior case, with the primary difference being the function spaces on which it is defined. Another important difference is 

\begin{definition}
Let $\Omega$ be an unbounded simply connected domain with Riemann map $\varphi:\D\IntComp\rightarrow\Omega$, normalized so that $\varphi(\infty)=\infty$, and $\varphi'(\infty)=c>0$. Also write $\psi:=\varphi^{-1}$. The associated exterior Faber transform is the map $\Phi_{\varphi}:\A(\D)\rightarrow\A(\Omega\IntComp)$, given by the same formula:
\begin{equation}\label{eqn:ExteriorTransformFormula}
\Phi_{\varphi}(f)(w):=\AnalyticIn{f\circ\psi}{\Omega\IntComp}(w)=\oint_{\partial\Omega\IntComp}\dfrac{f\circ\psi(\xi)}{\xi-w}d\xi.
\end{equation}
\end{definition}
\noindent Its inverse $\Phi_{\varphi}^{-1}:\A(\Omega\IntComp)\rightarrow\A(\D)$ is given by $\Phi_{\varphi}^{-1}(f)=\AnalyticIn{f\circ\varphi}{\D}$.\\

The corresponding properties are completely analogous:
\begin{itemize}
    \item $\Phi_{\varphi}$ is a linear isomorphism $\A(\D)\rightarrow\A(\Omega\IntComp)$.
    \item $\Phi_{\varphi}$ restricts to an isomorphism $\Rat(\D\IntComp)\rightarrow\Rat(\Omega\IntComp)$.
    \item If $f\in C^0(\partial\D)$, then
    \begin{equation}\label{eqn:ExteriorFaberTransformProjectionExtension}
        \AnalyticIn{f\circ\psi}{\Omega\IntComp}=\Phi_\varphi\left(\AnalyticInNoBracket{f}{\D}\right).
    \end{equation}
    Similarly, when $\partial\Omega$ is piecewise $C^1$ and $f\in C^0(\partial\Omega)$,
    \begin{equation}\label{eqn:ExteriorInverseFaberTransformProjectionExtension}
        \AnalyticIn{f\circ\varphi}{\D}=\Phi_{\varphi}^{-1}\left(\AnalyticInNoBracket{f}{\Omega\IntComp}\right).
    \end{equation}
\end{itemize}

\noindent{\it Faber polynomials}: The Faber polynomials associated to an unbounded domain $\Omega$ (and $\varphi$) are defined by
$$F_n=\Phi_{\varphi}(z^n)=\AnalyticIn{\psi^n}{\Omega\IntComp},\;\;\;n\geq0.$$
Similarly, the inverse Faber polynomials $\{W_n\}_{n\geq0}$ are defined as the inverse Faber transform of the standard monomial basis. Concretely, $F_n$ and $W_n$ are the polynomial parts of $\psi^n$ and $\varphi^n$ respectively. They also satisfy the following asymptotic identities
$$F_n\circ\varphi(z)=z^n+O(1),\;\;\;\;\;\;W_n\circ\psi(w)=w^n+O(1).$$
For applications throughout the paper, an important fact is that the Faber transform of a rational function can be computed explicitly via the calculus of residues. If $w_0\in\Omega$ and $z_0\in\D$ ($\D\IntComp$ for the exterior transform), then
\begin{equation}\label{eqn:FTFormulae}
\begin{alignedat}{1}
\Phi_{\varphi}\left(\dfrac{1}{(z-z_0)^{n}}\right)(w)&=\frac{1}{(n-1)!}\left.\left(\frac{\varphi'(\xi)}{w-\varphi(\xi)}\right)^{(n-1)}\right\vert_{\xi=z_0}\\
\Phi_{\varphi}^{-1}\left(\dfrac{1}{(w-w_0)^{n}}\right)(z)&=\frac{1}{(n-1)!}\left.\left(\frac{\psi'(\xi)}{z-\psi(\xi)}\right)^{(n-1)}\right\vert_{\xi=w_0}
\end{alignedat}
\end{equation}
These admit closed-form representations in terms of the incomplete Bell polynomials $B_{n,k}$:
\begin{equation}
    \Phi_{\varphi}\left(\dfrac{1}{(z-z_0)^{n}}\right)(w)=\sum_{k=1}^{n}\dfrac{(k-1)!}{(n-1)!}\dfrac{B_{n,k}(\varphi'(z_0),\hdots,\varphi^{(n-k+1)}(z_0))}{(w-\varphi(z_0))^{k}}
\end{equation}
The formula is analogous for $\Phi_{\varphi}^{-1}$. The specific formulae for $n=1$ and $n=2$ are provided below.

\begin{equation}\label{eqn:FaberPolyFormulae}
\begin{alignedat}{2}
\Phi_{\varphi}\left(\dfrac{1}{z-z_0}\right)(w)&=\dfrac{\varphi'(z_0)}{w-\varphi(z_0)},\;\;\;&\Phi_{\varphi}\left(\dfrac{1}{(z-z_0)^{2}}\right)(w)&=\dfrac{\varphi''(z_0)}{w-\varphi(z_0)}+\dfrac{\varphi'(z_0)^2}{(w-\varphi(z_0))^2}\\
\Phi_{\varphi}^{-1}\left(\dfrac{1}{w-w_0}\right)(z)&=\dfrac{\psi'(w_0)}{z-\psi(w_0)},\;\;\;&\Phi_{\varphi}^{-1}\left(\dfrac{1}{(w-w_0)^{2}}\right)(z)&=\dfrac{\psi''(w_0)}{z-\psi(w_0)}+\dfrac{\psi'(w_0)^2}{(z-\psi(w_0))^2}\\
& & &\\
F_1(w)&=\dfrac{w}{c}-\dfrac{f_0}{c}, &F_2(w)&=\dfrac{w^2}{c^2}-\dfrac{2f_0}{c^2}w+\dfrac{f_0^2-2cf_1}{c^2}\\
W_1(z)&=cz+f_0, &W_2(z)&=c^2z^2+2cf_0z+f_0^2+2cf_1,
\end{alignedat}
\end{equation}
where the $f_j$ are the coefficients in the Laurent expansion of $\varphi$ at $\infty$, $\varphi(z)=cz+f_0+f_1z^{-1}+...$. These formulae are used repeatedly throughout the paper.

\section{Appendix 2: Proofs}

\subsection{Electrostatic Interpretation of LQDs Proof}\label{subsection:ElectrostaticLQDInterpretation}

\begin{proof}[Proof of Theorem \ref{thm:ElectrostaticLQDInterpretation}]\label{proof:ElectrostaticLQDInterpretation}
We first consider the non-singular case ($0\notin\Omega$):\\
For the forward direction, note that if $\Omega\in\QD_0(h)$ then the quadrature identity implies that for each $w\in\Omega\IntComp$,
\begin{align*}
    C_{\rho_0}^{\Omega}(w)&=\int_{\Omega}\dfrac{|\xi|^{-2}}{w-\xi}dA(\xi)=\oint_{\partial\Omega}\dfrac{h(\xi)}{w-\xi}=h(w).
\end{align*}
For the reverse direction, suppose that $C_{\rho_0}^{\Omega}=h$ on $\Omega\IntComp$. Note that this identity extends continuously to the boundary, as the poles of $h$ are absent in an open neighborhood of $\partial\Omega$. Then, by Lemma \ref{lemma:LOGRenormCauchyTrans}, for each $w\in(\D_{r})\IntComp\supseteq\partial\Omega$,
\begin{align*}
    C_{\rho_0}^{\Omega}(w)&=C_{\rho_0}^{\C\setminus\D_r}(w)-C_{\rho_0}^{\Omega\IntComp\setminus\D_r}(w)\\
    &=\dfrac{\ln|w|^2-\ln|r|^2}{w}-C_{\rho_0}^{\Omega\IntComp\setminus\D_r}(w).
\end{align*}
Substituting our formula $C_{\rho_0}^{\Omega}(w)=h(w)$, we find that for all $w\in\partial\Omega$,
\begin{align*}
    \dfrac{\ln|w|^2}{w}&=h(w)+\dfrac{\ln|r|^2}{w}+C_{\rho_0}^{\Omega\IntComp\setminus\D_r}(w).
\end{align*}
Hence, if $f\in L_a^1(\Omega;\rho_0)$ then Green's theorem implies
\begin{align*}
    \int_{\Omega}\dfrac{f(w)}{|w|^2}dA(w)&=\oint_{\partial\Omega}f(w)\dfrac{\ln|w|^2}{w}dw\\
    &=\oint_{\partial\Omega}f(w)\left(h(w)+\dfrac{\ln|r|^2}{w}+C_{\rho_0}^{\Omega\IntComp\setminus\D_r}(w)\right)dw\\
    &=\oint_{\partial\Omega}f(w)h(w)dw,
\end{align*}
where the last step follows from the fact that $f(w)C_{\rho_0}^{\Omega\IntComp\setminus\D_r}(w)$ and $f(w)w^{-1}$ are analytic in $\Omega$. Hence, $\Omega\in\QD_0(h)$.\\\\
We now consider the singular case ($0\in\Omega$):\\
For the forward direction, suppose that $\Omega\in\QD_0(h;q)$. By Lemma \ref{lemma:LOGRenormCauchyTrans}
\begin{align*}
    C_{\rho_0}^{\Omega\setminus\D_r}(w)&=C_{\rho_0}^{\C\setminus\D_r}(w)-C_{\rho_0}^{\Omega\IntComp}(w)=\dfrac{\ln|w|^2-\ln|r|^2}{w}-C_{\rho_0}^{\Omega\IntComp}(w).
\end{align*}
By Theorem \ref{thm:LQDCEZero}, for each $w\in\partial\Omega$,
\begin{align*}
    C_{\rho_0}^{\Omega\setminus\D_r}(w)&=\dfrac{\ln|w|^2-\ln|r|^2}{w}-\left(\dfrac{\ln|w|^2}{w}-h(w)-\frac{q}{w}\right)\\
    &=h(w)+\dfrac{q-\ln|r|^2}{w}.
\end{align*}
For the reverse direction, suppose that $C_{\rho_0}^{\Omega\setminus\D_r}(w)=h(w)+\dfrac{q-\ln|r|^2}{w}$ for all $w\in\Omega\IntComp$. Again by Lemma \ref{lemma:LOGRenormCauchyTrans}, for each $w\in(\D_{r})\IntComp\supseteq\partial\Omega$,
\begin{align*}
    C_{\rho_0}^{\Omega\setminus\D_r}(w)&=C_{\rho_0}^{\C\setminus\D_r}(w)-C_{\rho_0}^{\Omega\IntComp}(w)=\dfrac{\ln|w|^2-\ln|r|^2}{w}-C_{\rho_0}^{\Omega\IntComp}(w).
\end{align*}
Hence, $\frac{\ln|w|^2}{w}\dEquals h(w)+\dfrac{q}{w}+C_{\rho_0}^{\Omega\IntComp}(w)$. If $f\in L_a^1(\Omega;\rho_0)$ then Green's theorem implies
\begin{align*}
    \int_{\Omega}\dfrac{f(w)}{|w|^2}dA(w)&=\oint_{\partial\Omega}f(w)\dfrac{\ln|w|^2}{w}dw\\
    &=\oint_{\partial\Omega}f(w)\left(h(w)+\dfrac{q}{w}+C_{\rho_0}^{\Omega\IntComp}(w)\right)dw\\
    &=\oint_{\partial\Omega}f(w)h(w)dw,
\end{align*}
where the last step follows from the fact that  $f(w)C_{\rho_0}^{\Omega\IntComp}(w)$ and $f(w)w^{-1}$ are analytic in $\Omega$. Hence, $\Omega\in\QD_0(h)$.
\end{proof}

\subsection{Inner/Outer Factorization Formula}
\begin{proof}[Proof of Theorem \ref{thm:RiemannMapInnerOuterRepresentation}]\label{proof:RiemannMapInnerOuterRepresentation}\;\\
We will begin by considering the {\bf bounded} case: Let $\varphi:\D\rightarrow\Omega$ be the Riemann map associated to a simply connected bounded domain $\Omega$ with piecewise $C^1$ boundary such that $0\notin\partial\Omega$. Then $\varphi\in H^\infty(\D)\cap C^{0}(\Cl(\D))$ by Caratheodory's theorem, and has no roots on $\partial\D$. Hence, by Lemma \ref{lemma:InnerOuterFactorization}, we can write $\varphi=\varphi_{\rm in}\varphi_{\rm out}$, where $\varphi_{\rm in}$ is a Blaschke product consisting of the zeros of $\varphi$ and $\varphi_{\rm out}$ is an outer function given by
\begin{align*}
    \varphi_{\rm out}(z)&=\exp\left(\dfrac{1}{2}\oint_{\partial\D}\dfrac{\xi+z}{\xi-z}\dfrac{\ln|\varphi(\xi)|^2}{\xi}d\xi\right).
\end{align*}
Reflecting about the unit circle, we obtain for each $z\in\D\IntComp$
\begin{align*}
    \varphi_{\rm out}^{\#}(z)&=\exp\left(-\dfrac{1}{2}\oint_{\partial\D}\dfrac{\xi^{-1}+z^{-1}}{\xi^{-1}-z^{-1}}\dfrac{\ln|\varphi(\xi)|^2}{\xi^{-1}}\left(-\dfrac{d\xi}{\xi^2}\right)\right)=\exp\left(\oint_{\partial\D\IntComp}\dfrac{\ln|\varphi(\xi)|^2}{\xi-z}d\xi-\dfrac{1}{2}\oint_{\partial\D\IntComp}\dfrac{\ln|\varphi(\xi)|^2}{\xi}d\xi\right)
\end{align*}
The second integral is simply a constant depending on $\varphi$, and we recognize the first integral as the Cauchy projection onto $\D\IntComp$. Hence, we obtain $\varphi_{\rm out}^{\#}(z)=C\exp\left(\AnalyticIn{\ln|\varphi(z)|^2}{\D\IntComp}\right)$. As $\partial\Omega$ is piecewise $C^1$ and $\ln|w|^2\in C^0(\partial\Omega)$, Equation \ref{eqn:InteriorInverseFaberTransformProjectionExtension} tells us that
$$\varphi_{\rm out}(z)=\overline{C}\exp\left(\Phi_{\varphi}^{-1}\left(\AnalyticIn{\ln|w|^2}{\Omega\IntComp}\right)^{\#}(z)\right).$$
Moreover, by construction, if $0\notin\Omega$, then $\varphi(z)\neq0$, so $\varphi_{\rm in}=1$. On the other hand, if $0\in\Omega$, then $\varphi$ has a unique zero $z_0\in\D$, in which case the Blaschke product collapses to a single factor $\varphi_{\rm in}=b_{z_0}$.

We will now consider the {\bf unbounded} case: Suppose $\varphi:\D\IntComp\rightarrow\Omega$ is the Riemann map associated to a simply connected unbounded domain $\Omega$ with piecewise $C^1$ boundary such that $0\notin\partial\Omega$, and set $f(z):=\frac{\varphi(z)}{z}$. Then $f\in H^{\infty}(\D\IntComp)\cap C^0(\Cl(\D\IntComp))$ by Caratheodory's theorem, and has no roots in $\partial\D\IntComp$. Hence, by Lemma \ref{lemma:InnerOuterFactorization}, we can write $f=f_{\rm in}f_{\rm out}$, where $f_{\rm in}$ is a Blaschke product consisting of the zeros of $\varphi$ and $f_{\rm out}$ is an outer function given by
\begin{align*}
    f_{\rm out}(z)&=\exp\left(\dfrac{1}{2}\oint_{\partial\D\IntComp}\dfrac{\xi+z}{\xi-z}\dfrac{\ln\left|\frac{\varphi(\xi)}{\xi}\right|^2}{\xi}d\xi\right)=\exp\left(\dfrac{1}{2}\oint_{\partial\D\IntComp}\dfrac{\xi+z}{\xi-z}\dfrac{\ln\left|\varphi(\xi)\right|^2}{\xi}d\xi\right)
\end{align*}
Reflecting about the unit circle, and taking $z\in\D$, we obtain
\begin{align*}
    f_{\rm out}^{\#}(z)&=\exp\left(-\dfrac{1}{2}\oint_{\partial\D\IntComp}\dfrac{\xi^{-1}+z^{-1}}{\xi^{-1}-z^{-1}}\dfrac{\ln\left|\varphi(\xi)\right|^2}{\xi^{-1}}\left(-\dfrac{d\xi}{\xi^2}\right)\right)=\exp\left(\oint_{\partial\D}\dfrac{\ln|\varphi(\xi)|^2}{\xi-z}d\xi-\dfrac{1}{2}\oint_{\partial\D}\dfrac{\ln|\varphi(\xi)|^2}{\xi}d\xi\right)
\end{align*}
The second integral is simply a constant depending on $\varphi$, and we recognize the first integral as the Cauchy projection onto $\D$. Hence, we obtain $f_{\rm out}^{\#}(z)=C\exp\left(\AnalyticIn{\ln|\varphi(z)|^2}{\D}\right)$. As $\ln|w|^2\in C^0(\partial\Omega)$, Equation \ref{eqn:ExteriorInverseFaberTransformProjectionExtension} tells us that
\begin{align*}
    f_{\rm out}(z)&=\overline{C}\exp\left(\Phi_{\varphi}^{-1}\left(\AnalyticIn{\ln|w|^2}{\Omega\IntComp}\right)^{\#}(z)\right).
\end{align*}
By construction, if $0\notin\Omega$, then $f(z)\neq0$, so $f_{\rm in}=1$. On the other hand, if $0\in\Omega$, then $f$ has a unique zero $z_0\in\D$, in which case the Blaschke product collapses to a single factor $f_{\rm in}=b_{z_0}$.
\end{proof}

\subsection{Symmetry of the Domain Implies Symmetry of the Quadrature Function}
\begin{proof}[Proof of Lemma \ref{lemma:LQDDomainSymmetryToQuadSymmetry}]\label{proof:LQDDomainSymmetryToQuadSymmetry}
We will proceed with the proof of the lemma only in the singular case; however, the non-singular argument is entirely analogous.
Suppose $0\in\Omega$ is a $k-$fold rotationally symmetric disjoint union of LQDs, $\Omega=\bigsqcup_{l}\Omega_l$, $\Omega_l\in\QD_0(h_l)$ and we set $h=\sum_lh_l$. Also suppose $\Omega_0$ is the connected component of $\Omega$ containing zero by Theorem \ref{thm:ElectrostaticLQDInterpretation}, for each $w\in\Omega\IntComp$,
\begin{align*}
    h_0(w)&=C_{\rho_0}^{\Omega_0\setminus\D_r}(w)-\dfrac{q-\ln|r|^2}{w}=\int_{\Omega_0\setminus\D_r}\dfrac{|\xi|^{-2}}{\xi-w}dA(w)-\dfrac{q-\ln|r|^2}{w},\\
    h_{l>0}(w)&=C_{\rho_0}^{\Omega_l}(w)=\int_{\Omega_l}\dfrac{|\xi|^{-2}}{\xi-w}dA(w).
\end{align*}
Summing these expressions over $l$, we obtain
\begin{align*}
    h(w)&=C_{\rho_0}^{\Omega\setminus\D_r}(w)-\dfrac{q-\ln|r|^2}{w}=\int_{\Omega\setminus\D_r}\dfrac{|\xi|^{-2}}{\xi-w}dA(w)-\dfrac{q-\ln|r|^2}{w}.
\end{align*}
Hence, applying the symmetry of $\Omega$, we find that for each $j\in\Z$ and $w\in\Omega\IntComp$
\begin{align*}
    e^{2\pi i\frac{j}{k}}h(e^{2\pi i\frac{j}{k}}w)&=e^{2\pi i\frac{j}{k}}\int_{\Omega\setminus\D_r}\dfrac{|\xi|^{-2}}{\xi-e^{2\pi i\frac{j}{k}}w}dA(w)-e^{2\pi i\frac{j}{k}}\dfrac{q-\ln|r|^2}{e^{2\pi i\frac{j}{k}}w}\\
    &=\int_{\Omega\setminus\D_r}\dfrac{|e^{2\pi i\frac{j}{k}}\xi|^{-2}}{e^{-2\pi i\frac{j}{k}}\xi-w}dA(w)-\dfrac{q-\ln|r|^2}{w}\\
    &=\int_{\Omega\setminus\D_r}\dfrac{|\xi|^{-2}}{\xi-w}dA(w)-\dfrac{q-\ln|r|^2}{w}=h(w).
\end{align*}
By the identity theorem, we conclude that $e^{2\pi i\frac{j}{k}}h(e^{2\pi i\frac{j}{k}}w)=h(w)$ for all $j\in\Z$, and $w\in\C$. Now consider the function $f(w)=w^{-(k-1)}h(w)$, so that $f(e^{2\pi i\frac{j}{k}}w)=f(w)$.

If we write $f(w)=\frac{P_0(w)}{P_1(w)}$ for coprime polynomials $P_0$ and $P_1$, then we find that for each root $z_0$ of $P_0$ or $P_1$, every element of its orbit, $\{e^{2\pi i\frac{j}{k}}z_0\}$, is also a root, and of the same multiplicity. In particular if $z_0$ is a root of multiplicity $n$ of $P_i$, then
$$P_{i}(w)=(w-z_0)^n(w-e^{2\pi i\frac{1}{k}}z_0)^n\hdots(w-e^{2\pi i\frac{k-1}{k}}z_0)^n\widetilde{P_i}(w)=(w^k-z_0^k)^n\widetilde{P_i}(w),$$
where $\widetilde{P_i}$ has no roots on the orbit of $z_0$. Iterating this procedure, we find that we can write $P_0$ and $P_1$ each as a product of factors of the form $(w^k-a^k)$. In particular, there exists a rational function $g$ such that $f(w)=g(w^k)$. Hence, $h(w)=w^{k-1}g(w^k)$ for some rational function $g$.
\end{proof}

\end{document}